\documentclass[11pt, USletter, reqno, twoside, makeidx]{amsart}
\usepackage[margin=2.8cm, marginparsep=0.2cm, marginparwidth=2.4cm]{geometry}
\usepackage[color=blue!20!white,textsize=tiny]{todonotes}
\usepackage[font={small,it}]{caption}
\usepackage{amsmath} 
\usepackage{amsthm}
\usepackage{amssymb} 
\usepackage{amsfonts}\usepackage{tikz-cd}
\usepackage{graphicx}
\usepackage{xcolor}
\definecolor{darkgreen}{rgb}{0,0.5,0}
\usepackage[normalem]{ulem}
\usepackage{comment}
\usepackage{xypic} 
\usepackage[all]{xy}
\usepackage{tikz-cd}
\usepackage[autostyle]{csquotes}
\usepackage[makeroom]{cancel}
\usepackage{braket}

\usepackage{tikz}

\usepackage{caption,subcaption,pinlabel}
\usepackage{graphics}
\usepackage{hyperref}

\usepackage{comment}
\usepackage{enumitem}
\usepackage{mathtools}
\usepackage{tikz}
\usetikzlibrary{matrix,arrows,decorations.pathmorphing}
\usepackage{mathabx}
\usetikzlibrary{matrix,calc}
\usepackage{scrextend}

\newtheorem{theorem}{Theorem}[section]
\newtheorem{lemma}[theorem]{Lemma}

\newtheorem{proposition}[theorem]{Proposition}
\newtheorem{question}[theorem]{Question}

\newtheorem{corollary}[theorem]{Corollary}

\theoremstyle{definition}
\newtheorem{remark}[theorem]{Remark}
\newtheorem{definition}[theorem]{Definition}

\newcommand{\f}{\mathbb{F}}

\newcommand{\from}{\colon}
\newcommand{\two}{\underline{2}}

\def\im{\text{im }}
\def\spinc{\text{spin}^c}

\newcommand{\Z}{\mathbb{Z}}
\newcommand{\Q}{\mathbb{Q}}
\newcommand{\R}{\mathbb{R}}
\newcommand{\C}{\mathbb{C}}

\newcommand{\gr}{\text{gr}}

\newcommand{\s}{\mathfrak{s}}

\newcommand{\T}{\mathbb{T}}

\newcommand{\Char}{\text{Char}}
\newcommand{\spa}{\text{span}}
\newcommand{\Hom}{\text{Hom}}

\def\HF {\mathit{HF}}

\newcommand\HFp {\HF^+}

\newcommand \HFm {\HF^-}

\newcommand{\Hl}{\mathbb{H}}

\newcommand{\Hla}{\mathbb{H}}
\newcommand{\Cla}{\mathbb{C}}

\newcommand{\SWF}{\mathit{SWF}}

\newcommand{\Hty}{\mathcal{H}}

\newcommand{\F}{\mathbb{F}}
\newcommand{\bbH}{\mathbb{H}}
\newcommand{\cT}{\mathcal{T}}

\def\B {\widetilde{H}_*^{S^1}}
\def\cB {c\widetilde{H}_*^{S^1}}
\def\tB {t\widetilde{H}_*^{S^1}}

\def\rX{\widetilde{X}}
\def\O{\mathcal{O}}

\usepackage{extarrows}

\newcommand{\Pin}{\operatorname{Pin}}
\newcommand{\ts}{\mathfrak{t}}
\newcommand{\hocolim}{\mathrm{hocolim}}

\newcommand{\spwt}{S^{((w(\square_d) - h)/2)\C}}
\numberwithin{equation}{section}

\newcommand{\topo}{\mathrm{Top}_*}
\newcommand{\Cat}{\mathcal{C}}
\newcommand{\Ide}{\mathrm{id}}

\author{Irving Dai}
\address{Department of Mathematics, The University of Texas at Austin}
\email{irving.dai@math.utexas.edu}
\author{Hirofumi Sasahira}
\address{Faculty of Mathematics, Kyushu University, 744, Motooka, Nishi-Ku, Fukuoka, 819-0395, Japan}
\email{hsasahira@math.kyushu-u.ac.jp}
\author{Matthew Stoffregen}
\address{Department of Mathematics, Michigan State University}
\email{stoffre1@msu.edu}

\begin{document}
\title{Lattice homology and Seiberg-Witten Floer spectra}

\begin{abstract}
Using lattice homology, we give an explicit combinatorial description of the Seiberg-Witten Floer spectrum $\SWF(Y)$ for $Y$ an almost-rational plumbed homology sphere. This class of manifolds includes all Seifert fibered rational homology spheres with base orbifold $S^2$. Using our computations, we provide a calculation of Manolescu's $\kappa$-invariant for certain connected sums of these spaces.
\end{abstract}
\vspace{-1cm}
\maketitle

\noindent
\vspace{-1cm}
\section{Introduction}

In \cite{furuta-108}, Furuta initiated the use of stable homotopy theory in the study of the Seiberg-Witten equations via his proof of the 10/8-theorem. Building on these techniques, Manolescu \cite{ManolescuSWF} defined a Seiberg-Witten Floer stable homotopy type $\SWF(Y)$ for $Y$ a rational homology sphere, partially implementing a program started by Floer \cite{floer-homotopy, CJS} of realizing Floer homologies as the homology of some naturally-occurring spectra. Indeed, work of Lidman and Manolescu \cite{LidmanMan} has shown that for rational homology spheres, $\SWF(Y)$ recovers the monopole Floer homology of Kronheimer and Mrowka \cite{KMmonopoles}. Recent research has focused on defining a Seiberg-Witten Floer stable homotopy type for 3-manifolds with $b_1 > 0$ and developing further formal properties of this theory \cite{KronheimerManolescu, KLS1, KLS2,sasahira-stoffregen, Sasahira-Stoffregen_Triangle}. The construction of a homotopy type for other invariants from low-dimensional or symplectic topology has likewise attracted a great deal of attention; see for example \cite{LipshitzSarkar, ManolescuSarkar, abouzaid-blumberg, abouzaid-mclean-smith, khovanov-lipshitz}. 


Following work of Fr\o yshov \cite{Fro96}, Manolescu used Pin(2)-equivariant Seiberg-Witten Floer homology to define a new suite of Floer-theoretic invariants with a wide array of topological applications, exemplified by his disproof of the triangulation conjecture \cite{ManolescuTriangulation}. (See also \cite{lin-pin}.) Despite this success, there are few explicit computations of Seiberg-Witten Floer spectra in the literature; or, for that matter, Floer spectra in others contexts. Although a significant amount of machinery has been established for computing and understanding Seiberg-Witten Floer \textit{homology}, progress on the homotopy side has been slow due to a lack of the typical tools present in Floer theory.

Indeed, to the best of the authors' knowledge, an (essentially) exhaustive list of all known non-trivial Seiberg-Witten Floer homotopy types is presented in \cite{ManolescuIntersection}: these consist of Brieskorn spheres of the form $\Sigma(2, 3, k)$.\footnote{The behavior of $\SWF$ under connected sums and orientation reversal is also understood.} Such examples follow from the work of Mrowka, Ozsv\'ath, and Yu \cite{MOY}, who provided an explicit description for the moduli spaces of the Seiberg-Witten equations on Seifert fibered homology spheres. In a handful of particularly simple cases, the critical points are such that the Seiberg-Witten Floer spectrum is completely determined, giving the computations in \cite{ManolescuIntersection}.

The goal of the present paper is to expand the available computations of Seiberg-Witten Floer spectra for 3-manifolds.  We work with the class of almost-rational plumbed homology spheres; this includes all Seifert fibered rational homology spheres. Our approach is based on ideas from lattice homology, a combinatorial invariant for plumbed 3-manifolds which has its roots in work of Ozsv\'ath and Szab\'o \cite{OSplumbed} and was formally introduced and developed by N\'emethi \cite{Nemethi}. Lattice homology is now known, by work of Zemke, to be isomorphic to Heegaard Floer homology \cite{Zemkelattice} and has been used by several authors to understand various other invariants such as involutive Heegaard Floer homology and instanton Floer homology \cite{DaiManolescu, ABDS}.\footnote{It should be noted that some of these results are established over different coefficient rings. For example, the full isomorphism of Zemke \cite{Zemkelattice} is currently known over $\F = \Z/2\Z$, while here we will generally work over $\Z$.}

A key input from the work of N\'emethi \cite{Nemethi} is a re-formulation of lattice homology analogous to discrete Morse theory or persistent homology. We draw inspiration from this model to construct an $S^1$-equivariant CW complex which we then prove computes the $S^1$-equivariant Seiberg-Witten Floer spectrum of any almost-rational homology sphere. Using some delicate equivariant homotopy theory, we are also able to determine the $\Pin(2)$-equivariant Seiberg-Witten Floer spectrum of these examples. We use this to analyze Manolescu's $\kappa$-invariant \cite{ManolescuIntersection} for certain connected sums. 

Due to the work of N\'emethi, there is also a close connection between lattice homology and the study of surface singularities \cite{NemethiOS, Nemethi}. We demonstrate that the work in the present paper has a curious overlap with this program by showing that in certain cases, the Seiberg-Witten Floer spectrum can be interpreted as a mapping cone between various sheaf cohomology groups associated to the resolution of a singularity. See Section~\ref{sec:ag} for a precise discussion.


\subsection{Statement of results}

Let $\Gamma$ be a negative-definite plumbing graph. We denote the associated plumbing of $2$-spheres by $W_\Gamma$ and its boundary by $Y_\Gamma$. As discussed in Section~\ref{sec:4}, each $\spinc$-structure $\s$ on $Y_\Gamma$ corresponds to an equivalence class $[k] \subset H^2(W_\Gamma; \Z)$ of characteristic elements on $W_\Gamma$. 

In general, the \textit{lattice homology} $\Hla(\Gamma, [k])$ is a combinatorially-defined $\F[U]$-module which is now known to be isomorphic to the Heegaard Floer homology of $(Y_\Gamma, \s)$. If $\Gamma$ is an \textit{almost-rational plumbing} (see \cite[Definition 8.1]{NemethiOS}), then the lattice homology is especially simple and its output can be described in terms of an object called a \textit{graded root} (see Section~\ref{sec:2.1}). We refer to the boundary of such a plumbing as an \textit{almost-rational plumbed homology sphere}, or sometimes just an AR homology sphere. In Section~\ref{sec:2.2} we describe how to construct an $S^1$-equivariant spectrum $\Hty(\Gamma, [k])$ from the data of a graded root. The main theorem of this paper will be to show that this computes the $S^1$-equivariant Seiberg-Witten Floer spectrum of $(Y_\Gamma, \s)$:

\begin{theorem} \label{thm:1.1}
Let $Y_\Gamma$ be an almost-rational plumbed homology sphere and $\s$ be a $\spinc$-structure on $Y_\Gamma$. Then we have an $S^1$-equivariant homotopy equivalence
\[
\Hty(\Gamma, [k]) = \SWF(Y_\Gamma, \s).
\]
\end{theorem}
\noindent
Since the lattice homology $\Hla(\Gamma, [k])$ is combinatorial, this provides a combinatorial calculation of $\SWF(Y_\Gamma, \s)$. Note that our construction of $\Hty(\Gamma, [k])$ factors through $\Hla(\Gamma, [k])$. Rather disappointingly, this indicates that (for such 3-manifolds) the Seiberg-Witten Floer spectrum of $(Y_\Gamma, \s)$ is determined by its Floer homology; see Remark~\ref{rem:2.1}.

In the case that $[k]$ represents a self-conjugate $\spinc$-structure, it is also possible to upgrade $\Hty(\Gamma, [k])$ to a $\Pin(2)$-spectrum.  The proof of Theorem \ref{thm:1.1} extends in a somewhat delicate way to show: 

\begin{theorem} \label{thm:1.2}
 Let $Y_\Gamma$ be an almost-rational plumbed homology sphere and $\s$ be a self-conjugate $\spinc$-structure on $Y_\Gamma$. Then we have a $\Pin(2)$-equivariant homotopy equivalence
\[
\Hty(\Gamma, [k]) = \SWF(Y_\Gamma, \s).
\]
\end{theorem}


One especially effective application of Manolescu's construction has been in constraining the topology of spin $4$-manifolds $X$ with $\partial X=Y$. While several such restrictions may be derived using Seiberg-Witten Floer homology (via the use of the invariants $\alpha,\beta,$ and $\gamma$ of \cite{ManolescuTriangulation}), if $b_2^+(X)$ is large then the most useful constraint comes from the $\kappa$-invariant of \cite{ManolescuIntersection}. Explicitly, consider the quantity
\[
\xi(Y)=\max_{ \substack{p,q\in \mathbb{Z} \\ q>1} } \ \{p-q\mid \exists X \text{ spin with } \partial X =Y \text{ and } Q(X)=p(-E_8)\oplus q(\begin{smallmatrix}1 & 0\\0 & 1\end{smallmatrix})\},
\]
where $Q(X)$ denotes the intersection form of $X$. In \cite{ManolescuIntersection}, it is shown that  
\[
\xi(Y)\leq \kappa(Y)-1
\]
for all integer homology spheres $Y$. 

Since $\kappa(Y, \s)$ is defined in terms of the $K$-theory of $\SWF(Y, \s)$ it is much harder to compute than the homological invariants $\alpha,\beta,$ and $\gamma$. The value of $\kappa$ for many Seifert spaces was calculated by Ue in \cite{ue-k-theory}, where it is was also shown that $0\leq \kappa(Y,\s) + \bar{\mu}(Y,\s)\leq 2$ for all Seifert fibered rational homology spheres. Here, $\bar{\mu}(Y, \s)$ is the Neumann-Siebenmann invariant. We sharpen and extend Ue's result to the class of all AR homology spheres:

\begin{corollary}\label{cor:a}
Let $Y$ be an AR homology sphere and $\s$ be a self-conjugate $\spinc$-structure on $Y$. Then there are two possibilities for $\kappa(Y, \s)$:
\begin{enumerate}
\item If $- \bar{\mu}(Y, \s) = \delta(Y, \s)$, then $\kappa(Y, \s)= - \bar{\mu}(Y, \s)$.
\item If $- \bar{\mu}(Y, \s) < \delta(Y, \s)$, then $\kappa(Y, \s)= - \bar{\mu}(Y, \s)+2$.  
\end{enumerate}
Here, $\delta(Y, \s)$ is the monopole Fr\o yshov invariant.
\end{corollary}

Corollary~\ref{cor:a} is similar to \cite[Corollary 1.2]{Stoffregen} and \cite[Theorem 2.4]{DaiMonopole}, in which it is shown that the $\alpha$, $\beta$, and $\gamma$-invariants for an AR homology sphere are determined by $\delta(Y, \s)$ and $\bar{\mu}(Y, \s)$. Explicitly, we have
\[
\beta(Y, \s) = \gamma(Y, \s) = - \bar{\mu}(Y, \s),
\]
while $\alpha(Y, \s)$ is either $\delta(Y, \s)$ or $\delta(Y, \s) + 1$, depending on whether the parity of $- \bar{\mu}(Y, \s)$ coincides with that of $\delta(Y, \s)$. (Our orientation conventions are such that the set of AR homology spheres encompasses the class of \textit{negative} Seifert spaces.) Hence in particular
\[
-\bar{\mu}(Y, \s) = \delta(Y, \s) \text{ or } -\bar{\mu}(Y, \s) < \delta(Y, \s).
\]
In the former case, we show that the local equivalence class of $(Y, \s)$ is trivial; this should be thought of as giving the dichotomy in Corollary~\ref{cor:a}.

We now specialize to the case when $(Y, \s)$ is \textit{projective}, as defined in \cite[Section 5.2]{Stoffregen}. The class of projective AR homology spheres includes all positive integer surgeries on (right-handed) torus knots. Our goal will be to analyze the $\kappa$-invariant for connected sums of projective AR homology spheres:


\begin{corollary}\label{cor:b}
Let $\{(Y_i, \s_i)\}_{i=1}^{n}$ be a family of projective AR homology spheres. Then
\[
\kappa(\#_{i=1}^n (Y_i, \s_i))\leq -\sum_{i = 1}^n \bar{\mu}(Y_i, \s_i) + 2\lceil n/2 \rceil.
\]
\end{corollary}

Note that understanding the behavior of $\kappa$ under connected sums is significantly more difficult than the corresponding question for $\alpha$, $\beta$, and $\gamma$ (see \cite{Stoffregenconnectedsum, Lin-Connected-Sums} for results concerning the latter). Although it is possible to compute the entire $K$-theory in the case that $(Y, \s)$ is projective, this does not immediately yield Corollary~\ref{cor:b} since equivariant $K$-theory does not have a K{\"u}nneth formula. We also establish a more refined version of Corollary~\ref{cor:b} in certain cases:

\begin{corollary}\label{cor:c}
Let $\{(Y_i, \s_i)\}_{i=1}^{n}$ be a family of projective AR homology spheres. Suppose that $n$ is even and we have $- \bar{\mu}(Y_i, \s_i) \leq \delta(Y_i, \s_i) - 2$ for each $i$. Then
\[
\kappa(\#_{i = 1}^n(Y_i, \s_i)) = - \sum_{i = 1}^n \bar{\mu}(Y_i, \s_i) + n.
\]
\end{corollary}


By applying Corollary~\ref{cor:c} to self-connected sums of a single projective AR homology sphere, we obtain:

\begin{corollary} \label{cor:d}
Let $(Y, \s)$ be a projective AR homology sphere with $- \bar{\mu}(Y, \s) \leq \delta(Y, \s) - 2$. Then
\[
\gamma(\#_{i = 1}^n(Y, \s))-\kappa(\#_{i = 1}^n(Y, \s))\to \infty
\]
as $n\to \infty$. Furthermore,
\[
\kappa(\#_{i = 1}^n(-Y, \s))-\alpha(\#_{i = 1}^n(-Y, \s))\to \infty
\]
as $n\to \infty$.
\end{corollary}

This is in contrast to the asymptotic behavior of the invariants $\alpha$, $\beta$, $\gamma$, and $\delta$. Indeed, in \cite[Theorem 1.3]{Stoffregenconnectedsum}, it is shown that
\[
\alpha(\#_{i = 1}^n (Y, \s)) - n \delta(Y, \s), \quad \beta(\#_{i = 1}^n (Y, \s)) - n \delta(Y, \s), \quad \text{and} \quad \gamma(\#_{i = 1}^n (Y, \s)) - n \delta(Y, \s)
\]
are bounded functions of $n$.

We also explore some connections between $\Hty(\Gamma, [k])$ and N\'emethi's work regarding singularity theory \cite{NemethiOS, Nemethi}. To put this in context, recall that if $(X, 0)$ is an isolated surface singularity whose link is a rational homology sphere, then taking a \textit{good resolution} $\smash{\rX}$ of $X$ gives a plumbing of $2$-spheres $W_\Gamma$ equipped with a preferred analytic structure. In \cite{Nemethi}, N\'emethi establishes an upper bound for the \textit{geometric genus} of $\smash{\rX}$ in terms of the lattice homology $\Hla(\Gamma, [K])$.\footnote{Here, $K$ is the canonical characteristic element on $W_\Gamma$.}

We show that if this upper bound is sharp, then it is possible to construct a certain mapping cone $\Hty(V_1, \ldots, V_n)$ between various sheaf cohomology groups on $\smash{\rX}$. (Our construction cannot be carried out unless the upper bound is sharp.) Moreover, $\Hty(V_1, \ldots, V_n)$ turns out to be $S^1$-homotopy equivalent to the lattice spectrum $\Hty(\Gamma, [K])$. Hence for certain plumbing graphs, the lattice spectrum $\Hty(\Gamma, [K])$ (which is a topological invariant) can be interpreted in terms of natural objects arising from certain analytic structures on $W_\Gamma$. The terms in the following theorem are defined in Section~\ref{sec:ag}:

\begin{theorem}\label{thm:1.3}
Let $\Gamma$ be an extremal (almost-rational) plumbing graph equipped with an extremal analytic structure. Let $(x_i)_{i = 0}^n$ be an extremal path of divisors. Consider the sequence of formal differences
\[
V_i = H^1(\rX, \O_{x_i})^* - H^0(\rX, \O_{x_i})^*
\]
equipped with the inclusions induced by the sheaf maps $\O_{x_{i+1}} \rightarrow \O_{x_i}$. Up to suspension, the contracted mapping cone $\Hty(V_1, \ldots, V_n)$ is $S^1$-homotopy equivalent to the lattice homotopy type $\Hty(\Gamma, [K])$. 
\end{theorem}
\noindent
Theorem~\ref{thm:1.3} is broadly reminiscent of the description of the critical sets by Mrowka, Ozsv\'ath, and Yu \cite{MOY}, who showed that if $Y$ is a Seifert fibered space, the critical points of the Chern-Simons-Dirac functional (and flows between them) may be interpreted in terms of the algebraic geometry of a particular ruled surface.

Finally, we indicate the construction of a general lattice spectrum in Section~\ref{sec:5.1}. As the current paper focuses on applications to Seiberg-Witten Floer homotopy, we leave a full development of this theory to future work.

\subsection*{Organization} In Section \ref{sec:2}, we provide a short summary of the construction and proof of Theorem~\ref{thm:1.1}.  In Section~\ref{sec:3} we review what we will need about stable homotopy categories and Seiberg-Witten Floer spectra, including the Bauer-Furuta maps and an exact triangle in Seiberg-Witten Floer spectra. In Section~\ref{sec:4} we review lattice homology. Theorems~\ref{thm:1.1} and \ref{thm:1.2} are proven in Sections~\ref{sec:5} and \ref{sec:6}. In Section~\ref{sec:7}, we analyze Manolescu's $\kappa$-invariant for certain connected sums and establish some miscellaneous theorems pertaining to the lattice spectrum. Section~\ref{sec:ag} explores the connection between our construction and the work of N\'emethi in singularity theory. Finally, Section~\ref{sec:8} provides an appendix which collects various technical results related to lattice homology.

\subsection*{Acknowledgments} The authors would like to thank Teena Gerhardt, Kristen Hendricks, Ciprian Manolescu, Ian Montague, Tomasz Mrowka, Andras N\'emethi, and Ian Zemke for helpful conversations. The first author was supported by NSF DMS-1902746 and NSF DMS-2303823. The second author was supported by JSPS KAKENHI Grant Number 23K03115. The third author was supported by NSF DMS-2203828. 

\section{Overview and Sketch of Proof}\label{sec:2}

In this section, we describe how to compute the ($S^1$-)Seiberg-Witten Floer spectra of AR plumbed homology spheres. In order to communicate the main ideas of the paper, we also give a brief outline of the proof. Details may be found in Section~ \ref{sec:5}.


\subsection{Preliminaries}\label{sec:2.1}
Our proposed algorithm relies on understanding the output of lattice homology, as defined in \cite{OSplumbed, Nemethi}. In particular, central to our discussion will be the notion of a graded root. We thus review this here.

Throughout, let $Y$ be a plumbed homology sphere with plumbing graph $\Gamma$. Denote by $W_\Gamma$ the plumbed $4$-manifold corresponding to $\Gamma$ and let $\s$ be a $\spinc$-structure on $Y$. As discussed in Section~\ref{sec:4}, we may identify $\s$ with an equivalence class $[k] \subset H^2(W_\Gamma; \Z) = \Hom(H_2(W_\Gamma; \Z), \Z)$ of characteristic elements on $W_\Gamma$. 

The \textit{(zeroth) lattice homology} $\Hl_0(\Gamma, [k])$ is a combinatorially-defined $\Z[U]$-module constructed from $(\Gamma, [k])$. This was first introduced by Ozsv\'ath and Szab\'o in \cite{OSplumbed} and generalized to the \textit{higher lattice homology groups} $\Hl_i(\Gamma, [k])$ by N\'emethi \cite{Nemethi}. In \cite{OSplumbed, Nemethi}, it is shown that the $\Hl_i(\Gamma, [k])$ are invariants of $(Y, \s)$, rather than the plumbing $\Gamma$. Moreover, if $Y$ is an AR plumbed homology sphere, then $\Hl_0(\Gamma, [k])$ is isomorphic to $\HFm(Y, \s)$ and all higher lattice groups vanish; see \cite[Theorem 1.2]{OSplumbed}, \cite[Theorem 8.3]{NemethiOS}, and \cite[Theorem 5.2.2]{Nemethi}.\footnote{For concreteness, we write $\HFm$ here, as this is the original setting of \cite{OSplumbed} and \cite{Nemethi}. However, the appropriate monopole Floer or Seiberg-Witten Floer homology groups can also be substituted.} See Section~\ref{sec:4.1} for the definition of lattice homology.

Importantly, $\Hl_0(\Gamma, [k])$ comes with a preferred basis over $\Z$, which is again independent of $\Gamma$. We record this preferred basis in the form of a \textit{graded root} $R$. This is an upwards-opening tree with an infinite downwards stem, where each vertex of $R$ is given a grading (represented by its height) lying in a coset of $2\Z$ in $\Q$. The vertices of $R$ correspond to the preferred generators of $\Hl_0(\Gamma, [k])$ over $\Z$; following edges downwards records multiplication by $U$. Note that the fact that $\Hl_0(\Gamma, [k])$ is represented by $R$ places certain restrictions on $\Hl_0(\Gamma, [k])$: for instance, all generators of $\Hl_0(\Gamma, [k])$ differ in grading by an element of $2\Z$. An example of a graded root is given on the left of Figure~\ref{fig:2.1}.

\begin{figure}[h!]
\includegraphics[scale = 0.8]{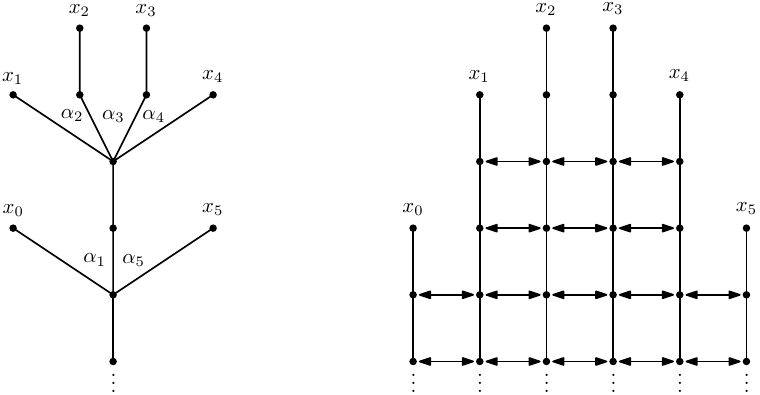}
\caption{Left: a graded root. Right: a presentation of the corresponding $\Z[U]$-module.}\label{fig:2.1}
\end{figure}

The structure of $R$ gives a preferred presentation for $\Hl_0(\Gamma, [k])$ as a $\Z[U]$-module, as follows. Label the leaves of $R$ from left-to-right by $x_0, x_1, \ldots, x_n$ and label the upwards-opening angles of $R$ from left-to-right by $\alpha_0, \alpha_1, \ldots, \alpha_{n-1}$. Let $\gr(x_i)$ be the Maslov grading of $x_i$ and let $\gr(\alpha_i)$ be the Maslov grading of the vertex at which $\alpha_i$ is based. Then $\Hl_0(\Gamma, [k])$ is isomorphic to
\begin{equation}\label{eq:2.1}
\spa_{\Z[U]}\{x_0, x_1,\ldots, x_n\}/\sim
\end{equation}
where $\sim$ is generated over $\Z[U]$ by the relations
\[
U^{(\gr(x_i)-\gr(\alpha_i))/2} x_i \sim U^{(\gr(x_{i+1})-\gr(\alpha_i))/2} x_{i+1}
\]
for $0 \leq i \leq n - 1$. Visually, this is formed as follows: for each leaf $x_i$, we draw a $\Z[U]$-tower starting in grading $\gr(x_i)$. We then glue each successive pair of towers $\Z[U] \cdot x_i$ and $\Z[U] \cdot x_{i+1}$ along a subtower starting in grading $\gr(\alpha_i)$, as in the right of Figure~\ref{fig:2.1}. We will use $R$ interchangeably to refer to: (a) the graded root (as a combinatorial object); (b) the corresponding $\Z[U]$-module; and even (c) the presentation (\ref{eq:2.1}) of this module.

\subsection{The algorithm.}\label{sec:2.2}
Let $S^0 = \{\mathrm{pt}\}^+$ (the point with a disjoint basepoint) and $S^2 = \C^+$ be given the usual $S^1$-actions. Denote the $S^1$-equivariant inclusion of $S^0$ into $S^2$ by $i \colon S^0 \hookrightarrow S^2$. Suspending this map gives an inclusion 
\[
\Sigma^{n\C} i \colon S^{2n} \hookrightarrow S^{2n + 2}
\]
for any $n$, and thus an increasing sequence of $S^1$-equivariant inclusions
\[
S^0 \hookrightarrow S^2 \hookrightarrow S^4 \hookrightarrow \cdots,
\]
where $S^{2n} = (n\mathbb{C})^+$. Abusing notation, we denote each such inclusion (or composition of inclusions) uniformly by $i \colon S^{2m} \hookrightarrow S^{2n}$ for $m \leq n$.  

Given a graded root $R$, we now construct an $S^1$-equivariant space $X$ in analogy with the presentation (\ref{eq:2.1}). First, cut $R$ at some vertex of grading $h$ on the infinite stem, such that any vertices of valency greater than two lie at or above the cut. For each leaf $x_i$ of $R$, introduce an $S^1$-sphere of dimension $\gr(x_i) - h$. Define $X$ by starting with the disjoint union of these spheres and quotienting by the equivalence relation $\sim$:
\begin{equation}\label{eq:2.2}
X = \left(\bigsqcup_{i=0}^n S^{(\gr(x_i) - h)}\right)/\sim
\end{equation}
where $\sim$ is defined as follows. For each angle $\alpha_i$, glue together the pair of successive spheres $S^{(\gr(x_i) - h)}$ and $S^{(\gr(x_{i+1}) - h)}$ along the images of the inclusions
\[
i \colon S^{(\gr(\alpha_i) - h)} \hookrightarrow S^{(\gr(x_i) - h)} \quad \text{and} \quad i \colon S^{(\gr(\alpha_i) - h)} \hookrightarrow S^{(\gr(x_{i+1}) - h)}.
\]
See the right of Figure~\ref{fig:2.2} for an example. The parallel between (\ref{eq:2.1}) and (\ref{eq:2.2}) is clear: the spheres in (\ref{eq:2.2}) play the roles of the $\Z[U]$-towers in (\ref{eq:2.1}), while the identification $\sim$ of (\ref{eq:2.2}) is analogous to the subtower-gluing relation of (\ref{eq:2.1}). We will explain this further presently.

\begin{figure}[h!]
\includegraphics[scale = 1]{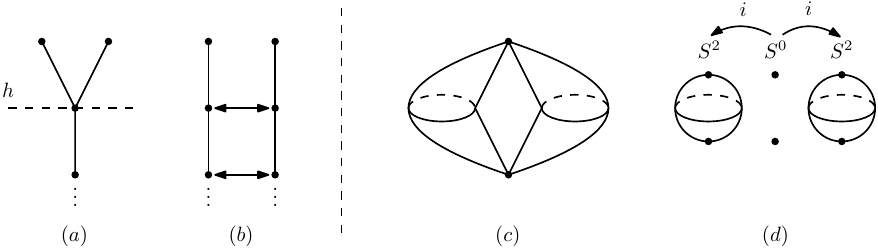}
\caption{Illustration of the example $Y = \Sigma(2, 3, 7)$. Left: $(a)$ the graded root corresponding to $\HFm(\Sigma(2,3,7))$ (together with a choice of $h$) and $(b)$ its presentation from (\ref{eq:2.1}). Right: (c) the well-known spectrum $\SWF(\Sigma(2, 3, 7))$, which is evidently formed by gluing the components of (d) together according to (\ref{eq:2.2}).}\label{fig:2.2}
\end{figure}

Decreasing the cutoff grading $h$ corresponds to increasing the dimensions of all of the spheres in the construction of $X$, which is easily seen to be the same as suspending $X$ by some $\mathbb{C}^n$. Hence $X$ is well-defined up to suspension. We then define the $(S^1$-)lattice homotopy type associated to $(\Gamma, [k])$ to be
\[
\mathcal{H}_0(\Gamma, [k]) = (X, 0, -h/2);
\]
see Section \ref{sec:stable homotopy} for the meaning of this notation. We often write this as the formal suspension $\Sigma^{(h/2)\C}X$, although $h/2$ here may be rational. Note that the $S^1$-fixed point set of $X$ is $S^0$.

The main theorem of this paper will be to show that $\mathcal{H}_0(\Gamma, [k])$ is indeed homotopy equivalent to $\SWF(Y, \s)$. In the case that $\s$ is self-conjugate, it is possible to put a Pin(2)-action on $X$ which calculates the Pin(2)-equivariant Seiberg-Witten Floer spectrum of $Y$. We discuss this in Section~\ref{sec:6.2}.

\begin{remark}\label{rem:2.1}
\textit{A priori}, our construction of $X$ depends on more than the combinatorial structure of $R$: we have also chosen a particular planar embedding of $R$ in order to label the leaves of $R$ from left-to-right. However, it is not difficult to show that choosing a different planar embedding does not change the homeomorphism class of $X$. In Lemma~\ref{lem:graded-root-switch}, we will show that $X$ depends only on the isomorphism class of $R$ as a $\Z[U]$-module, rather than its graded root structure. However, the proof of Theorem~\ref{thm:1.1} requires an understanding of the output of lattice homology as a graded root, which is why we emphasize this point-of-view.
\end{remark}

\subsection{Outline of proof}\label{sec:2.3}
For the reader who is broadly familiar with the basics of both lattice homology and Seiberg-Witten Floer homotopy, we now outline the proof of Theorem~\ref{thm:1.1}. 
\noindent
\\
\\
\textit{Computing the co-Borel homology.} The first step of the proof is to show that the $S^1$-equivariant homology of $X$ is in fact isomorphic (as an $\Z[U]$-module) to $R$. This certainly needs to be the case if Theorem~\ref{thm:1.1} is to hold, since $R$ is isomorphic to the monopole Floer (and thus the Seiberg-Witten Floer) homology of $(Y, \s)$. More precisely, we wish to show that (up to an overall grading shift)
\[
\cB(X) \cong R,
\]
where $\cB(X)$ is the \textit{co-Borel homology} of $X$. The definition of co-Borel homology may be found in \cite[Section 2.2]{ManolescuTriangulation}, but this is almost unimportant: all we will need is that the co-Borel homology of $S^{2n}$ is a single $\Z[U]$-tower which starts in Maslov grading $2n$: 
\[
\cB(S^{2n}) \cong \Z[U]_{2n}.
\]
Note here that $\deg(U) = -2$. Moreover, it is straightforward to check that the map $i_*$ on co-Borel homology induced by the inclusion $i \colon S^{2m} \hookrightarrow S^{2n}$
\[
i_* \colon \cB(S^{2m}) \cong \Z[U]_{2m} \rightarrow \cB(S^{2n}) \cong \Z[U]_{2n},
\]
simply maps the first tower to the subtower of the second starting in grading $2m$.

It is then not difficult to show that (roughly speaking) taking the co-Borel homology converts (\ref{eq:2.2}) into the presentation (\ref{eq:2.1}). Each sphere in the construction of $X$ contributes a single $\Z[U]$-tower to the co-Borel homology. The identification $\sim$ of (\ref{eq:2.2}) (combined with the above observation regarding $i_*$) shows that successive $\Z[U]$-towers are glued together along subtowers, precisely as (\ref{eq:2.1}). See Lemma~\ref{lem:5.5} for further discussion.
\noindent
\\
\\
\textit{Constructing a map of spectra.} Having verified that $X$ has the correct $S^1$-equivariant homology, the second step of the proof is to construct (after appropriate suspension) an actual $S^1$-equivariant map from $X$ to $\SWF(Y, \s)$:
\[
\cT \colon \mathcal{H}_0(\Gamma, [k]) \rightarrow \SWF(Y, \s).
\] 
For this, it will be helpful to draw on some intuition coming from lattice homology. We thus sketch how $\Hla_0(\Gamma, [k])$ is actually defined. In order to construct $\Hla_0(\Gamma, [k])$, we find a sequence of characteristic elements $x_1, \ldots, x_n$ in $[k]$ satisfying certain combinatorial conditions (see Theorem~\ref{thm:4.9}). Each of these is given the Maslov grading 
\[
\gr(x_i) = \dfrac{1}{4}\left(x_i^2 + |\Gamma|\right). 
\]
The lattice homology $\Hla_0(\Gamma, [k])$ is constructed by quotienting the $\Z[U]$-span of $\{x_0, x_1, \ldots, x_n\}$ by a certain equivalence relation $\sim$ defined in Section~\ref{sec:4.1}:
\[
\Hla_0(\Gamma, [k]) = \spa_{\Z[U]}\{x_0, x_1, \ldots, x_n\}/\sim.
\]
This presentation of $\Hla_0(\Gamma, [k])$ is precisely the graded root $R$ discussed in Section~\ref{sec:2.1}. The point we are making here is that the leaves of $R$ correspond to certain characteristic elements in $[k]$, and thus to certain $\spinc$-structures on $W_\Gamma$.

This interpretation of $\Hla_0(\Gamma, [k])$ gives rise to the \textit{lattice homology isomorphism map}
\[
\T \colon R \cong \Hla_0(\Gamma, [k]) \rightarrow \HFm(Y, \s)
\] 
defined in \cite{OSplumbed}. Let
\[
F_{W_\Gamma, x_i} \colon \HFm(S^3) \rightarrow \HFm(Y, \s)
\]
be the cobordism map associated to the (punctured) cobordism $W_\Gamma$ from $S^3$ to $(Y, \s)$, equipped with the $\spinc$-structure corresponding to $x_i$. (Note that the $x_i$ restrict to $\s$ on $Y$.) Set
\[
\T(x_i) = F_{W_\Gamma, x_i}(1),
\]
for some fixed generator $1 \in \HFm(S^3)$. Extending this map $\Z[U]$-linearly gives a map from $\spa_{\Z[U]}\{x_0, x_1, \ldots, x_n\}$ to $\HFm(Y, \s)$. In \cite{OSplumbed}, it is shown that $\T$ descends to the quotient by $\sim$ due to the adjunction relations of \cite[Theorem 3.1]{OS:four}. These relate the images of the cobordism maps corresponding to different $\spinc$-structures on $W_\Gamma$. In \cite[Theorem 8.3]{NemethiOS} (see also \cite[Theorem 1.2]{OSplumbed} and \cite[Theorem 5.2.2]{Nemethi}), it is shown that if $Y$ is an AR plumbed homology sphere, then $\T$ is an isomorphism. See Section~\ref{sec:4.2} for further discussion.

Roughly speaking, our map $\cT$ is defined by mimicking the construction of $\T$ in the setting of Floer homotopy. For each $x_i$, consider the relative Bauer-Furuta invariant associated to $(W_\Gamma, x_i)$. This gives a homotopy class of map
\[
\Psi_{W_\Gamma, x_i} \colon \Sigma^{(\gr(x_i)/2)\C} \SWF(S^3) \rightarrow \SWF(Y, \s);
\]
see Section~\ref{sec:3.4}. Now recall that in the suspension $\Sigma^{(h/2)\C} X$, the leaf $x_i$ corresponds to a subsphere 
\[
\Sigma^{(h/2)\C}(S^{(\gr(x_i) - h)}) = \Sigma^{(h/2)\C} \Sigma^{(\gr(x_i)/2-h/2)\C} S^0 = \Sigma^{(\gr(x_i)/2)\C} S^0.
\]
We thus define $\cT$ on $\Sigma^{(h/2)\C}(S^{(\gr(x_i) - h)}) \subset \mathcal{H}_0(\Gamma, [k])$ to be $\Psi_{W_\Gamma, x_i}$. Taking the disjoint union of these Bauer-Furuta maps defines $\cT$ on the disjoint union of spheres in (\ref{eq:2.2}). In Section~\ref{sec:3.5}, we establish an adjunction relation for Seiberg-Witten Floer spectra which (roughly speaking) allows us conclude that the Bauer-Furuta maps associated to two consecutive $x_i$ must coincide (up to homotopy) along a subsphere of a particular dimension. We use this to show that $\cT$ descends to a map on the quotient space $\mathcal{H}_0(\Gamma, [k])$. See Lemma~\ref{lem:5.6} for further discussion. 
\noindent
\\
\\
\textit{The equivariant Whitehead theorem.} As we show in Lemma~\ref{lem:5.6}, putting the above ingredients together gives a commutative diagram
\[
\begin{tikzpicture}[scale=1]
\node (A) at (0,0) {$\cB(\mathcal{H}_0(\Gamma, [k]))$};
\node at (0,-0.75) {\rotatebox{270}{\scalebox{1}[1]{$\cong$}}};
\node (B) at (0, -1.5) {$R$};
\node at (0,-2.25) {\rotatebox{270}{\scalebox{1}[1]{$\cong$}}};
\node (C) at (0,-3) {$\Hla_0(\Gamma, [k])$};
\node (D) at (5,-1.5) {$\cB(\SWF(Y,\s))$};
\path[->,font=\scriptsize,>=angle 90]
(A) edge node[above]{$\cT_*$} (D)
(C) edge node[below]{$\T$} (D);
\end{tikzpicture}.
\]
Here, $\cT_*$ is the map on co-Borel homology induced by $\cT \colon \mathcal{H}_0(\Gamma, [k]) \rightarrow \SWF(Y, \s)$. The map $\T$ out of the combinatorial object $\Hla_0(\Gamma, [k])$ is defined as in the previous section, except with $\HFm(Y,\s)$ replaced by $\smash{\cB(\SWF(Y, \s))}$; we note that \emph{a priori}, even though $\cT_*$ is canonical, the map $\cT$, up to homotopy, may depend on choices. That is, on each generator $x_i$ of $\Hla_0(\Gamma, [k])$, we set $\T(x_i) = F_{W_\Gamma, x_i}(1)$, where $F_{W_\Gamma, x_i}$ is the map on Seiberg-Witten Floer \textit{homology} corresponding to the $\spinc$-structure $x_i$ on $W_\Gamma$.

There are two components of the final step. First, we show that $\T$ is an isomorphism. Note that $\Hla_0(Y, [k])$ is isomorphic to $\HFm(Y, \s)$ by \cite[Theorem 8.3]{NemethiOS}, so it is certainly the case that $\Hla_0(Y, [k])$ and $\smash{\cB(\SWF(Y,\s))}$ are isomorphic, since the latter is isomorphic to monopole Floer homology \cite{LidmanMan} and thus Heegaard Floer homology \cite{KLT, CGH}. However, this isomorphism is not known to be natural with respect to cobordism maps. Hence this does not necessarily show that $\T$ (defined in the context of Seiberg-Witten Floer homology) is itself an isomorphism. 

Instead, we simply replicate the original argument of \cite{OSplumbed, NemethiOS}, replacing Heegaard Floer homology with Seiberg-Witten Floer homology. The argument of \cite{OSplumbed, NemethiOS} depends on certain formal properties of the Heegaard Floer package, the most prominent of which is the existence of a surgery exact sequence. In the setting of Seiberg-Witten Floer homology, this has recently been established by the second and third authors \cite{Sasahira-Stoffregen_Triangle}. For readers unfamiliar with \cite{OSplumbed, NemethiOS}, we give an overview of the argument in Section~\ref{sec:8}.

The above commutative diagram then shows that the induced map on co-Borel homology $\cT_*$ is an isomorphism. The final component of the proof is to appeal to an equivariant Whitehead theorem. Indeed, in Lemma~\ref{lem:5.7} we show that in this situation, the fact that $\cT$ induces an isomorphism on co-Borel homology suffices, along with limited other information from the geometric situation, to establish that $\cT$ is an $S^1$-equivariant homotopy equivalence. 


\section{Review of Seiberg-Witten Floer Spectrum}\label{sec:3}

In this section, we will review the Seiberg-Witten Floer spectrum and relative Bauer-Furuta invariant.   Also we will recall some basic facts from homotopy theory.

\subsection{Stable homotopy categories} \label{sec:stable homotopy} \label{sec:3.1}

 We will  define an $S^1$-equivariant stable homotopy  category $\mathfrak{C}_{S^1}$  to define $S^1$-equivalent Seiberg-Witten Floer spectrum. The category $\mathfrak{C}_{S^1}$ was introduced in \cite[Section 6]{ManolescuSWF} and is an $S^1$-equivariant analogue of the Spanier-Whitehead category.
  
An  object of $\mathfrak{C}_{S^1}$ is  a triple $(W, m, n)$, where $W$ is a pointed $S^1$-CW complex,  $m \in \Z$ and $n \in \Q$.  For objects $(W_0, m_0, n_0)$ and $(W_1, m_1, n_1)$ of $\mathfrak{C}_{S^1}$, the set of morphisms is defined by
\[
     \operatorname{Mor}_{\mathfrak{C}_{S^1}}((W_0, m_0, n_0), (W_1,m_1,n_1)) =
    \lim_{p, q \rightarrow \infty} [ \Sigma^{p\R\oplus q\C} W_0,  \Sigma^{(p+m_0 -m_1) \R \oplus (q+n_0-n_1)\C} W_1]_{S^1}^0
\]
if $n_1 - n_0 \in \Z$, and we define it to be the empty set otherwise, where $[X,Y]_{S^1}$ is the set of $S^1$-equivariant (unstable) homotopy classes  $X\to Y$.   For a finite dimensional complex vector space $F$ with an $S^1$-action induced by the complex structure, we define the suspension $\Sigma^{F} (W, m, n)$ and desuspension $\Sigma^{-F} (W, m, n)$  by
\[
          \Sigma^{F} (W, m, n) := (\Sigma^{F} W, m, n),  \      \Sigma^{-F} (W, m, n) := (W, m, n + \dim_{\C} F). 
\]
For a finite dimensional real vector space $F$ with the trivial $S^1$-action, we define
\[ 
         \Sigma^{F} (W, m, n) := (\Sigma^{F} W, m, n), \   \Sigma^{-F}(W, m, n) := (\Sigma^{F} W, m + 2 \dim_{\R} F, n). 
\] 
The reason why we define $\Sigma^{-F}(W, m, n)$ to be  $(\Sigma^{F}W, m +  2\dim_{\R} F, n)$ rather than $(W, m + \dim_{\R} F, n)$ is the following.   Note that the homotopy class of a trivialization $f : F \rightarrow \R^{p}$ is not unique since $\pi_{0}(O(p)) = \Z_2$.    Hence we do not have a canonical identification  between $(W, m, n)$ and $\Sigma^{F}(W, m+\dim_{\R} F, n)$.  However, the homotopy class of the trivialization $f \oplus f : F \oplus F \rightarrow \R^{2p}$ is independent of the choice of $f$ and we have a canonical identification
\[
      \Sigma^{F}(\Sigma^{F} W, m + 2\dim_{\R} F, n) = (\Sigma^{F \oplus F} W, m+ 2\dim_{\R} F, n)
      \cong (\Sigma^{2p\R} W, m+2p, n) \cong (W, m, n)
\]
up to homotopy. 

For $q \in \Q$,  we denote  $(S^0, 0, q)$ by  $S^{-q\C}$.

We will introduce a $\Pin(2)$-equivariant stable homotopy category $\mathfrak{C}_{\Pin(2)}$ to define $\Pin(2)$-equivariant Seiberg-Witten Floer spectrum. Let $\tilde{\R}$ be the non-trivial, one-dimensional  real representation of $\Pin(2)$.   An object of $\mathfrak{C}_{\Pin(2)}$ is a triple $(W, m, n)$, where $W$ is a pointed $\Pin(2)$-CW complex, $m \in \Z$ and $n \in \Q$.   We define the set of morphisms by
 \[
 \begin{split}
 &\operatorname{Mor}_{\mathfrak{C}_{\Pin(2)}}((W_0, m_0, n_0), (W_1, m_1, n_1))  \\
  & \ := \lim_{p, q, r \rightarrow \infty} 
 [ \Sigma^{ p\R \oplus q\tilde{\R}  \oplus r\mathbb{H}} W_0,    \Sigma^{p\R \oplus (q+m_0-m_1) \tilde{\R} \oplus (r + n_0 - n_1)\mathbb{H}} W_1]_{\Pin(2)}^0, 
 \end{split}
 \] 
if $n_0 - n_1 \in \Z$ and we define it to be the empty set otherwise.  For a $\Pin(2)$-representation $F$ which is isomorphic to $\tilde{\R}^{a} \oplus \mathbb{H}^{b}$, we define
\[
      \Sigma^{F} (W, m, n) := (  \Sigma^{F}W, m, n), \    \Sigma^{-F} (W, m, n) := (\Sigma^{F^{S^1}} W, m + 2a, n+b). 
\]
Here $F^{S^1}$ is the $S^1$-fixed point set which is isomorphic to $\tilde{\R}^{a}$. 
For $q \in \Q$, we denote $(S^0, 0, q)$ by $S^{-q\mathbb{H}}$.


\subsection{Seiberg-Witten Floer spectrum}\label{sec:3.2}

We will explain how to define the Seiberg-Witten Floer spectrum $\SWF(Y, \s)$ briefly.  See \cite{ManolescuSWF} for the details. 

Let $(Y, \s)$ be a $\spinc$ rational homology 3-sphere and take a Riemann metric $g$ on $Y$. Write $S$ for the spinor bundle and $A_0$ for a flat spin$^c$-connection on $Y$.    We have the Clifford multiplication
\[
    \rho : T^* Y \otimes \C   \xrightarrow{\cong} \mathfrak{sl}(S). 
\]
The Chern-Simons-Dirac function $CSD:i\Omega^1(Y) \oplus \Gamma(S) \rightarrow \R$ is defined by
\[
           CSD(a, \phi) = -\frac{1}{2} \int_{Y} a \wedge da +  \frac{1}{2} \int_{Y} \left< D_{a} \phi, \phi \right> d\mu_{g}. 
\]
Here $D_{a} : \Gamma(S) \rightarrow \Gamma(S)$ is the Dirac operator associated to the connection $A_0 + a$.  The gradient vector field of $CSD$ with respect to the $L^2$-inner product is given by
\[
    \operatorname{grad}_{L^2} CSD  (a, \phi) = ( *da + q(\phi), D_a \phi  ), 
\]
where
\[
    q(\phi) = \rho^{-1} \left( \phi \otimes \phi^* - \frac{1}{2} \operatorname{Tr}(\phi \otimes \phi^*)  id  \right).
\]
 The gradient trajectory equation
 \begin{equation} \label{SW eq tilde gamma}
   \begin{split}
  &   \tilde{\gamma}   : \R \rightarrow i \Omega^1(Y) \oplus \Gamma(S),   \\
  &   \frac{\partial \tilde{\gamma}}{\partial t}(t) = - \operatorname{grad}_{L^2} CSD ( \tilde{\gamma}(t))
  \end{split}
\end{equation}
of $CSD$ is the Seiberg-Witten equations on $Y \times \R$. 

To define $\SWF(Y,\s)$, we need to take the Coulomb gauge slice:
\[
      V := i \ker (d^* : \Omega^1(Y) \rightarrow \Omega^0(Y)) \oplus \Gamma(S). 
\]
Let $\Pi_{C}$ be the map $i\Omega^1(Y) \oplus \Gamma(S) \rightarrow V$ defined by
\[
     \Pi_{C}(a, \phi) := ( a - d\xi(a), e^{\xi(a)} \phi ), 
\]
where is $\xi(a)$ is the unique function $Y \rightarrow i\R$  which satisfies 
\[
      \Delta \xi (a) = d^* a,   \quad \int_{Y} \xi(a)  d\mu_g = 0. 
\]
The pushforward $\Pi_{C * (a,\phi)} : i\Omega^1(Y) \oplus \Gamma(S) \rightarrow V$ of $\Pi_{C}$ at $(a, \phi)$ is given by the formula
\[
      \Pi_{C * (a, \phi)} (b, \psi) = (b - d \xi(b), \psi + \xi(b) \phi). 
\]
We consider the following equation for $\gamma : \R \rightarrow V$:
\begin{equation}  \label{SW eq gamma}
     \frac{\partial \gamma}{\partial t}(t) = - \Pi_{C *}  \operatorname{grad}_{L^2} CSD (\gamma(t)). 
\end{equation}
We can write
\[
       \Pi_{C *} \operatorname{grad}_{L^2} CSD(a, \phi) = l(a, \phi) + c(a, \phi), 
\]
where $l$ is the self-adjoint operator on $V$ defined by 
\[
           l = *d \oplus D_{A_0}, 
\]
and the non-linear term $c(a, \phi)$ is given by
\[  
        c(a, \phi) = \left( q(\phi) - d\xi(q(\phi)),  \rho(a) \phi  + \xi( q(\phi)  ) \phi \right).   
\]  
 We can see that if $\tilde{\gamma}$ is a solution to (\ref{SW eq tilde gamma}), then $\gamma := \Pi_{C} \tilde{\gamma}$ is a solution to (\ref{SW eq gamma}).  Conversely, if $\gamma$ is a solution to (\ref{SW eq gamma}),  a lift $\tilde{\gamma}$ of $\gamma$ is a solution to (\ref{SW eq tilde gamma}).  

The vector field $\Pi_{C *} \operatorname{grad}_{L^2} CSD$ on $V$ is the gradient vector field of $CSD|_{V}$ with respect to a metric on $V$.  See \cite[Section 2]{KLS1}. Hence if $\gamma : \R \rightarrow  V$ is a solution to (\ref{SW eq gamma}), then $CSD(\gamma(t))$ decreases in $t$.   We define the energy $E(\gamma)$ of $\gamma$ by 
\[
   E(\gamma) = 
     \lim_{t \rightarrow -\infty} CSD(\gamma(t)) -  \lim_{t \rightarrow \infty} CSD(\gamma(t))
     \in \R_{\geq 0} \cup \{ \infty \}. 
\]
We have the following compactness result for finite energy solutions to (\ref{SW eq gamma}).

\begin{proposition}  \label{prop:compactness gamma} 
Let $k$ be  a non-negative number. Then there is a positive constant $R = R_{k}$ such that for any solution $\gamma : \R \rightarrow V$ with $E(\gamma) < \infty$ and $t \in \R$, we have
\[
      \| \gamma(t) \|_{L^2_{k}(Y)} < R. 
\]
\end{proposition}

We will take finite dimensional approximations of (\ref{SW eq gamma}) as follows.  Fix $k \geq 2$ and let $R = R_k$ be the positive constant  of Proposition \ref{prop:compactness gamma}.  For $\lambda, \mu \in \R$ with $\lambda < \mu$, let $V_{\lambda}^{\mu}$ denote the subspace of $V$ spanned by eigenvectors of $l$ whose eigenvalues are in $(\lambda, \mu]$.  By the standard theory of elliptic differential operators, the dimension of $V_{\lambda}^{\mu}$ is finite.  Let $p_{\lambda}^{\mu}$ be the $L^2$-projection onto $V_{\lambda}^{\mu}$ and $\chi : V \rightarrow [0,1]$ be a smooth cut off function such that
\[
     \chi|_{B(V_{\lambda}^{\mu},2R)} \equiv 1, \ 
     \operatorname{supp}(\chi)  \subset B(V_{\lambda}^{\mu}, 3R), 
\] 
where for $r > 0$ 
\[
       B(V_{\lambda}^{\mu}, r) = \{  (a, \phi) \in V_{\lambda}^{\mu} | \| (a, \phi )\|_{L^2_k(Y)} \leq r  \}. 
\]
We consider the following equation 
\[
       \frac{\partial \gamma}{\partial t}(t) = - \chi (\gamma(t)) \{ l (\gamma(t)) + p_{\lambda}^{\mu} c(\gamma(t)   \}
\]
for $\gamma : \R \rightarrow V_{\lambda}^{\mu}$. 

This equation generates  an approximated Seiberg-Witten flow
\[
     \varphi_{\lambda}^{\mu} : V_{\lambda}^{\mu} \times \R \rightarrow V_{\lambda}^{\mu}. 
\]
This flow is $S^1$-equivariant, where the $S^1$-action on $V_{\lambda}^{\mu}$ is defined by complex multiplication on the spinor component.  Using Proposition \ref{prop:compactness gamma}, we can prove the following:

\begin{proposition}[\cite{ManolescuSWF}]
For $\lambda \ll 0$ and $\mu \gg 0$,  the ball $B(V_{\lambda}^{\mu}, R)$ is an isolating neighborhood with respect to the flow $\varphi_{\lambda}^{\mu}$.  That is,  the maximal invariant set $\operatorname{Inv}( B(V_{\lambda}^{\mu}, R), \varphi_{\lambda}^{\mu}  )$  in $B(V_{\lambda}^{\mu}, R)$ is included in the interior $\operatorname{Int}(B(V_{\lambda}^{\mu}, R))$. 
\end{proposition}

This proposition implies that we can take an $S^1$-equivaraint index pair $(N_{\lambda}^{\mu}, L_{\lambda}^{\mu})$ of the  invariant set $\operatorname{Inv}( B(V_{\lambda}^{\mu}, R), \varphi_{\lambda}^{\mu}  )$ \cite{Conley, FloerConley}.

Let $n(Y, g, \s)$ be the rational number defined by
\begin{equation}  \label{eq: n(Y,g s)}
     n(Y, g, \s) := \operatorname{ind} D_{\hat{A}_0} - \frac{c_1(\ts)^2 -\sigma(X)}{8}.
\end{equation}
Here $(X, \hat{g},  \ts)$  is a compact,  $\spinc$ Riemannian  $4$-manifold with boundary $(Y, g,  \s)$,  $\hat{A}_0$ is a $\spinc$ connection on $X$ with $\hat{A}_0|_{Y} = A_0$, $D_{\hat{A}_0}$ is the Dirac operator on $X$ associated to $\hat{A}_0$ and $\operatorname{ind} D_{\hat{A}_0}$ is the Atiyah-Patodi-Singer index \cite{APS1} of $D_{\hat{A}_0}$.  We can show that $n(Y, g, \s)$  depends only on $Y, g$ and $\s$.

\begin{definition}
We define the $S^1$-equivariant  Seiberg-Witten Floer spectrum $\SWF(Y, \s)$ to be the object $\Sigma^{-V_{\lambda}^{0}}(N_{\lambda}^{\mu} / L_{\lambda}^{\mu}, 0, n(Y, g, \s))$ of the category $\mathfrak{C}_{S^1}$ for fixed real numbers $\lambda, \mu$ with $\lambda \ll 0$, $\mu \gg 0$. 
\end{definition}

If $\s$ is a spin structure on $Y$,  the spinor bundle $S$ is a quaternionic vector bundle and we have a $\Pin(2)$-action on $V$, where the action of $j$ is given by
\[
     j(a, \phi) = (-a, \phi j). 
\]
The equation (\ref{SW eq gamma}) is $\Pin(2)$-equivariant,  and we can choose a $\Pin(2)$-equivariant index pair $(N_{\lambda}^{\mu}, L_{\lambda}^{\mu})$ for $\operatorname{Inv} ( B(V_{\lambda}^{\mu}, R), \varphi_{\lambda}^{\mu}  )$. 

\begin{definition}
For a closed, spin $3$-manifold $(Y, \s)$,  we define the $\Pin(2)$-equivariant Seiberg-Witten Floer spectrum $\SWF(Y, \s)$ to be the object $\Sigma^{-V_{\lambda}^{0}}(N_{\lambda}^{\mu} / L_{\lambda}^{\mu}, 0, \frac{1}{2}n(Y, g, \s))$ of $\mathfrak{C}_{\Pin(2)}$ for fixed real numbers $\lambda, \mu$ with $\lambda \ll 0, \mu \gg 0$. 
\end{definition}

We can prove that $\SWF(Y, \s)$ is independent of the choices  and is an invariant of $(Y, \s)$. 

\begin{proposition} [\cite{ManolescuSWF}]
The object $\SWF(Y, \s)$ is independent of the choices of  $\lambda, \mu$, $(N_{\lambda}^{\mu}, L_{\lambda}^{\mu})$ and metric $g$,  up to canonical isomorphisms in $\mathfrak{C}_{S^1}$.  If $\s$ is spin,  $\SWF(Y, \s)$ is independent of the choices, up to canonical isomorphisms in $\mathfrak{C}_{\Pin(2)}$. 
\end{proposition}

We have a duality between $\SWF(Y, \s)$ and $\SWF(-Y, \s)$.

\begin{proposition}[\cite{ManolescuGluing}] \label{prop:duality morphism}
There is a natural S-duality morphism
\[
   \eta : \SWF(Y, \s) \wedge \SWF(-Y, \s) \rightarrow S^0
\]
in $\mathfrak{C}_{S^1}$. If $\s$ is spin, we can take $\eta$ as a morphism in $\mathfrak{C}_{\Pin(2)}$. 
\end{proposition}

Let $(Y, \s)$ be a closed, $\spinc$ $3$-manifold with $b_1 > 0$.  We write  $\mathfrak{C}_{S^1}^{\mathrm{ind}}$, $\mathfrak{C}_{S^1}^{\mathrm{pro}}$ for the categories of inductive systems and projective systems in $\mathfrak{C}_{S^1}$.  In \cite {KLS1}, two types of Seiberg-Witten Floer spectra are defined: 
\[
   \begin{split}
    &   \underline{\mathit{swf}}^{A}(Y, \s, g, A_0) = (I_0 \stackrel{i_0}{\rightarrow } I_1 \stackrel{i_1}{\rightarrow} I_2 \stackrel{i_2}{\rightarrow} \cdots) \in \mathrm{Ob}( \mathfrak{C}_{S^1}^{\mathrm{ind}}),  \\
   & \underline{\mathit{swf}}^{R}(Y,\s, g, A_0) = ( \bar{I}_0 \stackrel{j_0}{\leftarrow} \bar{I}_1 \stackrel{j_1}{\leftarrow} \bar{I}_2 \stackrel{j_2}{\leftarrow} \cdots) \in \mathrm{Ob}(\mathfrak{C}_{S^1}^{\mathrm{pro}}).
   \end{split}
\]
Here $g$ is a Riemannian metric on $Y$,  $A_0$ is a fixed $\spinc$ connection on $Y$,  $I_m, \bar{I}_m$ are objects of $\mathfrak{C}$ which are desuspension of Conley indices of approximated Seiberg-Witten flow and $i_m, j_m$ are morphisms of $\mathfrak{C}$ which are induced by attractor and repeller maps.  These Floer spectra are called unfolded Seiberg-Witten Floer spectra and it is expected that their  $S^1$-homologies are isomorphic to the monopole Floer homologies with coefficients in a local system whose fiber is the group ring $\Z[H^1(Y;\Z)]$.  

These Seiberg-Witten Floer spectra depend on the choices of $g, A_0$.    If $c_1(\s)$ is torsion, we can define normalized Seiberbg-Witten Floer spectra which are independent of the choices of $g,A_0$.   Assume that $c_1(\s)$ is torsion. Then we can define a rational number $n(Y, \s, g, A_0)$ as in (\ref{eq: n(Y,g s)}).  Put
\[
    \begin{split}
         & \underline{\SWF}^{A}(Y, \s) = ( \underline{\mathit{swf}}^{A}(Y, \s, g, A_0), 0,  n(Y, \s, g, A_0)  ) \in \mathrm{Ob}(\mathfrak{C}_{S^1}^{\mathrm{ind}}), \\
         & \underline{\SWF}^{R}(Y, \s) = ( \underline{\mathit{swf}}^{R}(Y, \s, g, A_0), 0, n(Y, \s, g, A_0)  ) \in \mathrm{Ob}( \mathfrak{C}_{S^1}^{\mathrm{pro}}).
    \end{split}
\]
We can show that $\underline{\SWF}^{A}(Y, \s)$, $\underline{\SWF}^{R}(Y, \s)$ are independent of choices of $g, A_0$ up to isomorphisms.

For a spin 3-manifold $(Y, \s)$, we can define Seiberg-Witten Floer spectra $\underline{\mathit{swf}}^{A}(Y, \s, g, A_0)$, $\underline{\SWF}^{A}(Y, \s)$, $\underline{\mathit{swf}}^{R}(Y, \s, g, A_0)$,  $\underline{\SWF}^{R}(Y, \s)$ in the categories $\mathfrak{C}_{\Pin(2)}^{\mathrm{ind}}$, $\mathfrak{C}_{\Pin(2)}^{\mathrm{pro}}$. 

In  \cite[Section 4.3]{KLS2}, it was proved that  $\underline{\mathit{swf}}^{A}(Y, \s, g, A_0)$ and $\underline{\mathit{\SWF}}^{A}(Y, \s)$ are Spanier-Whitehead dual to $\underline{\mathit{swf}}^{R}(-Y, \s, g, A_0)$ and $\underline{\SWF}^{R}(-Y, \s)$.

There is another version of Seiberg-Witten Floer spectrum. In \cite{sasahira-stoffregen},  the second and third authors defined  Seiberg-Witten Floer spectrum whose $S^1$-homology is conjecturally isomorphic to the monopole Floer homology with coefficients in $\Z$. It is defined as a parametrized spectrum over the Picard torus $\mathrm{Pic}(Y)$.

\subsection{CW complex of Type SWF}\label{sec:3.3}

Following  \cite{ManolescuIntersection} and  \cite{ManolescuTriangulation}, we will make the following definition:

\begin{definition}
Let $X$ be a pointed, finite $S^1$-CW complex and $l$ be a non-negative integer. We say that $X$ is of type SWF at level $l$ if the following conditions are satisfied: 

\begin{enumerate}

\item
The action of $S^1$  is free on $X - X^{S^1}$. 

\item
The $S^1$-fixed point set $X^{S^1}$ is homotopy equivalent to $S^{l\R}$. 

\end{enumerate}

Let $X$ be a pointed $\Pin(2)$-CW complex and $l$ be a non-negative integer.  We say that $X$ is of type SWF at level $l$ if the following conditions are satisfied: 

\begin{enumerate}

\item
The action of $\Pin(2)$ is free on $X - X^{S^1}$. 

\item 
The $S^1$-fixed point set $X^{S^1}$ is $\Pin(2)$-homotopy equivalent to $S^{l\tilde{\R}}$. 

\end{enumerate}

\end{definition}

\begin{lemma}   \label{lem:fixed pt}
Let $X$ be a pointed $S^1$-CW complex of type SWF. For a subgroup $H$ of $S^1$, we have
\[
      X^{H} =
      \left\{
         \begin{array}{ll}
            X & \text{if $H = 1$},  \\
           X^{S^1} & \text{if $H \not= 1$}. 
         \end{array}
      \right. 
\] 
Here $X^{H}$ is the $H$-fixed point set.

Let $X$ be a pointed $Pin(2)$-CW complex of type SWF. For a subgroup $H$ of $\Pin(2)$, we have
\[
     X^{H} =
     \left\{
       \begin{array}{ll}
         X & \text{if $H = 1$},  \\
         X^{S^1} & \text{if $H \not=1$, $H \subset S^1$, }  \\
         X^{\Pin(2)} & \text{if $H \not \subset S^1$. }
       \end{array}
     \right.
\]
\end{lemma}

\begin{proof}
The statement follows from the condition that the action is free on $X - X^{S^1}$. 
\end{proof}

We also recall some general facts:

\begin{theorem}[Equivariant Whitehead theorem] \label{thm: eq Whitehead}
Let $G$ be a compact Lie group and $X, Y$ be $G$-CW complexes.  A $G$-map $f : X \rightarrow Y$ is a $G$-homotopy equivalence if and only if the restriction $f^{H} : X^H \rightarrow Y^{H}$ is a weak homotopy equivalence for each closed subgroup $H$ of $G$. \end{theorem}

\begin{proposition}
Let $X, Y$ be simply connected CW complexes. A continuous map $f : X \rightarrow Y$ is a homotopy equivalence if and only if the induced map $f_* : H_*(X;\Z) \rightarrow H_*(Y;\Z)$ is an isomorphism. 
\end{proposition}

\begin{lemma}[{\cite[Chapter II, Lemma (4.15)]{tom_dieck_transformation}}]
Let $G$ be a compact Lie group and $X, Y$ be pointed $G$-CW complexes such that $X = \Sigma A$ for some  pointed $G$-CW complex $A$ and $Y^{G}$ is simply connected. We write  $[X, Y]_{G}$ and $[X, Y]_{G}^0$ for the sets of homotopy classes of $G$-maps $X \rightarrow Y$ and homotopy classes of $G$-maps $X \rightarrow Y$ preserving base points respectively. Then the forgetful map
\[
    [X, Y]_{G}^0  \rightarrow [X, Y]_{G}
\]
is bijective. 
\end{lemma}

Note that for a CW complex $X$, the suspension $\Sigma^{\R^p} X$ with $p \geq 2$ is simply connected by the Freudenthal suspension theorem. 
From these facts, we obtain:

\begin{corollary} \label{cor:eq h.e SWF type}
Let $X$ and $Y$ be pointed $S^1$-CW complexes of type SWF and let $f : X \rightarrow Y$ be an $S^1$-map. Then $f$ represents an isomorphism in the $S^1$-equivariant stable homotopy category $\mathfrak{C}_{S^1}$ if the induced homomorphisms
\[
          f_* : H_{*}(X;\Z) \rightarrow H_{*}(Y;\Z), \
          f_* : H_{*}(X^{S^1};\Z) \rightarrow H_{*}(Y^{S^1};\Z)
\] 
are isomorphic. 

Let $X$ and $Y$ be pointed $\Pin(2)$-CW complex of type SWF and $f : X \rightarrow Y$ be a $\Pin(2)$-map.  Then $f$ represents an  isomorphism in  the $\Pin(2)$-equivariant stable homotopy category $\mathfrak{C}_{\Pin(2)}$ if the induced homomorphisms
\[
     f_* : H_*(X;\Z) \rightarrow H_*(Y;\Z), \ 
     f_* : H_{*}(X^{S^1}; \Z) \rightarrow H_*(Y^{S^1};\Z), \
     f_* : H_{*}(X^{\Pin(2)};\Z) \rightarrow H_*(Y^{\Pin(2)};\Z)
\]
are isomorphic. 

\end{corollary}

We can apply this corollary to the Conley index of the approximated Seiberg-Witten flow $\varphi_{\lambda}^{\mu}$ because of the following lemma:

\begin{lemma}
Let $(Y, \s)$ be a  $\spinc$ rational homology 3-sphere and $\varphi_{\lambda}^{\mu}$ be the approximated Seiberg-Witten flow of $(Y, \s)$.  For $\lambda \ll 0$ and $\mu \gg 0$, the $S^1$-equivariant Conley index $I_{\lambda}^{\mu} = N_{\lambda}^{\mu} / L_{\lambda}^{\mu}$ of $\operatorname{Inv}( B(V_{\lambda}^{\mu}, R); \varphi_{\lambda}^{\mu} )$  is a pointed $S^1$-CW complex of type SWF at level $\dim V_{\lambda}^{0}(\R)$. Here $V_{\lambda}^{0}(\R) = V_{\lambda}^{0} \cap i\Omega^1(Y)$.  

 If $\s$ is spin, $I_{\lambda}^{\mu}$ is a pointed  $\Pin(2)$-CW complex of type SWF at level $\dim V_{\lambda}^{0}(\tilde{\R})$.  Here $V_{\lambda}^{0}(\tilde{\R}) := V_{\lambda}^0 \cap \Omega^1(Y)$, which is isomorphic to a direct sum of copies of  $\tilde{\R}$ as a $\Pin(2)$-space. 
\end{lemma}

\begin{proof}

Since the $S^1$-fixed point set of $V_{\lambda}^{\mu}$ is $V_{\lambda}^{\mu}(\R)$ and the $S^1$-action on $V_{\lambda}^{\mu} - V_{\lambda}^{\mu}(\R)$  is free, the $S^1$-action is free on $I_{\lambda}^{\mu} - (I_{\lambda}^{\mu})^{S^1}$.
The approximated Seiberg-Witten flow $\varphi_{\lambda}^{\mu}$ on $V_{\lambda}^{\mu}(\R)$  is generated by the linear operator $-*d$.  Hence $(I_{\lambda}^{\mu})^{S^1}$ is homotopy equivalent to $(V_{\lambda}^{0}(\R))^+$.  Therefore $I_{\lambda}^{\mu}$ is of type SWF at level $\dim V_{\lambda}^{0}(\R)$. 

The proof for the spin case is similar. 
\end{proof}

\subsection{Relative Bauer-Furuta invariant}\label{subsec:relative-bauer-furuta}\label{sec:3.4}

Let $(X, \ts)$ be a compact, $\spinc$ 4-manifold with boundary $(Y, \s)$. Take a Riemann metric $\hat{g}$ on $X$ such that $\hat{g}$ is isometric to $g + dt^2$ near the boundary $Y$, where $g$ is a Riemann metric on $Y$. Assume that $b_1(X) = 0$ and $b_1(Y) = 0$.    Choose  a $spin^c$ connection $\hat{A}_0$ on $X$.  
For $\mu  > 0$, we have the Seiberg-Witten map
\[
  \begin{split}
   & SW^{\mu} :  L^2_k(S^+_X)  \times L^2_{k}(i \Omega_{CC}^1(X)) \rightarrow 
         L^2_{k-1}(S^-_{X}) \times   L^2_{k-1}(i\Omega^+(X)) \times L^2_{k-\frac{1}{2}}(V^{\mu})  \\
    & SW^{\mu}(\hat{\phi}, \hat{a}) 
    = (D_{X, \hat{A}_0} \hat{\phi} + \rho(\hat{a}) \hat{\phi}, d^+ \hat{a} + F_{\hat{A}_0}^+ - q(\hat{\phi}),  p^{\mu} i^* (\hat{\phi}, \hat{a})). 
   \end{split}
\]
Here $i : Y \rightarrow X$ is the inclusion,  $S^{\pm}_{X}$ are the spinor bundles on $X$ and $\Omega^1_{CC}(X)$ is the space of $1$-forms satisfying the double Coulomb condition \cite{Khandhawit}. 
Take a finite dimensional subspace $U'$ in $L^2_{k-1}(S^-_X) \times L^2_{k-1}(i\Omega^+(X))$ and $\lambda < 0$  such that  the image $\mathrm{Im} (D_{X, \hat{A}_0}, d^+, p^{\mu} i^*)$ and $U' \times V_{\lambda}^{\mu}$ are transverse in the codomain of $SW^{\mu}$.  Put
\[
        U := (D_{X, \hat{A}_0}, d^+, p^{\mu} i^*)^{-1}(U' \times V_{\lambda}^{\mu}). 
\]
We consider a finite dimensional approximation of $SW^{\mu}$ defined by
\begin{equation}\label{eq:finite-dimensional-approximation-of-seiberg-witten-map}
   \begin{split}
        & SW_{U'}^{\lambda, \mu} :  U  \rightarrow U' \times V_{\lambda}^{\mu},  \\
        & SW_{U'}^{\lambda,\mu} =  (p_{U'} \times id_{V^{\mu}} ) \circ SW^{\mu}|_{U}.
    \end{split}
\end{equation}
For $U'$ large, $\lambda \ll 0$ and $\mu \gg 0$, we can take an index pair $(N_{\lambda}^{\mu}, L_{\lambda}^{\mu})$ for the approximated Seiberg-Witten flow $\varphi_{\lambda}^{\mu}$ such that the map $SW_{U'}^{\lambda,\mu}$ induces an $S^1$-map
\[
           f_{U}^{\lambda, \mu} :    U^+ \rightarrow  (U')^+ \wedge (N_{\lambda}^{\mu}/L_{\lambda}^{\mu}).
\]
See \cite[Proposition 6]{ManolescuSWF} and \cite[Proposition 4.5]{Khandhawit}.
 The map $f_{U}^{\lambda, \mu}$ represents  a morphism
\[
     \Phi_{X, \ts} :   S^{\frac{c_1(\ts)^2 - \sigma(X)}{8} \C} \rightarrow \Sigma^{b^+(X)\R} \SWF(Y,\s)
\]
in the category $\mathfrak{C}_{S^1}$. 

Let $(X, \ts)$ be a $\spinc$-cobordism  from $(Y_0, \s_0)$ to $(Y_1, \s_1)$. Then the Seiberg-Witten equations on $X$ define a morphism
\[
     \Phi_{X, \ts} : S^{ \frac{c_1(\ts)^2 - \sigma(X)}{8} \C }  \rightarrow 
       \Sigma^{b^+(X)\R} \SWF(-Y_0, \s_0) \wedge \SWF(Y_1, \s_1).
\]
Composing with the duality morphism $\eta$ in Proposition  \ref{prop:duality morphism}, we get a morphism
\[
 \begin{split}
    \Psi_{X, \ts} : \Sigma^{\frac{c_1(\ts)^2 - \sigma(X)}{8} \C} \SWF(Y_0,\s_0) 
                         &\stackrel{\Phi_{X,\ts}}{\longrightarrow}
                         \Sigma^{b^+(X) \R} \SWF(-Y_0, \s_0) \wedge \SWF(Y_1, \s_1)  \wedge \SWF(Y_0, \s_0)    \\
                        & \stackrel{\eta}{\longrightarrow} \Sigma^{b^+(X)\R} \SWF(Y_1, \s_1)
  \end{split}
\]
in the category $\mathfrak{C}_{S^1}$. 
In the case when $\ts$ is self-conjugate, we can define $\Psi_{X, \ts}$ as a morphism in the category $\mathfrak{C}_{\Pin(2)}$, where $b^+(X)\mathbb{R}$ is replaced with $b^+(X)\tilde{\mathbb{R}}$.

Say $\s$ is a self-conjugate $\spinc$-structure on $Y$.  For $\ts,\bar{\ts}$ a pair of conjugate $\spinc$ structures restricting to $\s$, with $\bar{\ts}\neq \ts$, we can still form morphisms in $\mathfrak{C}_{\Pin(2)}$ by the following mechanism.  For the spinor bundles $S^{\pm}_{X,\ts}$, $S^{\pm}_{X,\bar{\ts}}$ the (tautological) conjugation action defines a complex-antilinear isomorphism of bundles $j\colon S^{\pm}_{X,\ts}\to S^{\pm}_{X,\bar{\ts}}$; there is also an identification of $\spinc$-connections on $S^{\pm}_{X,\ts},S^{\pm}_{X,\bar{\ts}}$, and these identifications are compatible with the double Coulomb condition.  Write $U',U,V^{\mu}_{\lambda}$ as above, for $\ts$, and write $\bar{U'},\bar{U}$ for the images of $U',U$ under these identifications.  We write $j$ for the identifications $U\to \bar{U}$ and $U'\to\bar{U}'$; we also define $j\colon \bar{U}\to U$ and $j\colon \bar{U}' \to U'$ from the bundle map $-j^{-1}$. 

Conjugation $j\colon S^{\pm}_{X,\ts}\to S^{\pm}_{X,\bar{\ts}}$ restricts on the boundary to a map $S_{Y,\s}\to S_{Y,\bar{\s}}$; by perhaps changing $j$ by a harmonic gauge transformation, we may assume that $j$ on the $4$-manifold agrees with $j\colon S_{Y,\s}\to S_{Y,\bar{\s}}$ defined for the three-manifold.  The restriction operation on connections agrees with the action of $j$ on $\spinc$-connections on $Y$ as well.  

Indeed, if the boundary $\spinc$-structure $\s$ is not self-conjugate, we may also consider $\bar{V}^{\mu}_{\lambda}$, which is defined similarly to the $\bar{U},\bar{U}'$ above.  Indeed, letting $[\s]$ be the orbit of $\s$ under conjugation (i.e. $[\s]=\{\s,\bar{\s}\}$), we define a $\Pin(2)$-equivariant spectrum
\[
\SWF(Y,[\s])=\vee_{\s\in[\s]}\SWF(Y,\s)
\]
with $j$-action defined as follows.  The configuration spaces $V^{\mu}_{\lambda}$ and $\bar{V}^{\mu}_{\lambda}$ both have components coming from sections of spinor bundles $\Gamma(S_{Y,\s})$ and $\Gamma(S_{Y,\bar{\s}})$; as before, we define $j$ on the spinor component of $V^{\mu}_{\lambda}$ and $\bar{V}^{\mu}_{\lambda}$ to be via the complex-antilinear bundle (tautological) bundle map $j\colon S_{Y,\s}\to S_{Y,\bar{\s}}$.  Also as before, there is a canonical identification of $\spinc$ connections on $S_{Y,\s},S_{Y,\bar{s}}$, giving us a map $V^{\mu}_{\lambda}\to \bar{V}^{\mu}_{\lambda}$.  As before, we use the bundle map $-j^{-1}\colon S_{Y,\bar{\s}}\to S_{Y,\s}$ to define $j\colon \bar{V}^{\mu}_{\lambda}\to V^\mu_{\lambda}$.    

Then, we have a pair of morphisms as in (\ref{eq:finite-dimensional-approximation-of-seiberg-witten-map}):
\[
SW_{U'}^{\lambda, \mu} :  U  \rightarrow U' \times V_{\lambda}^{\mu},   \qquad \qquad \mbox{ and } SW_{\bar{U}'}^{\lambda, \mu} :  \bar{U}  \rightarrow \bar{U}' \times V_{\lambda}^{\mu},  
\]

By inspecting the Seiberg-Witten map $SW$, we see that $SW_{U'}^{\lambda, \mu}=j^{-1}SW_U^{\lambda, \mu}j$.  The somewhat awkward definition of $j$ on $\bar{U}'\to U'$ (resp. $\bar{U}\to U$) was dictated by the requirement that $i^*$ be $j$-equivariant (in the event that $\s$ is not self-conjugate, one may construct a $S^1\times C_2$-equivariant map, if $\SWF(Y,[\s])$ is built as a $S^1\times C_2$-spectrum rather than as a $\Pin(2)$-spectrum). 

Indeed, we obtain that for $\ts,\s$ as above, we obtain a morphism:
\begin{equation}\label{eq:non-spin-bf}
\Phi_{X, [\ts]} :   S^{\frac{c_1(\ts)^2 - \sigma(X)}{8} \C}\vee S^{\frac{c_1(\ts)^2 - \sigma(X)}{8} \C} \rightarrow \Sigma^{b^+(X)\tilde{\R}} \SWF(Y,\s)
\end{equation}
in $\mathfrak{C}_{\Pin(2)}$.  The $j$-action on the left-hand side is $j(x,y)=(-y,x)$.  The restriction of this morphism to either sphere factor is exactly $\Phi_{X,\ts}$ (resp. $\Phi_{X,\bar{\ts}}$ as defined previously) in $\mathfrak{C}_{S^1}$. 

If $b_1(X) > 0$, the relative Bauer-Furuta invariant is defined as a morphism
\[
        \Phi_{X, \ts} : \mathit{Th}_{\mathrm{Ind} D} \rightarrow \Sigma^{b^+(X)\R} \SWF(Y, \s)
\]
in $\mathfrak{C}_{S^1}$. Here $\mathrm{Ind} D \in K(\mathrm{Pic}(X))$ is the index bundle of the family of Dirac operators on $X$ parametrized by the Picard torus $\mathrm{Pic}(X)$ and $\mathit{Th}_{\mathrm{Ind} D} \in \mathrm{Ob}(\mathfrak{C}_{S^1})$ is the Thom spectrum of $\mathrm{Ind} D$.  

Let $Y = -Y_0 \coprod Y_1$.  Composing the duality morphism $\eta$ of $\SWF(Y_0,\s)$ and $\SWF(-Y_0,\s_0)$ with $\Phi_{X, \ts}$, we obtain a morphism
\[
     \Psi_{X, \ts} : \mathit{Th}_{\mathrm{Ind} D}  \wedge \SWF(Y_0, \s_0) \rightarrow \Sigma^{b^+(X) \R} \SWF(Y_1, \s_1). 
\]
Restricting  $\Psi_{X, \ts}$  to the fiber, we get  a morphism 
\[
      \Psi_{X, \ts}^{\mathrm{fib}} : 
      \Sigma^{\frac{c_1(\s)^2 - \sigma(X)}{8}\C} \SWF(Y_0, \s_0) \rightarrow
      \Sigma^{b^+(X) \R} \SWF(Y_1, \s_1). 
\]

If $b_1(Y) > 0$, we have the type A and type R  Bauer-Furuta invariants  
\[
  \begin{split}
    &   \underline{\psi}_{X, \ts}^{A} : 
       \mathit{Th}_{\mathrm{Ind} D} \wedge \underline{\mathit{swf}}^{A}(Y_0, \s_0, g_0, A_0) \rightarrow 
        \Sigma^{b^+(X) \R \oplus V(Y_0)} \underline{\mathit{swf}}^{A}(Y_1, \s_1, g_1, A_1),   \\
   & \underline{\psi}_{X, \ts}^{R} : 
      \mathit{Th}_{\mathrm{Ind} D}  \wedge \underline{\mathit{swf}}^{R}(Y_0, \s_0, g_0, A_0) \rightarrow 
       \Sigma^{b^+(X) \R \oplus V(Y_1)} \underline{\mathit{swf}}^{R}(Y_1, \s_1, g_1, A_1), 
  \end{split}
\]
which are morphisms in $\mathfrak{C}^{\mathrm{ind}}_{S^1}$, $\mathfrak{C}^{\mathrm{pro}}_{S^1}$. 
Here $V(Y_i) = \mathrm{Coker} ( H^1(X;\R) \rightarrow H^1(Y_i;\R)  )$ for $i=0,1$  and  $\mathrm{Ind} D \in K(\mathrm{Pic}(X, Y))$ is the index bundle of a family of Dirac operators on $X$ parametrized by the relative Picard torus 
\[
         \mathrm{Pic}(X, Y) = \frac{\ker ( H^1(X;\R) \rightarrow H^1(Y;\R)) }{\ker ( H^1(X;\Z) \rightarrow H^1(Y;\Z))}.
\]
See \cite{KLS2} for the details.    

Restricting $\underline{\psi}_{X, \ts}^{A}, \underline{\psi}_{X, \ts}^{R}$ to the fibers, we have morphisms 
\[
    \begin{split}
     &   \underline{\psi}_{X, \ts}^{A, \mathrm{fib}} : \Sigma^{a\C}  \underline{\mathit{swf}}^{A}(Y_0, \s_0, g_0, A_0)   \rightarrow 
                                                         \Sigma^{b^+(X) \R\oplus V(Y_0)}  \underline{\mathit{swf}}^{A}(Y_1, \s_1, g_1, A_1), \\
    &    \underline{\psi}_{X, \ts}^{R, \mathrm{fib}} : \Sigma^{a\C}  \underline{\mathit{swf}}^{R}(Y_0, \s_0, g_0, A_0)   \rightarrow 
                                                         \Sigma^{b^+(X) \R \oplus V(Y_1)}  \underline{\mathit{swf}}^{R}(Y_1, \s_1, g_1, A_1).                                                                                            
    \end{split} 
\] 
Here $a \in \Z$ is the rank of $\mathrm{Ind} D$.

If $c_1(\s_0), c_1(\s_1)$ are torsion, we can define the normalized Bauer-Furuta invariants
\[
     \begin{split}
        &   \underline{\Psi}_{X, \ts}^{A} : 
       \mathit{Th}_{\mathrm{Ind} D} \wedge \underline{\mathit{SWF}}^{A}(Y_0, \s_0, g_0, A_0) \rightarrow 
        \Sigma^{b^+(X) \R \oplus V(Y_0)} \underline{\mathit{SWF}}^{A}(Y_1, \s_1, g_1, A_1),   \\
   & \underline{\Psi}_{X, \ts}^{R} : 
      \mathit{Th}_{\mathrm{Ind} D}  \wedge \underline{\mathit{SWF}}^{R}(Y_0, \s_0, g_0, A_0) \rightarrow 
       \Sigma^{b^+(X) \R \oplus V(Y_1)} \underline{\mathit{SWF}}^{R}(Y_1, \s_1, g_1, A_1). 
     \end{split}
\]
 Restricting  to the fibers, we have
\[
          \begin{split}
     &   \underline{\Psi}_{X, \ts}^{A, \mathrm{fib}} : \Sigma^{\frac{c_1(\ts)^2 - \sigma(X)}{8}\C}  \underline{\mathit{SWF}}^{A}(Y_0, \s_0)   \rightarrow 
                                                         \Sigma^{b^+(X) \oplus V(Y_0)}  \underline{\mathit{SWF}}^{A}(Y_1, \s_1), \\
    &    \underline{\Psi}_{X, \ts}^{R, \mathrm{fib}} : \Sigma^{\frac{c_1(\ts)^2 - \sigma(X)}{8}\C}  \underline{\mathit{SWF}}^{R}(Y_0, \s_0)   \rightarrow 
                                                         \Sigma^{b^+(X) \oplus V(Y_1)}  \underline{\mathit{SWF}}^{R}(Y_1, \s_1).                                                                                            
    \end{split} 
\]

\subsection{Adjunction relation and blow up formula}\label{subsec:adjunct}
In this subsection, we will prove  the adjunction relation and the blow-up formula  of the Bauer-Furuta invariant. 
To do it, we will need the following lemma.

\begin{lemma}[{\cite[Lemma 3.8]{bauer-furuta}}]   \label{lem : inclusion}
Let 
\[
       f : S^{p\R \oplus q\C}  \rightarrow  S^{p\R \oplus (q+k)\C}
\]
be an $S^1$-map such that the restriction
\[
     f^{S^1} :  S^{p\R}\rightarrow  S^{p\R}
\]
has degree $1$. Then $k \geq 0$ and $f$ is $S^1$-homotopic to the inclusion 
\[
       S^{ p\R \oplus q\C} \hookrightarrow S^{ p\R \oplus (q+k)\C}.
\]
\end{lemma}

In fact, more generally, stable $S^1$-maps $S^{p\R\oplus q\C}\to S^{p\R\oplus (q+k)\C}$ up to homotopy are determined by the degree of $f^{S^1}$ (as a map of nonequivariant spaces).  
See \cite[Chapter 2, Section 4]{tom_dieck_transformation}.

For an object $(W, m, n)$ of the category $\mathfrak{C}_{S^1}$,  let
\[
     U : (W, m, n) \rightarrow (\Sigma^{\C} W, m, n)
\]
be the morphism represented by the inclusion
\[
         U : W \hookrightarrow \Sigma^{\C} W,  \  w \mapsto (0, w). 
\]
For a map $f : \Sigma^{\C} W \rightarrow W'$, the composite
\[
       f U : W \rightarrow W' 
\]
is the restriction of $f$ to $W$.

\begin{proposition}\label{prop:adjunction}
Let  $(X, \ts)$ be a $\spinc$ cobordism from $(Y_0, \s_0)$ to $(Y_1, \ts_1)$ with  $b_1(X)  = b_1(\partial X) = 0$. Suppose that we have an embedded sphere $S$ in $X$ with $S \cdot S  < 0$.  Let $L$ be the complex line bundle with
\[
        c_1(L) =  PD [S].
\]
Put
\begin{equation}  \label{eq:s' n}
      \ts' = \ts \otimes L, \    n = \frac{ \left< c_1(\ts), [S] \right> + [S] \cdot [S]}{2} \in \Z.  
\end{equation}
If $n \geq 0$, we have
\[
      \Psi_{X, \ts'} U^{n} = \Psi_{X, \ts},   
\]
and if $n < 0$, we have
\[
     \Psi_{X, \ts'} =  \Psi_{X, \ts} U^{-n}.  
\]

\end{proposition}

\begin{proof}
Let $N$ be a compact tubular neighborhood of $S$ in $X$. The boundary $\partial N$ is diffeomorphic to the lens space $L(k,1)$, where $k = |S \cdot S|$.    We have
\[
      b_1(\partial N) = 0,  \ b_1(N) = 0,  \ b^+(N) = 0. 
\]
Since $\partial N$ has a positive scalar curvature metric,    the Seiberg-Witten Floer spectrum $\SWF(\partial N, \ts|_{\partial N})$ is isomorphic to  $S^{a\C} (= (S^0, 0, -a))$ for some $a \in \Q$ in $\mathfrak{C}_{S^1}$. Since $b^+(N) = 0$,  we have the relative Bauer-Furuta invariants 
\[
  \begin{split}
       &  \Psi_{N, \ts|_N} : S^{ \frac{c_1(\ts|_{N})^2 - \sigma(N)}{8} \C} \rightarrow S^{a\C},  \\
       &   \Psi_{N, \ts'|_N} : S^{ \frac{c_1(\ts'|_{N})^2 - \sigma(N)}{8} \C}  \rightarrow S^{a\C} 
  \end{split}
\]
which are represented by $S^1$-maps
\[
   \begin{split}
        &   f :  S^{p \R  \oplus q\C}  \rightarrow S^{p \R \oplus (q+r)\C},   \\
         &  f' :  S^{p\R \oplus (q+n) \C}  \rightarrow S^{p\R \oplus (q+r) \C}.
   \end{split}        
\] 
Here $p, q \in \Z$ with $p, q \gg 0$, $r =  a - \frac{c_1(\ts|_{N})^2 - \sigma(N)}{8} \in \Z$,  and $n$ is the integer defined in (\ref{eq:s' n}). Note that $f^{S^1}, (f')^{S^1} : S^{p \R} \rightarrow S^{p \R}$ represent the stable homotopy class of the finite dimensional approximation of the operator 
\[
      (d^+, p^0 i^*) : L^2_k(i\Omega^{1}_{CC}(N)) \rightarrow L^2_{k-1}(i\Omega^+(N)) \times L^2_{k-\frac{1}{2}}(V^0 \cap i\Omega^1(\partial N))
\] 
which is an isomorphism since $b_1(N) =  0$,  $b^+(N) = 0$. Hence it follows  that $f^{S^1}$ and $(f')^{S^1}$ are homotopy equivalences.

By Lemma \ref{lem : inclusion},   we have $r \geq \max \{ 0, n \}$,  and $f$, $f'$ are $S^1$-homotopic to the inclusions.  Therefore if $n \geq 0$,  the maps   $f$ and $f' U^n$ are $S^1$-homotopic. This means that
\[
         \Psi_{N, \ts'|_{N}} U^{n} =  \Psi_{N, \ts|_{N}}. 
\]
Similarly, if $n < 0$,
\[ 
      \Psi_{N, \ts'|_{N}} =  \Psi_{N, \ts|_{N}} U^{-n}.
\]
Let $Z$ be $\overline{X - N}$. By the gluing theorem \cite{ManolescuGluing} for the relative Bauer-Furuta invariants, 
\[
    \begin{split}
       \Psi_{X, \ts} &= \eta  (\Psi_{Z, \ts|_{Z}} \wedge  \Psi_{N, \ts|_N}),  \\
       \Psi_{X, \ts'} &= \eta ( \Psi_{Z, \ts|_{Z}} \wedge  \Psi_{N, \ts'|_N}).
\end{split}
\]
Here $\eta$ is the duality morphism
\[
         \SWF(\partial N, \ts|_{\partial N}) \wedge \SWF(-\partial N, \ts|_{\partial N}) \rightarrow S^0
\]
 of Proposition \ref{prop:duality morphism}. 
Therefore if $n \geq 0$,
\[
       \Psi_{X, \ts'}  U^{n} =   \Psi_{X, \ts},
\]
and if $n < 0$,
\[
       \Psi_{X, \ts'} = \Psi_{X, \ts}  U^{-n}. 
\]
\end{proof}

Next we will  show a blow up formula for the relative Bauer-Furuta invariant, which is necessary to prove Theorem \ref{thm:1.2}.  For $k \in \Z$, denote by $\ts_{k}$ the $\spinc$ on $\overline{\mathbb{CP}}^2$ with $c_1(\ts_k) = (2k+1) PD(E)$, where $E \in H_2(\overline{\mathbb{CP}}^2;\Z)$ is the class of the exceptional sphere. The Bauer-Furuta invariant $\Psi_{\overline{\mathbb{CP}}^2, \ts_k}$ is represented by an $S^1$-map 
\[
           f_k :   S^{p\R \oplus q \C}   \rightarrow  S^{p\R \oplus \left(q + \frac{k(k+1)}{2} \right) \C} 
\]
for some $p, q \in \Z$ with $p, q \gg 0$. 
Since  $b_1(\overline{\mathbb{CP}}^2) = 0$ and  $b^+(\overline{\mathbb{CP}}^2) = 0$, as in the proof of Proposition \ref{prop:adjunction}, we can deduce that 
\[
     f^{S^1}_k : S^{p\R} \rightarrow S^{p \R}
\]
is a homotopy equivalence.  By Lemma \ref{lem : inclusion}, $f_k$ is $S^1$-homotopic to the inclusion 
\[
   U^{\frac{k(k+1)}{2}} : S^{p\R \oplus q \C} \hookrightarrow S^{p\R \oplus  \left(q+ \frac{ k(k+1)}{2}\right) \C }.
\]

\begin{proposition}
Let $(X, \ts)$ be a $\spinc$ cobordism from $(Y_0, \s_0)$ to $(Y_1, \s_1)$ with $b_1(X) = b_1(\partial X) = 0$.  Then we have
\[
    \Psi_{X \#  \overline{\mathbb{CP}}^2, \ts \# \ts_{k}}
    = U^{\frac{k(k+1)}{2}} \Psi_{X, \ts}. 
\]
\end{proposition}

\begin{proof}
First, note that  the duality morphism
\[
      \eta : \SWF(S^3) \wedge \SWF(-S^3) (=S^0 \wedge S^0 = S^0) \rightarrow S^0
\]
is the identity.  

Let $D$ be a small closed disk in $X$ and put $X_0 := X - \mathrm{Int} D$.   The relative  Bauer-Furuta invariant
\[
     \Psi_{D} :  S^0 \rightarrow \SWF(S^3) = S^0
\]
of $D$ is the identity.  Hence by the gluing formula \cite{ManolescuGluing}, 
\[
    \Psi_{X, \ts} = \eta ( \Psi_{X_0, \ts|_{X_0}} \wedge \Psi_{D}) = \Psi_{X_0, \ts|_{X_0}}. 
\]
Put $X_1 = \overline{\mathbb{CP}}^2 - \mathrm{Int} D'$, where $D'$ is a small closed  disk in $\overline{\mathbb{CP}}^2$. Similarly we have
\[
      U^{\frac{k(k+1)}{2}} =  \Psi_{ \overline{\mathbb{CP}}^2, \ts_{k}} =  \Psi_{X_1, \ts_{k}|_{X_1}}. 
\]
Therefore
\[
     \Psi_{X \# \overline{\mathbb{CP}}^2, \ts \# \ts_{k} } 
     = \eta ( \Psi_{X_0, \ts|_{X_0}} \wedge \Psi_{X_1, \ts_{k}|_{X_1}}  )
     = \Psi_{X, \ts} \wedge U^{\frac{k(k+1)}{2}}
     = U^{\frac{k(k+1)}{2}} \Psi_{X, \ts}. 
\]
\end{proof}

\subsection{Surgery exact triangle for Seiberg-Witten Floer spectra}\label{sec:3.5}

Take a knot $K$ in a closed, oriented 3-manifold $Y$ and let $m$ be the meridian of $K$.  Choose $h \in H_1(Y - N(K);\Z)$ with $m \cdot h = -1$, where $N(K)$ is a tubular neighborhood of $K$ in $Y$.   
Denote by $Y_{h}(K)$ the $3$-manifold obtained by attaching a solid torus to $Y - N(K)$ with framing associated to  $h$. 
For each integer $n$, put
\[
       Y_{3n} = Y, \  
       Y_{3n+1} = Y_{h}(K), \ 
       Y_{3n+2} = Y_{h+m}(K).
\] 

We assume that $b_1(Y_n) = 0$ for all $n$.  
Put
\[
       \SWF(Y_n) := \vee_{\mathfrak{s} \in \mathrm{Spin}^c(Y_n)} \SWF(Y_n, \s). 
\]
For each $n$, we have an elementary cobordism $W_n$ from $Y_n$ to $Y_{n+1}$, obtained by adding a $2$-handle to $Y_{n} \times [0,1]$.  A computation shows 
\[
     b_1(W_n) = 0,   \  b_2(W_n) = 1. 
\]
Let $W_n'$ be the composite of the cobordsims $W_n$ and $W_{n+1}$.  Then $W_n'$ includes a $2$-sphere $E_n$ with $[E_n] \cdot [E_n] = -1$.  Fix a $\spinc$ structure $\ts'_{n,0}$ on $W_n'$ and generator $\sigma_n$ of $H^2(W_n;\Z)/\mathrm{Tor} \cong \Z$. Define 
\[
    \epsilon_n : \mathrm{Spin}^c(W_n) \rightarrow \{ \pm 1 \}
\]
as follows. For $n$ even, put $\epsilon_{n}(\s) = + 1$ for all $\ts \in \mathrm{Spin}^c(W_n)$.  Let $n$ be odd. For $\ts \in \mathrm{Spin}^c(W_n)$, define
\[
      \epsilon_n(\ts) =
      \left\{
         \begin{array}{ll}
             +1 & \text{if $c_1(\ts) = c_1(\ts_{n,0}'|_{W_n}) + 4k \sigma_n$ in $H^2(W_n;\Z) / \mathrm{Tor}$ for some integer $k$, }  \\
            -1 & \text{if $c_1(\ts) = c_1(\ts_{n,0}'|_{W_n}) + (4k+2) \sigma_n$ in $H^2(W_n;\Z) / \mathrm{Tor}$ for some integer $k$.}
         \end{array}
      \right. 
\]

Take a reduced $S^1$-homology theory $E_*$ and a positive integer $N$. 
Dfine a homomorphism
\[
    f_{n, N} : E_*(\SWF(Y_n)) \rightarrow E_*(\SWF(Y_{n+1}))
\]
by
\[
      f_{n,N}:= \sum_{\substack{\ts \in \mathrm{Spin}^c(W_n) \\ |\left< c_1(\ts), \sigma_n \right>| \leq N}}  \epsilon_n(\ts) (\Psi_{W_n, \ts})_{*}, 
\]
where $\Psi_{W_n, \ts}$ is the relative Bauer-Furuta invariant. 
The collection $\{ f_{n, N} \}_{N=1, 2,\dots}$ induces  a homomorphism
\[
   \begin{split}
     &  f_{n \bullet} : E_{\bullet}(\SWF(Y_n)) \rightarrow E_{\bullet}(\SWF(Y_{n+1})). 
   \end{split}
\] 
Here $E_{\bullet}(\SWF(Y_n))$ is the completion with respect to a decreasing  filtration $\{ E_{* < -p}(\SWF(Y_n)) \}_{p= 1, 2,\dots}$:
\[
     E_{\bullet}(\SWF(Y_n)) = \varprojlim_{p} E_{*}(\SWF(Y_n)) / E_{* < -p}(\SWF(Y_n)). 
\]

In \cite{Sasahira-Stoffregen_Triangle}, the second and third authors proved the following: 

\begin{theorem}
We have the following exact sequence: 
\[
     \cdots \xrightarrow{f_{n-1 \bullet}}
         E_{\bullet}(\SWF(Y_n))   \xrightarrow{f_{n \bullet}}
         E_{\bullet}(\SWF(Y_{n+1}))  \xrightarrow{f_{n+1 \bullet}}
         E_{\bullet}(\SWF(Y_{n+2}))   \xrightarrow{f_{n+2 \bullet}} \cdots.
\]
\end{theorem}

Since $\tilde{H}_{-p}^{S^1}(\SWF(Y_n);\Z) = 0$ for $p$  large, we have $\tilde{H}_{\bullet}^{S^1}(\SWF(Y_n);\Z) = \tilde{H}_{*}^{S^1}(\SWF(Y_n);\Z)$.  Therefore we obtain the following corollary.

\begin{corollary}
We have the following exact sequence
\[
       \cdots 
      \xrightarrow{f_{n-1 *}} \tilde{H}_{*}^{S^1}(\SWF(Y_n);\Z)
      \xrightarrow{f_{n *}} \tilde{H}_{*}^{S^1}(\SWF(Y_{n+1});\Z) 
      \xrightarrow{f_{n+1 *}} \tilde{H}_{*}^{S^1}(\SWF(Y_{n+2});\Z)
      \xrightarrow{f_{n+2 *}}
      \cdots.
\]
\end{corollary}

\subsection{Connected sum formula for Seiberg-Witten Floer spectra}

We will prove the following: 

\begin{theorem} \label{thm:connected sum SWF}
Let $(Y_1, \s_1), (Y_2, \s_2)$ be closed, $\spinc$ 3-manifolds with $b_1 = 0$. Then we have an isomorphism
\[
       \SWF(Y_1 \# Y_2, \s_1 \# \s_2) \cong \SWF(Y_1, \s_1) \wedge \SWF(Y_2, \s_2)
\]
in $\mathfrak{C}_{S^1}$. 

If $(Y_1, \s_1), (Y_2, \s_2)$ are spin, we have
\[
       \SWF(Y_1 \# Y_2, \s_1 \# \s_2) \cong \SWF(Y_1, \s_1) \wedge \SWF(Y_2, \s_2)
\]
in $\mathfrak{C}_{\Pin(2)}$.

\end{theorem}

We have a natural $\spinc$ cobordism $(W, \ts)$ from $(Y_1, \s_1) \coprod (Y_2, \s_2)$ to $(Y_1 \# Y_2, \s_1 \# \s_2)$, which is a boundary connected sum of $(Y_1 \times [0,1],  \pi_1^* \s_1)$ and $(Y_2 \times [0,1], \pi_2^* \s_2)$.  Here $\pi_i : Y_i \times [0,1] \rightarrow Y_i$ is the projection. Note that 
\[
          \SWF(Y_1 \coprod Y_2, \s_1 \coprod \s_2) = \SWF(Y_1, \s_1) \wedge \SWF(Y_2, \s_2).
\]
We will prove that the Bauer-Furuta invariant
\[
         \Psi_{W, \ts}: \SWF(Y_1, \s_1) \wedge \SWF(Y_2, \s_2)  \rightarrow \SWF(Y_1 \# Y_2, \s_1 \# \s_2)
\]
is an isomorphism.

Put 
\[
     W' := W \cup_{Y_1 \# Y_2 } (-W),   \    \ts' = \ts \cup_{Y_1 \# Y_2} \ts.
\]   
Then $(W', \ts')$ is a $\spinc$ cobordism from $(Y_1 \coprod Y_2, \s_1 \coprod \s_2)$ to itself.   
By the gluing theorem \cite{ManolescuGluing},  
\[
       \Psi_{W', \ts'} = \Psi_{-W, \ts} \circ \Psi_{W, \ts}. 
\]
Note that  $(W', \ts')$ is a connected sum of $(Y_1 \times [0,1], \pi_1^*\s_1)$ and $(Y_2 \times [0,1], \pi_2^* \s_2)$. By the gluing formula and the fact that the duality morphism 
\[
     \eta_{S^3} : \SWF(S^3) \wedge SWF(-S^3) ( = S^0 \wedge S^0) \rightarrow S^0
\] 
is the identity,  we have
\[
       \Psi_{W', \ts'} =  \eta_{S^3} \circ (\Psi_{Y_1 \times [0,1],\pi_1^* \s_1} \wedge \Psi_{Y_2 \times [0,1], \pi_2^* \s_2})
                         = id_{ \SWF(Y_1, \s_1) \wedge \SWF(Y_2, \s_2)  }. 
\]
Here we have used $\Psi_{Y_i \times [0,1], \pi_i^* \s_i} = id_{\SWF(Y_i, \s_i)}$. See \cite[Section 6]{Sasahira-Stoffregen_Triangle}. 
 Therefore we have obtained
\[
      \Psi_{-W, \ts} \circ \Psi_{W, \ts} = id_{ \SWF(Y_1, \s_1) \wedge \SWF(Y_2, \s_2)  }.
\]

Put
\[
         W'' = (-W) \cup_{Y_1 \coprod Y_2} W, \ 
         \ts'' = \ts \cup_{Y_1 \coprod Y_2} \ts. 
\]
Then $(W'', \ts'')$ is a $\spinc$ cobordism from $(Y_1 \# Y_2, \s_1 \# \s_2)$ to itself and $b_1(W'') = 1$.  
Although $Y_1 \coprod Y_2$ is not connected, we can still apply  arguments in the proofs of \cite[Theorem 1]{ManolescuGluing}, \cite[Theorem 1.5]{KLS2} to $\Psi_{W'', \ts''}^{\mathrm{fib}}$  and   we have 
\begin{equation} \label{eq:BF W'' 1}
       \Psi_{W'', \ts''}^{\mathrm{fib}} 
       = \Psi_{W,\ts} \circ \Psi_{-W, \ts}. 
\end{equation}
Here $\Psi_{W'', \ts''}^{\mathrm{fib}}$ is the restriction of the Bauer-Furuta invariant $\Psi_{W'', \ts''}$ to the fiber of the Thom spectrum of the index bundle of Dirac operators on $W''$.  See Section \ref{subsec:relative-bauer-furuta}. 
We take a simple closed curve $c$ in $W''$ which consists of curves $c_1$ and $c_2$. Here $c_1$ is a curve in $-W$ from a point $p_1$ on $Y_1$ to a point $p_2$ on $Y_2$ and $c_2$ is a curve in $W$ from $p_1$ to $p_2$. The curve $c$ represents a generator of $H_1(W'';\Z) \cong \Z$.  Take a small compact neighborhood $N$ of $c$ in $W''$ with $N \cong S^1 \times D^3$.  Put $\tilde{W}'' = W'' - \mathrm{Int} N$.  
Let  $\s_{S^1 \times S^2}$ be the $\spinc$ structure on $S^1 \times S^2$ with $c_1 =0$.  By Theorem 1.5 of \cite{KLS2}, 
\begin{equation} \label{eq:BF W'' 2}
      \Psi_{W'', \ts''}^{\mathrm{fib}} 
      = \Psi_{W'', \ts''}|_{\mathrm{Pic}(W'', \partial N)}
      =  \eta_{S^1 \times S^2} \circ ( \underline{\Psi}^{R}_{\tilde{W}'', \ts''|_{\tilde{W}''}} \wedge \underline{\Psi}^{A}_{N, \ts''|_{N}}). 
\end{equation}
Here $\eta_{S^1 \times S^2}$ is the duality morphism of $\underline{\SWF}^{A}(S^1 \times S^2, \s_{S^1 \times S^2})$ and $\underline{\SWF}^{R}(-S^1 \times S^2, \s_{S^1 \times S^2})$. 
Note that if we attach $D^2 \times S^2$ to $\tilde{W}''$, we get the trivial cobordism:
\[
        \tilde{W}'' \cup_{S^1 \times S^2} D^2 \times S^2 =  Y_1 \# Y_2 \times [0,1]. 
\]
Hence
\begin{equation} \label{eq:BF D2 times S2}
   \eta_{S^1 \times S^2} \circ  ( \underline{\Psi}^{R}_{\tilde{W}'', \ts''|_{\tilde{W}''}} \wedge \underline{\Psi}^{A}_{D^2 \times S^2, \ts_{D^2 \times S^2}}) 
    = id_{\SWF(Y_1 \# Y_2, \s_1 \# \s_2)}.
\end{equation}
As shown in \cite[Example 1.4]{KLS2}, both of $\underline{\Psi}^{A}_{N, \ts''|_{N}}$ and $\underline{\Psi}^{A}_{D^2 \times S^2, \ts_{D^2 \times S^2}}$ are represented by the identity $S^{p\R \oplus q\C} \rightarrow S^{p\R \oplus q\C}$.
 Therefore it follows from (\ref{eq:BF W'' 1}), (\ref{eq:BF W'' 2}) and (\ref{eq:BF D2 times S2}) that 
\[
         \Psi_{W, \ts} \circ \Psi_{-W, \ts} = id_{\SWF(Y_1 \# Y_2, \ts_1 \# \ts_2)}.
\]
We have proved Theorem \ref{thm:connected sum SWF}. 


\section{Review of Lattice Homology}\label{sec:4}

In this section, we review the definition of lattice homology and discuss several lattice-theoretic results due to Ozsv\'ath and Szab\'o \cite{OSplumbed} and N\'emethi \cite{NemethiOS, Nemethi}. Throughout, let $\Gamma$ be a negative-definite plumbing graph. Denote the corresponding plumbed $4$-manifold by $W_\Gamma$ and its boundary by $Y_\Gamma = \partial W_\Gamma$. Let $L_\Gamma \cong \Z^{|\Gamma|}$ be the lattice formally spanned by the vertices of $\Gamma$ and let $(-, -)_\Gamma$ be the intersection pairing on $L_\Gamma$. 

Recall that an element $k \in \Hom(L_\Gamma, \Z)$ is \textit{characteristic} if $k(x) \equiv x^2 \bmod 2$ for all $x \in L_\Gamma$. Denote the set of characteristic elements by $\Char_\Gamma$. The set $\Char_\Gamma$ has an action of $L_\Gamma$ which sends $k \mapsto k + 2x^*$, where $x^*$ is the Poincar\'e dual $x^* = (x, -)_\Gamma$. The action of $L_\Gamma$ partitions the elements of $\Char_\Gamma$ into equivalence classes; we denote a generic equivalence class by $[k]$. It is a standard fact that the $\spinc$-structures on $W_\Gamma$ are in bijection with $\Char_\Gamma$, and that the $\spinc$-structures on $Y_\Gamma$ are in bijection with the set of equivalence classes $[k]$. 

Extend $(-, -)_\Gamma$ to a $\Q$-valued pairing on $L_\Gamma \otimes \Q$. Since $(-, -)_\Gamma$ has non-zero determinant, any element $k \in \Hom(L_\Gamma, \Z)$ may be written as $(x, -)_\Gamma$ for some rational vector $x \in L_\Gamma \otimes \Q$. This correspondence allows us to define a $\Q$-valued intersection pairing on $\Hom(L_\Gamma, \Z)$, which by abuse of notation we also denote by $(-,-)_\Gamma$.

Finally, note that each equivalence class $[k]$ is an affine lattice over $L_\Gamma$. This gives a cubical subdivision of $[k] \otimes \R \cong \R^{|\Gamma|}$ coming from the obvious cubical subdivision of $L_\Gamma \otimes \R \cong \R^{|\Gamma|}$:

\begin{definition}\label{def:4.1}
Define the \textit{zero-dimensional cubes} of $[k]$ to be the elements of $[k]$. In general, a \textit{$d$-dimensional cube} of $[k]$ is the convex hull (in $[k] \otimes \R$) of a set of $2^d$ elements of the form
\[
\left\{ k + \sum_{i = 1}^d c_i (2v_i^*) \ | \ (c_1, \ldots, c_d) \in \{0, 1\}^d \right\}
\]
for some $k \in [k]$ and set of distinct vertices $v_1, \ldots, v_d$ in $\Gamma$.\footnote{In order to avoid an expansion of notation, we will write a generic equivalence class by $[k]$ and denote a generic element of such a class also by $k$.} For example, a square in $[k]$ is the convex hull of $\{k, k + 2v^*, k + 2w^*, k + 2v^* + 2w^*\}$ for some characteristic element $k \in [k]$ and pair of distinct vertices $v$ and $w$. We denote a typical cube of dimension $d$ by $\square_d$. Strictly speaking, each cube (as a geometric object) is a subset of $[k] \otimes \R$, but we refer to these as being in $[k]$ for convenience. We often refer to zero-dimensional cubes as \textit{lattice points} and one-dimensional cubes as \textit{lattice edges}.
\end{definition}

\subsection{The lattice complex}\label{sec:4.1}

Following \cite[Definition 3.1.8]{Nemethi}, we now recall the definition of lattice homology. 

\begin{definition}\label{def:4.2}
Let $[k]$ be an equivalence class in $\Char_\Gamma$. Define the \textit{weight function} $w$ on $[k]$ as follows. For any zero-dimensional cube $k$ in $[k]$, define
\[
w(k) = \dfrac{1}{4}\left(k^2 + |\Gamma|\right).
\]
For any $d$-dimensional cube $\square_d$ in $[k]$, define $w(\square_d)$ to be the minimum over the weights of all lattice points appearing as vertices of $\square_d$:
\[
w(\square_d) = \min_k \{w(k) \ | \ k \text{ a vertex of } \square_d\}.
\]
Thus, $w$ is a function on the set of all cubes (of any dimension) in $[k]$.
\end{definition}

For any fixed equivalence class $[k]$, the weight function $w$ takes values in a coset of $2\Z$ in $\Q$. To see this, note that for any pair of adjacent lattice points, we have
\[
w(k + 2v^*) - w(k) = (k, v^*)_\Gamma + (v^*, v^*)_\Gamma = k(v) + v^2 \equiv 0 \bmod{2}
\]
since $k$ is characteristic.


\begin{definition}\label{def:4.3}
Define the \textit{lattice complex} $\Cla(\Gamma, [k])$ as follows. Let $\Cla_d(\Gamma, [k])$ be the free (infinitely-generated) $\Z[U]$-module formally spanned by the set of $d$-dimensional cubes in $[k]$:
\[
\Cla_d(\Gamma, [k]) = \spa_{\Z[U]} \{\square_d \ | \square_d \text{ a $d$-dimensional cube in } [k]\}.
\]
Set 
\[
\Cla(\Gamma, [k]) = \bigoplus_{d \geq 0} \Cla_d(\Gamma, [k]).
\]
This has two gradings. The first is just $d$, which we call the \textit{dimensional grading}. The second is the \textit{Maslov grading}, which is defined on cubes by setting $\gr(\square_d) = w(\square_d) + d$. We extend this to all of $\Cla(\Gamma, [k])$ by setting $\gr(U) = -2$. The differential is given by inserting powers of $U$ into the usual cellular/cubical differential to make $\gr(\partial) = -1$. That is, let
\[
\partial(\square_d) = \sum_{\square_{d-1} \in \text{ cellular boundary of } \square_d} U^{\left(w(\square_{d-1}) - w(\square_d)\right)/2} \epsilon(\square_{d-1}) \square_{d-1}
\]
and extend $\Z[U]$-linearly. Here, the signs $\epsilon(\square_{d-1}) = \pm 1$ are chosen in the usual way by thinking of the cubes of $[k]$ as coming from a cubical decomposition of $\R^{|\Gamma|}$ and choosing orientations. Note that $\partial$ drops both the dimensional and Maslov gradings by one. The homology of $(\Cla(\Gamma, [k]), \partial)$ is called the \textit{lattice homology} and is denoted by 
\[
\Hla(\Gamma, [k]) = \bigoplus_{d \geq 0} \Hla_d(\Gamma, [k]),
\]
where $\Hla_d(\Gamma, [k])$ is the homology in dimensional grading $d$. We write $\Hla(\Gamma)$ to denote the sum of $\Hla(\Gamma, [k])$ over all equivalence classes $[k]$, and similarly write $\Hla_d(\Gamma)$ to denote the analogous sum of the groups $\Hla_d(\Gamma, [k])$.
\end{definition}

In \cite[Proposition 3.4.2]{Nemethi}, it is shown that the lattice homology (in fact, the homotopy equivalence class of the lattice chain complex) is an invariant of $(Y_\Gamma, \s)$, rather than the plumbing graph $\Gamma$.

\begin{remark}\label{rem:4.4}
The definition of lattice homology we have presented here is essentially due to N\'emethi \cite{Nemethi}, but we use slightly different conventions. For the convenience of the reader familiar with \cite{Nemethi}, we translate between these conventions here:
\begin{enumerate}
\item In \cite{OSplumbed, NemethiOS, Nemethi}, the authors work with the lattice \textit{co}homology $\Hla^d(\Gamma, [k])$. One can show that (up to overall grading shift) $\Hla^d(\Gamma, [k])$ is dual to $\Hla_d(\Gamma, [k])$. This duality isomorphism multiplies Maslov grading by $-1$, so that the action of $U$ on $\Hla^d(\Gamma, [k])$ is still of degree $-2$. Correspondingly, the isomorphism theorems of \cite{OSplumbed, NemethiOS, Nemethi} are stated in terms of $\HFp(-Y_\Gamma, \s)$, rather than $\HFm(Y_\Gamma, \s)$.
\item In \cite{Nemethi}, the weight function is defined slightly differently than in Definition~\ref{def:4.2}. This is because in \cite{Nemethi}, the weight function is defined on lattice cubes in $L_\Gamma \otimes \R$, rather than $[k] \otimes \R$. Explicitly, in \cite{Nemethi} the weight function is defined by fixing a reference characteristic element $k_0 \in [k]$ and setting
\[
\widetilde{w}(x) = -\dfrac{1}{2}(k_0(x) + x^2)
\]
for any $x \in L_\Gamma$ and 
\[
\widetilde{w}(\square_d) = \max_x \{\widetilde{w}(x) \ | \ x \text{ a vertex of } \square_d\}
\]
for any lattice cube $\square_d$ in $L_\Gamma$; see \cite[Section 3.2.1]{Nemethi}. To see that $w$ (as in Definition~\ref{def:4.2}) and $\widetilde{w}$ (which we use to denote the weight function of  \cite[Section 3.2.1]{Nemethi}) contain the same information, note that the map $x \mapsto k_0 + 2x^*$ gives an isomorphism between $L_\Gamma$ and $[k]$. Computing, we see that 
\[
w(k_0 + 2x^*) = \dfrac{1}{4}\left( k_0^2 + 4k_0(x) + 4x^2 + |\Gamma|\right ) = -2\tilde{w}(x) + \dfrac{1}{4} \left( k_0^2 + |\Gamma| \right).
\]
Although the left- and right-hand sides are not identical, the differences are cosmetic. The minus sign in front of $\tilde{w}(x)$ is due to the usage of lattice cohomology, as explained above. The factor of two is a convention: in \cite[Section 3.1.3]{Nemethi}, the Maslov grading is defined by doubling the weight function, whereas we have absorbed this into the definition of $w$. Finally, the term $(k_0^2 + |\Gamma|)/4$ is a constant grading shift depending on the auxiliary choice of $k_0$. This is cancelled out manually in \cite[Theorem 5.2.2]{Nemethi}, whereas we have again instead absorbed this into the definition of $w$.
\end{enumerate}
\end{remark}

\subsection{The lattice map}\label{sec:4.2}

In this paper, we will primarily be concerned with the lattice homology $\Hla_0(\Gamma, [k])$ in dimensional grading zero. It is a fundamental result of N\'emethi that if $Y_\Gamma$ is an AR plumbed manifold then $\Hla_d(\Gamma, [k]) = 0$ for all $d \geq 1$ \cite[Theorem 4.3.3]{Nemethi}.

It will be helpful for us to re-interpret $\Hla_0(\Gamma, [k])$ more explicitly in terms of $\Char_\Gamma$. We have
 \begin{equation}\label{eq:4.1}
 \Hla_0(\Gamma, [k]) = \spa_{\Z[U]}\{k \in [k]\} / \sim,
 \end{equation}
where $\sim$ is generated as follows. Let $k$ be any characteristic vector in $[k]$ and $v$ be any vertex in $\Gamma$. Denote
\[
 n = \dfrac{1}{2}\left(w(k + 2v^*) - w(k)\right) = \dfrac{1}{2}\left(k(v) + v^2\right).
 \]
Then we declare
\[
\begin{cases}
k \sim U^n (k + 2v^*) &\text{ if } n \geq 0 \\
U^{-n} k \sim k  + 2v^*&\text{ if } n < 0.
\end{cases}
\]
The fact that this $\Z[U]$-module is isomorphic to $\Hla_0(\Gamma, [k])$ (as in Definition~\ref{def:4.3}) follows from observing that $\sim$ corresponds precisely to the relations introduced by $\im \partial \colon \Cla_1(\Gamma, [k]) \rightarrow \Cla_0(\Gamma, [k])$. 

The reason that this re-interpretation is helpful is that we may construct a map from $\Hla_0(\Gamma, [k])$ into the Floer homology of $Y_\Gamma$:

\begin{definition}\label{def:4.5}
Let $[k]$ be an equivalence class in $\Char_\Gamma$ and let $\s$ be the $\spinc$-structure on $Y_\Gamma$ corresponding to $[k]$. We define the \textit{lattice homology map}
\[
\T \colon \Hla_0(\Gamma, [k]) \rightarrow \HFm(Y_\Gamma, \s)
\]
as follows. Fix, once and for all, a generator $1 \in \HFm(S^3)$. Let $k$ be any characteristic vector in $[k]$. Consider the cobordism map
\[
F_{W_\Gamma, k} \colon \HFm(S^3) \rightarrow \HFm(Y_\Gamma, \s)
\]
associated to $W_\Gamma$ with characteristic vector $k$. We then define
\[
\T(k) = F_{W_\Gamma, k}(1),
\]
extending $\Z[U]$-linearly. The fact that $\T$ respects $\sim$ follows from the adjunction relations, as discussed in \cite{OSplumbed}. Note that the Floer-theoretic Maslov grading shift $\gr(F_{W_\Gamma, k})$ is precisely $\gr(k)$, so $\T$ is grading-preserving.\footnote{Here, we use the convention that $1 \in \HFm(S^3)$ has Maslov grading zero. The original convention of Ozsv\'ath and Szab\'o is that $1 \in \HFm(S^3)$ has Maslov grading $-2$, in which case $\T$ has grading shift $-2$.}
\end{definition}

\begin{remark}\label{rem:4.6}
In addition to the choice of generator $1 \in \HFm(S^3)$, each cobordism map $F_{W_\Gamma, k}$ is canonically defined only up to multiplication by $\pm 1$ unless additional data is chosen. However, it is straightforward to fix signs for each $F_{W_\Gamma, k}$ so that $\T$ satisfies the adjunction relations corresponding to $\sim$; see \cite[Section 2.1]{OSplumbed}.
\end{remark}
 
We have used the Heegaard Floer homology $\HFm$ in Definition~\ref{def:4.5}, but it is clear that $\HFm$ can be replaced by monopole Floer homology or (in light of \cite{Sasahira-Stoffregen_Triangle}) by Seiberg-Witten Floer \textit{homology}. As discussed in Section~\ref{sec:2}, the principal idea in this paper will be to construct analogues of $\Hla_0(\Gamma, [k])$ and $\T$ in the setting of Seiberg-Witten Floer \textit{homotopy}.

It is a fundamental result of Ozsv\'ath and Szab\'o \cite[Theorem 1.2]{OSplumbed} and N\'emethi \cite[Theorem 8.3]{NemethiOS} that $\T$ is an isomorphism if $\Gamma$ is (for example) an AR plumbing. The proof of this rests on certain formal properties of the Heegaard Floer package, chief among which is the presence of a surgery exact sequence. By work of the second and third authors, a similar surgery exact sequence holds for Seiberg-Witten Floer homology \cite{Sasahira-Stoffregen_Triangle}. Following the same arguments as in \cite{OSplumbed, NemethiOS} then shows that $\T$ is likewise an isomorphism in our setting. For readers unfamiliar with \cite{OSplumbed, NemethiOS}, an abbreviated discussion of this is included in Section~\ref{sec:8}.

\subsection{Paths and graded roots}\label{sec:4.3}

There is one more lattice-theoretic construction which we will need in this paper:

\begin{definition}
Define a \textit{path} in $[k]$ to be a sequence $(k_i)$ of lattice points in $[k]$ such that:
\begin{enumerate}
\item $k_i \neq k_j$ for $i \neq j$; and,
\item for each $i$, we have $k_{i+1} = k_i + 2v^*$ for some vertex $v \in \Gamma$.
\end{enumerate}
We will usually assume that the sequence $(k_i)$ is finite. Often, we will consider the set of edges from each $k_i$ to $k_{i+1}$ to also be a part of the data of $\gamma$.
\end{definition}

As discussed in Definition~\ref{def:4.1}, we may consider the cubical subdivision of $[k] \otimes \R \cong \R^{|\Gamma|}$ as a CW complex. It is clear from the Definition~\ref{def:4.3} that choosing any cubical subcomplex $C$ of $\R^{|\Gamma|}$ gives a subcomplex of $\C(\Gamma, [k])$, simply by considering the $\Z[U]$-span of all lattice cubes present in $C$. Since any path $\gamma$ is an example of such a subcomplex, we have:

\begin{definition}
Let $\gamma$ be a path in $[k]$. Define the \textit{path lattice complex} $\Cla(\gamma, [k])$ to be the subcomplex of $\Cla(\Gamma, [k])$ generated over $\Z[U]$ by the vertices and edges present in $\gamma$. The homology of this subcomplex is called the \textit{path lattice homology} and is denoted by $\Hla(\gamma, [k])$.
\end{definition}

\begin{lemma}
For any path $\gamma$, the path lattice homology $\Hla(\gamma, [k])$ is supported only in dimensional grading zero and is a submodule of $\Hla_0(\Gamma, [k])$.
\end{lemma}
\begin{proof}
The fact that $\Hla(\gamma, [k])$ is supported only in dimensional grading zero is clear from the fact that $\gamma$ is non-self-intersecting, as this implies there are no cycles in dimensional grading one. We thus check that the map from $\Hla(\gamma, [k])$ to $\Hla_0(\Gamma, [k])$ is injective. Recall that $\sim$ is generated by individual relations corresponding to the edges in $\Cla_1(\Gamma, [k])$. However, each individual relation is complete in the following sense: if $k$ and $k + 2v^*$ are two adjacent vertices in $\Hla_0(\Gamma, [k])$, then not only do $k$ and $k + 2v^*$ satisfy the obvious edge relation, but in fact this is the only relation between $k$ and $k + 2v^*$. Indeed, note that if $k$, $k + 2v^*$, $k + 2w^*$, and $k + 2v^* + 2w^*$ are four corners of a square, then the sum of any three edge relations is a $U$-multiple of the fourth, as can be seen by considering the boundary of the corresponding square in $\Cla_2(\Gamma, [k])$. A general relation between $k$ and $k + 2v^*$ arises from summing together the relations along a sequence of edges from $k$ to $k + 2v^*$. However, repeatedly applying the previous principle shows that we can reduce any such sum to the obvious individual edge relation from $k$ to $k + 2v^*$, as claimed. Hence if $k$ and $k + 2v^*$ are two successive vertices in $\gamma$, then any relation between them in $\Hla_0(\Gamma, [k])$ is already present in $\Hla(\gamma, [k])$.
\end{proof}

As in the previous section, we have a presentation of $\Hla(\gamma, [k])$ which is generated over $\Z[U]$ by the characteristic vectors $k_i$ in $\gamma$, subject to the same relation $\sim$ as in (\ref{eq:4.1}). It is a fundamental result due to N\'emethi \cite{NemethiOS} that if $\Gamma$ is an AR plumbing, then there exists a finite path $\gamma$ in $[k]$ which carries the lattice homology in the following sense:

\begin{theorem}\cite[Remark 9.4(a)]{NemethiOS}\label{thm:4.9}
Let $\Gamma$ be an AR plumbing and $[k]$ be an equivalence class in $\Char_\Gamma$. Then there exists a finite path $\gamma$ in $[k]$ such that the obvious inclusion map
\[
i_* \colon \Cla(\gamma, [k]) \rightarrow \Cla(\Gamma, [k])
\]
induces an isomorphism on homology. In particular, the restriction of $\T$ to $\Hla(\gamma, [k])$ is an isomorphism:
\[
\T | _{\Hla(\gamma, [k])} \colon \Hla(\gamma, [k]) \cong \Hla_0(\Gamma, [k]) \xrightarrow{\cong} \HFm(Y, \s).
\]
\end{theorem}
\begin{proof}
The path in question is found by repeatedly applying Laufer's algorithm, as discussed in \cite[Section 7]{NemethiOS}. Although \cite[Remark 9.4]{NemethiOS} states that the path is infinite, by \cite[Theorem 9.3(a)]{NemethiOS} we may cut off the path after some sufficiently high index (although this is not straightforward to determine explicitly).
\end{proof}

Path lattice homology reconciles the general construction of lattice homology with the graded root picture of Section~\ref{sec:2.1}. For any path $\gamma$, the presentation (\ref{eq:4.1}) makes $\Hla(\gamma, [k])$ into a graded root $R$, with the lattice points of $\gamma$ corresponding to the leaves of $R$. Technically, if $R$ is constructed directly from (\ref{eq:4.1}), then it will have a large number of redundant generators $k_i$: for example, we may have pairs of relations of the form $U^m k_{i-1} \sim k_{i} \sim U^n k_{i+1}$. In Section~\ref{sec:2.1}, we have used a minimal presentation of $R$ obtained by discarding such $k_i$ and working directly with (for example) $U^m k_{i-1} \sim U^n k_{i+1}$. We view redundant generators as degenerate leaves of $R$, whose corresponding towers are completely glued to a tower on one or the other side. Clearly, this does not make any difference to the graded root; however, in what follows it will be notationally convenient for us to still treat $k_i$ as a leaf.

\section{Lattice Spectrum}\label{sec:5}

In this section, we define the lattice spectrum associated to an AR plumbing. In fact, we first give a more general construction which applies to any negative-definite plumbing, although this will not be used in the current paper. We then finish the proof of Theorem~\ref{thm:1.1}. 

\subsection{The general construction}\label{sec:5.1}
Let $\Gamma$ be a negative-definite plumbing and $[k]$ be an equivalence class in $\Char_\Gamma$. Recall that the weight function $w$ of Definition~\ref{def:4.2} is valued in a coset of $2\Z$ in $\Q$; fix an auxiliary cutoff weight $h$ lying in this coset.  We write $S^{n\C}=(\C^n)^+$ as $S^1$-spaces.  

For each $d$-dimensional cube in $[k]$ with $w(\square_d) \geq h$, let 
\begin{equation}\label{eq:5.1}
\mathcal{F}(\square_d) = S^{((w(\square_d) - h)/2)\C} \wedge \square_d^+.
\end{equation}
We make $\smash{\mathcal{F}(\square_d)}$ into a pointed $S^1$-space by giving $\smash{\square_d^+}$ the trivial $S^1$-action. Note that $\smash{\mathcal{F}(\square_d)}$ is just the product $\smash{S^{((w(\square_d) - h)/2)\C} \times \square_d}$, except that $\{\infty\} \times \square_d$ is collapsed to a single basepoint to stay in the category of pointed $S^1$-spaces. If $d = 0$, then $\smash{\mathcal{F}(\square_d)}$ is just the sphere $\smash{\spwt}$.

Now suppose $\square_{d-1}$ is a face of $\square_d$, so that we have an inclusion 
\[
f \colon \square_{d-1}^+ \rightarrow \square_d^+. 
\]
This gives a way to glue $\mathcal{F}(\square_{d-1})$ and $\mathcal{F}(\square_d)$ together: for any $\smash{x \in \spwt}$ and $t \in \square_{d-1}^+$, identify the points
\begin{equation}\label{eq:5.2}
(i(x), t) \sim (x, f(t))
\end{equation}
in $\smash{S^{((w(\square_{d-1}) - h)/2)\C} \wedge \square_{d-1}^+}$ and $\smash{\spwt \wedge \square_d^+}$, respectively. Here, $i$ is the usual inclusion map of $S^1$-spheres induced by the inclusion 
\[
\C^{(w(\square_{d-1}) - h)/2} \hookrightarrow \C^{(w(\square_{d})-h)/2}
\]
along the first $(w(\square_{d-1})-h)/2$-coordinates. Here, we are using the fact that if $\square_{d-1}$ is a face of $\square_d$, then $w(\square_d) \leq w(\square_{d-1})$. We then let 
\begin{equation}\label{eq:5.3}
X = \left( \bigsqcup_{\substack{\square_d \text{ in } [k] \text{ with } w(\square_d) \geq h}} \mathcal{F}(\square_d) \right) / \sim,
\end{equation}
where the quotient $\sim$ glues together $\mathcal{F}(\square_{d-1})$ and $\mathcal{F}(\square_d)$ according to (\ref{eq:5.2}) whenever $\square_{d-1}$ is a face of $\square_d$. See Figure~\ref{fig:5.1}. 

\begin{remark}\label{rem:5.1}
It may be psychologically easier for the reader to ignore basepoints and replace (\ref{eq:5.1}) with
\[
\mathcal{F}'(\square_d) = S^{((w(\square_d) - h)/2)\C} \times \square_d.
\]
The difference is largely cosmetic: let $X'$ be the space constructed as in (\ref{eq:5.3}) by using $\mathcal{F}'$ instead of $\mathcal{F}$. Let $S_h$ be the subset of $[k] \otimes \R$ consisting of the union of all cubes in $[k]$ of weight greater than or equal to $h$. Then the $S^1$-fixed point set of $X'$ is homeomorphic to 
\[
S^0 \times S_h = (\{0\} \times S_h) \sqcup (\{\infty\} \times S_h). 
\]
We obtain $X$ from $X'$ by collapsing $\{\infty\} \times S_h$ in $X'$ to a single basepoint. As shown in \cite[Corollary 3.2.5]{Nemethi}, $S_h$ is contractible for $h$ sufficiently negative, in which case $X$ and $X'$ are $S^1$-homotopy equivalent. Note that we can additionally collapse $\{0\} \times S_h$ to a single point in order to make the fixed-point set a literal copy of $S^0$, which again does not change the homotopy type.
\end{remark}

\begin{figure}[h!]
\includegraphics[scale = 0.8]{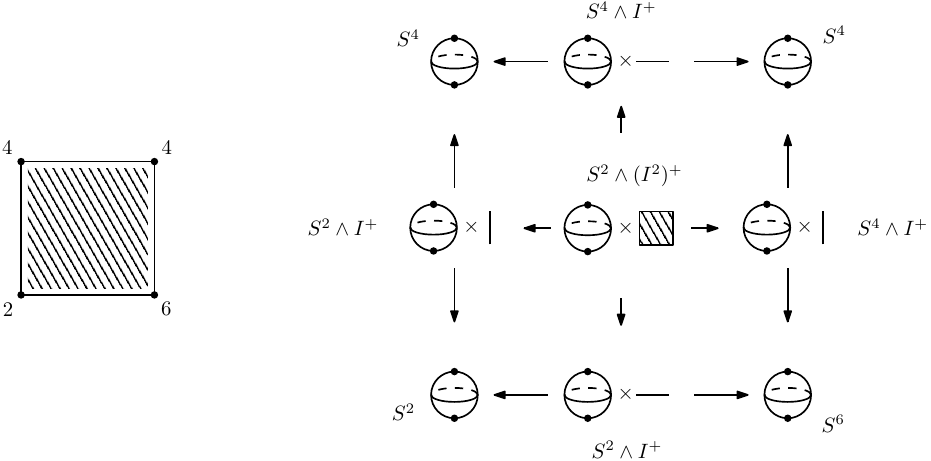}
\caption{A schematic depiction of the construction (\ref{eq:5.3}) restricted to a single square in $[k]$. Left: a square in $[k]$ with example values of $w$ on its vertices. Right: taking $h = 0$, the pointed $S^1$-space constructed according to (\ref{eq:5.3}). Each arrow represents gluing a portion of the boundary of the domain to a subset (or the entirety) of the codomain. For example, the upper edge $S^2 \times I \subset \partial(S^2 \times I^2)$ of the central square is glued homeomorphically to the subset $S^2 \times I \subset S^4 \times I$ along the top row of the diagram.}\label{fig:5.1}
\end{figure}

For $h$ sufficiently negative, the inclusion of $S_h$ into $S_{h-2}$ is a homotopy equivalence, and in fact the latter deformation retracts onto the former \cite[Corollary 3.2.5]{Nemethi}. It follows that up to $S^1$-homotopy equivalence, decreasing $h$ corresponds to suspending $X$.

\begin{definition}\label{def:5.2}
The ($S^1$-)\textit{lattice spectrum} is defined to be
\[
\mathcal{H}(\Gamma, [k]) = (X, 0, -h/2) \in \mathrm{Ob}(\mathfrak{C}_{S^1})
\]
for $h$ sufficiently negative.
\end{definition}

Several theorems from lattice homology have analogous statements in the context of lattice homotopy. In particular, it is possible to upgrade \cite[Proposition 3.4.2]{Nemethi} to show that $\mathcal{H}(\Gamma, [k])$ is an invariant of $(Y_\Gamma, \s)$, rather than $\Gamma$. Although this is not difficult, a full discussion of the proof is rather tedious. Since in this paper we will not use $\mathcal{H}(\Gamma, [k])$, we leave this to the enterprising reader.


\subsection{Homotopy coherent diagrams}\label{sec:homotopycoherent}
It is possible to view the lattice spectrum as a homotopy colimit \cite{bousfield-kan}, which we now explain. Although this will not be used until Section~\ref{sec:connectedsum}, we provide a brief digression here to place the construction of Section~\ref{sec:5.1} in a broader context. 

\begin{definition}\label{def:cubecategory}
Let $S$ be a (finite) cubical subcomplex of $\mathbb{Z}^n$. We think of $S$ as a cubical complex $S_h$ as in Remark~\ref{rem:5.1} and assume it is contractible. The \textit{cube category} associated to $S$ is the category $\mathcal{C}$ whose objects are the lattice cubes $\square_d$ (of any dimension) in $S$. If $\square_{d} \subset \square_{d+i}$ -- that is, if $\square_{d}$ is a face of $\square_{d+i}$ -- then there is a unique morphism $\square_{d+i}\to \square_{d}$; otherwise the morphism set is empty. We write $\phi_{\square_{d+i},\square_d}$ for the unique morphism $\square_{d+i}\to \square_{d}$, when it exists.
\end{definition}

\begin{definition}\label{def:weightfunction}
A \textit{weight function} on a cube category is a function $w \colon \mathrm{Ob}(\mathcal{C}) \rightarrow \mathbb{Z}$ such that $w(\square_{d+i}) \leq w(\square_d)$ whenever $\mathrm{Mor}(\square_{d+i}, \square_d)$ is non-empty. A cube category equipped with a weight function will be referred to a \textit{weighted cube category}. (Since all of our lattice cube categories will be weighted, sometimes we will drop this adjective.)
\end{definition}

We now recall the definition of a homotopy coherent diagram; for a discussion very convenient for our purposes, see \cite[Section 4.2]{LLS}. In \cite{Vogt}, the following definition is stated non-equivariantly; we state it here in the equivariant setting.

\begin{definition}[{\cite[Definition 2.3]{Vogt}}]\label{def:homco}
	Fix a compact Lie group $G$ and let $G$-$\topo$ be the category of pointed $G$-spaces. A \textit{homotopy coherent diagram on $\Cat$} is an assignment $F$ which sends each $x \in \mathrm{Ob}(\Cat)$ to a space $F(x) \in G$-$\topo$ and sends each sequence of composable morphisms 
	\[
	x_0 \xrightarrow{f_1} x_1 \xrightarrow{f_2} \cdots \xrightarrow{f_n} x_n
	\]
	to a continuous map
	\[
	F(f_n,\dots,f_1)\from [0,1]^{n-1} \times F(x_0) \to F(x_n)\phantom{\xrightarrow{f}}
	\]
	with $F(f_n,\dots,f_1)([0,1]^{n-1}\times \{ *\})=*$.  These maps
	are required to satisfy the following compatibility conditions:
	\begin{align}
		\nonumber F(f_n,&\dots,f_1)(t_1, \dots,t_{n-1})=  \\ 
		& \begin{cases}
			F(f_n,\dots,f_2)(t_2,\dots,t_{n-1}) &f_1 = \Ide\\
			F(f_n ,\dots,\hat{f}_i,\dots,f_1)(t_1,\dots,t_{i-1}\cdot t_i,\dots,t_{n-1}),&f_i=\Ide \ (1 < i < n)\\
			F(f_{n-1},\dots,f_1)(t_1,\dots,t_{n-2}) &f_n = \Ide\\
			F(f_n,\dots,f_{i+1})(t_{i+1},\dots, t_{n-1}) \circ F(f_i,\dots,f_1)(t_1,\dots,t_{i-1}) & t_i=0 \ (1 \leq i < n) \\
			F(f_n,\dots,f_{i+1}\circ f_i,\dots,f_1)(t_1,\dots,\hat{t}_{i},\dots,t_{n-1}) & t_i=1 \ (1 \leq i < n).
		\end{cases}\label{eq:compat}
	\end{align}
	When $\Cat$ does not contain any non-identity isomorphisms, homotopy
	coherent diagrams may be defined only in terms of non-identity
	morphisms and the last two compatibility conditions. For $n = 2$, for example, the condition of homotopy coherence says that if $f_1 \colon x_0 \rightarrow x_1$ and $f_2 \colon x_1 \rightarrow x_2$ are composable morphisms, then we have a homotopy
\[
F(f_2, f_1) \colon [0, 1] \times F(x_0) \rightarrow F(x_2)
\]
from $F(f_2) \circ F(f_1)$ to $F(f_2 \circ f_1)$. The reader should think of the structure maps for $n \geq 3$ as providing various higher homotopies.
\end{definition}

Associated to a homotopy coherent diagram $F$, there is a well-defined homotopy colimit, $\hocolim_{C}F$, an object of $G-\topo$; for a description convenient for our purposes, see \cite[Definition 4.10]{LLS} (or \cite[Definition 4.15]{stoffregen-zhang}, for the statement in the equivariant case). To see the connection to our setting, let $\Cat$ be a weighted cube category. Consider the homotopy coherent diagram (with $G = S^1$) on $\Cat$ defined as follows: 
\begin{enumerate}
\item For each object $\square_d$, let $F(\square_d) = S^{(w(\square_d)/2)\C}$.
\item For each morphism $\phi_{\square_{d+i}, \square_d}$, let $F(\phi_{\square_{d+i}, \square_d})$ be the inclusion from $S^{(w(\square_{d+i})/2)\C}$ into $S^{(w(\square_{d})/2)\C}$ induced by the inclusion of $\C^{w(\square_{d+i})/2}$ into $\C^{w(\square_{d})/2}$ along the first $w(\square_{d+i})/2$ coordinates.
\item For each sequence of composable morphisms, define
\[
F(f_n,\dots,f_1)\from [0,1]^{n-1} \times F(x_0) \to F(x_n)
\]
to be equal to the inclusion $F(\phi_{x_0, x_n}) \colon F(x_0) \rightarrow F(x_n)$ (assigned in the previous step), independent of the $[0, 1]^{n-1}$-coordinate(s).
\end{enumerate}
The $S^1$-space $X$ constructed in Section~\ref{sec:5.1} is then easily checked to be the homotopy colimit of the above homotopy coherent diagram in the case where $\Cat$ arises from the superlevel set $S_h$ of Remark~\ref{rem:5.1}. (The weight function $w(\square_d)$ of Definition~\ref{def:weightfunction} is the shifted weight function $w(\square_d) - h$ from Section~\ref{sec:5.1}.)

Note that in this case, the extra homotopy data of Definition~\ref{def:homco} is vacuous: since all morphisms are sent to linear inclusions along the leftmost coordinates of the codomain, $F$ respects composition on the nose and all homotopies are constant. However, one can certainly consider modifying the inclusion maps of Section~\ref{sec:5.1}, or even allowing the inclusion maps to vary depending on $\square_d$. To this end, we define the following notion:

\begin{definition}\label{def:linearcoherent}
Let $\Cat$ be a weighted cube category. We say that a homotopy coherent diagram $F$ on $\Cat$ is \textit{linear} if the following hold: 
\begin{enumerate}
\item For each object $\square_d$, we have $F(\square_d) = S^{(w(\square_d)/2)\C}$.
\item For each morphism $\phi_{\square_{d+i}, \square_d}$, we have that $F(\phi_{\square_{d+i}, \square_d})$ is a map from $S^{(w(\square_{d+i})/2)\C}$ into $S^{(w(\square_{d})/2)\C}$ which is induced by \textit{some} linear inclusion of $\C^{w(\square_{d+i})/2}$ into $\C^{w(\square_{d})/2}$. This linear inclusion does not have to be the inclusion map along the first $w(\square_{d+i})/2$ coordinates; instead, in general it is an element of the Stiefel manifold of complex $w(\square_{d+i})/2$-frames in $\mathbb{C}^{w(\square_{d})/2}$. We denote this Stiefel manifold by
\[
W(\square_{d},\square_{d+i}) = V_{w(\square_{d+i})/2}(\C^{w(\square_d)/2}).
\]
\item Consider any composable sequence of morphisms 
\[
x_0 \xrightarrow{f_1} x_1 \xrightarrow{f_2} \cdots \xrightarrow{f_n} x_n
\]
and let the corresponding structure map be
\[
F(f_n,\dots,f_1)\from [0,1]^{n-1} \times F(x_0) \to F(x_n)
\]
We require that for any fixed $t \in [0, 1]^{n-1}$, the map
\[
F(f_n,\dots,f_1)(t, -) \colon F(x_0) \to F(x_n)
\]
is induced by a linear inclusion of $\C^{w(x_0)/2}$ into $\C^{w(x_n)/2}$; that is, an element of $W(x_n, x_0)$. 
\end{enumerate}
Thus, for a linear homotopy coherent diagram, all of the inclusions -- and also all of the homotopies and higher homotopies -- are required to be induced by linear inclusions.
\end{definition} 

Given a linear homotopy coherent diagram on a weighted cube category, one can form the homotopy colimit to obtain a space much like the one constructed in Section~\ref{sec:5.1}. However, we stress that if $F_1$ and $F_2$ are two such diagrams on the same category, it is not clear that $\hocolim({F_1})$ and $\hocolim({F_2})$ coincide. Roughly speaking, these two spaces will be built out of the same spheres -- one for each lattice cube $\square_d$ -- but the gluing maps will be different. Using the methods of Section~\ref{sec:5.3}, it is not difficult to see that $\hocolim({F_1})$ and $\hocolim({F_2})$ will be $S^1$-spaces with the same (co-)Borel homology. However, we do not claim that they coincide in general, although we establish some special cases in Lemma~\ref{lem:lattice-homotopy-coherent-diagram}. See Section~\ref{sec:connectedsum} for further discussion.

\subsection{Path lattice homotopy}\label{sec:5.2}
We now resume the main discussion and return to the construction of Section~\ref{sec:5.1}. As in Section~\ref{sec:4.3}, it is clear that we may restrict the index set of (\ref{eq:5.3}) to any cubical subcomplex of $[k] \otimes \R$. In particular, let $\gamma$ be a path in $[k]$ and let $h$ be less than or equal to the minimum of $w$ on $\gamma$. Let $X$ be defined as in (\ref{eq:5.3}), but with index set restricted to the zero- and one-dimensional cubes present in $\gamma$. As long as $h$ is less than or equal to the minimum of $w$ on $\gamma$, it is clear that decreasing $h$ corresponds to suspending $X$. We thus obtain a well-defined $S^1$-equivariant spectrum: 




\begin{definition}\label{def:5.3}
If $\gamma$ is any path in $[k]$, the ($S^1$-)\textit{path spectrum} is defined to be
\[
\Hty(\gamma, [k]) = (X, 0, - h/2) \in \mathrm{Ob}(\mathfrak{C}_{S^1}) 
\]
for $h$ sufficiently negative.  The ($S^1$-)\textit{path homotopy type} is the stable homotopy type of $\Hty(\gamma,[k])$.  
\end{definition}

We will often abuse notation, and use $\Hty(\gamma,[k])$ to refer to any of (1) the spectrum of Definition \ref{def:5.3}, (2) a space underlying the spectrum in Definition \ref{def:5.3} or (3) the homotopy type of $\Hty(\gamma,[k])$; context should make clear what is meant. Up to $S^1$-equivariant homotopy, the path spectrum gives a spectrum which is of the same form as described in Section~\ref{sec:2.2}. The spheres of (\ref{eq:2.2}) correspond to the spheres $\mathcal{F}(\square_0)$ of (\ref{eq:5.3}). To make (\ref{eq:5.3}) into (\ref{eq:2.2}), we contract along the $I$-factor in each edge portion $\mathcal{F}(\square_1)$ and also collapse $\{0\} \times \gamma$ to a single point, as described in Remark~\ref{rem:5.1}. 

We now define the zeroth spectrum of an AR plumbed manifold to be the path spectrum associated to any path carrying the lattice homology, as in Theorem~\ref{thm:4.9}:

\begin{definition}\label{def:5.4}
Let $\Gamma$ be an AR plumbing and $[k]$ be an equivalence class in $\Char_\Gamma$. Define the \textit{(zeroth)} ($S^1$-)\textit{lattice spectrum} to be
\[
\Hty_0(\Gamma, [k]) = \Hty(\gamma, [k])
\]
for $\gamma$ any path carrying the lattice homology.
\end{definition}

There are several comments to be made regarding Definition~\ref{def:5.4}. First, it is not entirely obvious that $\Hty_0(\Gamma, [k])$ is independent of the choice of $\gamma$ (or if the equivalence is canonical). This will not actually be important for us: it turns out that the proof of Theorem~\ref{thm:1.1} does not depend on a particular choice of $\gamma$, so in fact Theorem~\ref{thm:1.1} provides an extremely roundabout way of showing that (the homotopy type of) $\Hty_0(\Gamma, [k])$ is independent of $\gamma$. Alternatively, the fact that Definition~\ref{def:5.4} coincides with the construction of Section~\ref{sec:2.2} shows that $\Hty_0(\Gamma, [k])$ depends only on the graded root $\Hla(\gamma, [k])$, rather than $\gamma$. One can prove that all paths $\gamma$ which carry the lattice homology give the same graded root structure up to collapsing degenerate leaves; this shows that $\Hty_0(\Gamma, [k])$ is independent of $\gamma$. 

It is also possible to show that if $\gamma$ is a path carrying the lattice homology, then $\Hty(\gamma, [k]) \simeq \Hty(\Gamma, [k])$. (This again implicitly shows that Definition~\ref{def:5.4} is independent of $\gamma$.) Hence we in fact have
\[
\Hty_0(\Gamma, [k]) = \Hty(\gamma, [k]) = \Hty(\Gamma, [k])
\]
for an AR plumbed manifold. However, for defining the map $\cT$ in the proof of Theorem~\ref{thm:1.1}, it will be critical that $\Hty_0(\Gamma, [k])$ is constructed using a path $\gamma$. Indeed, we do not know how to define $\cT$ on all of $\Hty(\Gamma, [k])$; see Remark~\ref{rem:5.8}. We thus write $\Hty_0(\Gamma, [k])$ rather than $\Hty(\Gamma, [k])$ to emphasize the restricted nature of the construction, even though we will often refer to $\Hty_0(\Gamma, [k])$ as the (full) lattice homotopy type in the case of an AR plumbing.

One can likewise show that $\Hty_0(\Gamma, [k])$ is an invariant of $(Y_\Gamma, \s)$, rather than $\Gamma$. Indeed, the invariance of lattice homology can be upgraded to show that the graded root structure is an invariant of $(Y_\Gamma, \s)$, rather than $\Gamma$; hence $\Hty_0(\Gamma, [k])$ is also an invariant of $(Y_\Gamma, \s)$. Alternatively, one can use the fact that $\Hty_0(\Gamma, [k]) = \Hty(\Gamma, [k])$ and show that the latter is independent of $\Gamma$, as discussed in Section~\ref{sec:5.1}. Since we will not need any of these statements in the current paper, we leave the proofs to the reader. Indeed, our proof that $\Hty_0(\Gamma, [k]) \simeq \SWF(Y_\Gamma , \s)$ does not depend on a choice of $\Gamma$, which implicitly establishes invariance.

Finally, note that in the case of path lattice homotopy, it is clear that changing the inclusion maps $i$ used in (\ref{eq:5.2}) does not alter the $S^1$-homotopy type of the resulting spectrum.



\subsection{Proof of Theorem \ref{thm:1.1}}\label{sec:5.3} 
We now turn to the proof of the main theorem. Let $\Gamma$ be an AR plumbing and $[k]$ be an equivalence class in $\Char_\Gamma$, with $R$ the graded root corresponding to $\Hla(\Gamma, [k])$.

\begin{lemma}\label{lem:5.5}
We have an isomorphism of $\Z[U]$-modules
\[
\cB(\Hty_0(\Gamma, [k])) \cong R \cong \Hla_0(\Gamma, [k]).
\]
\end{lemma}
\begin{proof}
Let $\gamma$ be a path in $[k]$ carrying the lattice homology. Denote the lattice points in $\gamma$ by $k_i$ and let $e_{i-1}$ and $e_i$ be the lattice edges in $\gamma$ which abut $k_i$. Throughout the proof, we use $\mathcal{F}'$ to construct $\Hty(\gamma, [k])$, as in Remark~\ref{rem:5.1}; we adopt notation from this remark as well. 

For each lattice point $k_i \in \gamma$, let $U_i$ be an open neighborhood of $\mathcal{F}'(k_i)$ in $X'$ which extends slightly more than halfway over $\mathcal{F}'(e_{i-1})$ and $\mathcal{F}'(e_i)$:
\[
U_i = \left(S^{((w(e_{i-1})- h)/2)\C} \times (1/2 - \epsilon, 1] \right) \cup S^{((w(k_i) - h)/2)/\C} \cup \left( S^{((w(e_i) - h)/2)\C} \times [0, 1/2 + \epsilon) \right).
\]
Clearly, $U_i$ equivariantly deformation retracts onto $S^{((w(k_i) - h)/2)\C}$, while each double intersection $U_i \cap U_{i+1}$ equivariantly deformation retracts onto $S^{((w(e_i) - h)/2)\C}$. All other double intersections and triple intersections are empty. The generalized Mayer-Vietoris principle then gives an exact sequence:
\[
\cdots \rightarrow \bigoplus_i \cB(U_i \cap U_{i+1}) \rightarrow \bigoplus_i \cB(U_i) \rightarrow \cB(X') \rightarrow \cdots.
\]
Computing the co-Borel homology of each term gives the short exact sequence
\[
0 \rightarrow \bigoplus_i \Z[U]_{w(e_i)-h} \rightarrow \bigoplus_i \Z[U]_{w(k_i)-h} \rightarrow \cB(X') \rightarrow 0.
\]
Here, the zero on the left-hand side is due to the fact that the map out of $\bigoplus_i \Z[U]_{w(e_i)-h}$ is easily seen to be an injection, while the zero on the right-hand side is because $\bigoplus_i \Z[U]_{w(e_i)-h}$ vanishes in odd gradings. This gives a presentation for $\smash{\cB(X')}$; we obtain a presentation for $\smash{\cB(\Hty(\Gamma, [k]))}$ by shifting the grading up by $h$. This is precisely the presentation of $\Hla_0(\Gamma, [k])$ as a graded root.
\end{proof}

We now construct a map from $\Hty_0(\Gamma, [k])$ into $\SWF(Y_\Gamma, \s)$:

\begin{lemma}\label{lem:5.6}
There exists an $S^1$-equivariant map
\[
\cT \colon \Hty_0(\Gamma, [k]) \rightarrow \SWF(Y_\Gamma, \s)
\]
such that
\[
\begin{tikzpicture}[scale=1]
\node (A) at (0,0) {$\cB(\mathcal{H}_0(\Gamma, [k]))$};
\node at (0,-0.75) {\rotatebox{270}{\scalebox{1}[1]{$\cong$}}};
\node (B) at (0, -1.5) {$R$};
\node at (0,-2.25) {\rotatebox{270}{\scalebox{1}[1]{$\cong$}}};
\node (C) at (0,-3) {$\Hla_0(\Gamma, [k])$};
\node (D) at (5,-1.5) {$\cB(\SWF(Y,\s))$};
\path[->,font=\scriptsize,>=angle 90]
(A) edge node[above]{$\cT_*$} (D)
(C) edge node[below]{$\T$} (D);
\end{tikzpicture}
\]
commutes. Here, $\cT_*$ is the map on co-Borel homology induced by $\cT$, the vertical isomorphisms are those of Lemma~\ref{lem:5.5}, and $\T$ is the map on lattice homology from Definition~\ref{def:4.5}.
\end{lemma}
\begin{proof}
As in the proof of Lemma~\ref{lem:5.5}, let $\gamma$ be a path in $[k]$ carrying the lattice homology. Denote the lattice points in $\gamma$ by $k_i$ and let $e_{i-1}$ and $e_i$ be the lattice edges in $\gamma$ which abut $k_i$. For each $k_i$, consider the relative Bauer-Furuta invariant associated to $(W_\Gamma, k_i)$. As discussed in Section~\ref{sec:3.4}, this gives a homotopy class of map
\[
\Psi_{W_\Gamma, k_i} \colon \Sigma^{(\gr(k_i)/2)\C} \SWF(S^3) \rightarrow \SWF(Y, \s)
\]
where $\gr(k_i) = (k_i^2 + |\Gamma|)/4$. Recall that in the suspension $\mathcal{H}_0(\Gamma, [k]) = \Sigma^{(h/2)\C} X$, the lattice point $k_i$ corresponds to a subsphere 
\[
\Sigma^{(h/2)\C}(\mathcal{F}(k_i)) = \Sigma^{(h/2)\C} \Sigma^{((\gr(k_i)-h)/2)\C} S^0 = \Sigma^{(\gr(k_i)/2)\C} S^0.
\]
We thus define $\cT$ on $\Sigma^{(h/2)\C}(\mathcal{F}(k_i)) \subset \mathcal{H}_0(\Gamma, [k])$ to be $\Psi_{W_\Gamma, k_i}$. 

Now let $k_i$ and $k_{i+1}$ be two successive lattice points in $\gamma$, so that $k_{i+1} = k_i + 2v^*$ for some vertex $v$ of $\Gamma$. Let
\[
n = \dfrac{1}{2}(w(k_{i+1}) - w(k_i)) = \dfrac{1}{2}(k_i(v) + v^2).
\]
Suppose $n \geq 0$; that is, $\gr(k_{i+1}) \geq \gr(k_i)$. According to the adjunction relation of Proposition \ref{prop:adjunction}, we have:
\[
\Psi_{W_\Gamma, k_i} \simeq U^n \Psi_{W_\Gamma, k_{i+1}} \simeq \Psi_{W_\Gamma, k_{i+1}} U^n.
\]
Explicitly, this means that $\Psi_{W_\Gamma, k_i}$ is homotopic to the composition
\[
\Sigma^{(\gr(k_i)/2)\C} S^0 \xhookrightarrow{\ \ U^n \ \ } \Sigma^{(\gr(k_{i+1})/2)\C} S^0 \xrightarrow{\Psi_{W_\Gamma, k_{i+1}}} \SWF(Y, \s).
\]
where $U^n$ is the usual inclusion of spheres. Choosing representatives for $\Psi_{W_\Gamma, k_i}$ and $\Psi_{W_\Gamma, k_{i+1}}$, we thus have a homotopy
\[
H_t \colon \Sigma^{(\gr(k_i)/2)\C} S^0 \rightarrow \SWF(Y, \s)
\]
where $H_0 = \Psi_{W_\Gamma, k_i}$ and $H_1$ is the restriction of $\Psi_{W_\Gamma, k_{i+1}}$ to ${\Sigma^{(\gr(k_i)/2)\C} S^0}$. Recall that
\[
\mathcal{F}(e_i) = S^{((w(e_i) - h)/2)\C} \wedge I^+ = \Sigma^{(\gr(k_i)/2)\C} S^0 \wedge I^+,
\]
where we have used the fact that $w(e_i) = \min \{ w(k_i), w(k_{i+1}) \} = w(k_i)$. We thus define $\cT$ on $\mathcal{F}(e_i)$ by setting $\cT(x, t) = H_t(x)$; this descends to a well-defined map on the quotient (\ref{eq:5.3}). The case when $n < 0$ is similar.

The proof of Lemma~\ref{lem:5.5} shows that $\smash{\cB(\Hty_0(\Gamma, [k]))}$ is generated by the induced images (in co-Borel homology) of the inclusions $\mathcal{F}(k_i) \hookrightarrow \Hty_0(\Gamma, [k])$. Under the isomorphism of Lemma~\ref{lem:5.5}, these correspond to the $\Z[U]$-towers which generate $\Hla_0(\Gamma, [k])$. On each of these towers, the maps $\cT_*$ and $\T$ coincide by construction. The establishes the claimed commutative diagram.
\end{proof}

Finally, we have the following:

\begin{lemma}\label{lem:5.7}
Let $f \colon X_1 \rightarrow X_2$ be a map of simply-connected, pointed $S^1$-CW complexes of type SWF. Suppose that:
\begin{enumerate}
\item $f$ induces a homotopy equivalence on $S^1$-fixed point sets, and
\item $f$ induces an isomorphism on Borel (or co-Borel) homology with $\Z$-coefficients:
\[
f_*\colon \tilde{cH}_*^{S^1}(X_1)\to\tilde{cH}^{S^1}_*(X_2).
\]
\end{enumerate}
Then $f$ is an $S^1$-homotopy equivalence.
\end{lemma}
\begin{proof}
We begin by noting that the second assumption is equivalent to the same assumption for \emph{unreduced} Borel homology.   Indeed, writing $G = S^1$, we have that the unreduced and reduced Borel homologies are the second and third terms (respectively) in the long exact sequence
\[
\cdots \to H_*((EG\times *)/G)\to H_*((EG\times X)/G)\to \widetilde{H}_*((EG_+\wedge X)/G)\to\cdots
\]
Clearly, $f$ induces an isomorphism on the first term; hence by the five-lemma it induces an isomorphism on unreduced Borel homology if and only if it induces an isomorphism on reduced Borel homology. We may also freely replace homology with cohomology, since a map induces an isomorphism on homology with ($\Z$-coefficients) if and only if it induces a map on cohomology (with $\Z$-coefficients).

Moreover, as discussed in \cite[Section 5.3]{Manlectures}, there is a natural long exact sequence relating Borel and co-Borel homology:
\[
\cdots \rightarrow \B(X_i) \rightarrow \cB(X_i) \rightarrow \tB(X_i) \rightarrow \cdots,
\]
where $\tB(X_i)$ is the \textit{Tate homology}. Importantly, Tate (co)homology can be computed by localizing the Borel (co)homology of the fixed point set (here we have suppressed grading shifts):
\[
t\widetilde{H}^*_{S^1}(X_i) = U^{-1}\widetilde{H}^*_{S^1}(X_i^{S^1}).
\]
The fact that $f$ induces a homotopy equivalence on fixed-point sets thus implies that $f$ induces an isomorphism on the third term of the above sequence. Hence by the five-lemma, $f$ induces an isomorphism on Borel homology if and only if it induces an isomorphism on co-Borel homology.

To establish the lemma, it suffices to show that $f_* \colon H_*(X_1) \rightarrow H_*(X_2)$ is an isomorphism. 
See Lemma \ref{lem:fixed pt} and Corollary \ref{cor:eq h.e SWF type}. 
For this, consider the Gysin sequence associated to the circle bundle $EG \times X_i \rightarrow (EG \times X_i)/G$:
\[
\cdots \rightarrow H^{*}((EG \times X_i)/G) \xrightarrow{\cup e} H^{*+2}((EG \times X_i)/G) \xrightarrow{\pi^*} H^{*+2}(EG \times X_i) \rightarrow \cdots.
\]
The first two terms of this sequence are none other than the (unreduced) Borel cohomology of $X_i$. Thus $f$ induces an isomorphism on the first two terms; the five-lemma then shows that $f$ induces an isomorphism on $H^*(EG \times X_i)$ (at least for $* \geq 2$, but note that $X_i$ is simply connected). Since $EG$ is contractible, this shows that $f$ induces an isomorphism on $H_*(X_i)$. Applying Corollary~\ref{cor:eq h.e SWF type} completes the proof.
\end{proof}

Putting everything together gives the proof of the main theorem:
\begin{proof}[Proof of Theorem \ref{thm:1.1}:]
We have constructed a map of $S^1$-spectra:
\[
\mathcal{T} \colon \Hty_0(\Gamma, [k]) \rightarrow \SWF(Y_\Gamma, \s).
\]
The commutative diagram of Lemma~\ref{lem:5.6}, combined with the fact that $\T$ is an isomorphism, shows that $\mathcal{T}$ induces an isomorphism on co-Borel homology. We may suspend the domain and the codomain of $\T$ so that both sides are simply connected. It is clear that $\cT$ is a homotopy equivalence of $S^1$-fixed point sets, since both the domain and codomain have $S^1$-fixed point set given by $S^0$. Applying Lemma~\ref{lem:5.7} then shows that $\mathcal{T}$ is a homotopy equivalence, as desired.
\end{proof}

\begin{remark}\label{rem:5.8}
The authors do not know how to define a map $\cT$ out of $\Hty(\Gamma, [k])$ in general. As in Lemma~\ref{lem:5.6}, it is straightforward to define $\cT$ on $\mathcal{F}(\square_d)$ for $d = 0$ or $1$, but it is unclear how to define $\cT$ on $\mathcal{F}(\square_d)$ for $d \geq 2$. Explicitly, consider a lattice square $\square_2$ with boundary edges $e_i$. On each $\mathcal{F}(e_i)$, the map $\cT$ may be defined by using a homotopy coming from Lemma~\ref{lem:5.6}. However, it is then unclear whether $\cT$ can be extended over $\mathcal{F}(\square_2)$. Indeed, note that Lemma~\ref{lem:5.6} does not even specify a preferred homotopy with which to define $\cT$. While for $\Hty(\gamma, [k])$, this does not affect the homotopy class of $\cT$, in general this need not be the case. It is possible that the homotopies of Lemma~\ref{lem:5.6} must be coherently chosen in order for $\cT$ to extend over $\mathcal{F}(\square_2)$.
\end{remark}

\section{Pin(2)-equivariant Results}\label{sec:6}

We now discuss how to upgrade our results to the Pin(2)-setting. As we will see, if $[k]$ is self-conjugate, it is not difficult to put a Pin(2)-action on the lattice spectrum $\Hty_0(\Gamma, [k])$. However, it turns out that the map 
\[
\mathcal{T} \colon \Hty_0(\Gamma, [k]) \rightarrow \SWF(Y_\Gamma, \s)
\]
constructed in Lemma~\ref{lem:5.6} is \textit{not} obviously Pin(2)-equivariant. Hence we cannot immediately generalize the proof of Theorem~\ref{thm:1.1}. We explain this subtlety below.

\subsection{Almost $J$-invariant paths} \label{sec:6.1}
Recall that we have identified the set of $\spinc$-structures on $W_\Gamma$ with $\Char_\Gamma$. Under this correspondence, the action of conjugation is given by the obvious negation map $J \colon k \mapsto -k$. Note that the weight function $w$ (as a function on $\Char_\Gamma$) is evidently $J$-invariant. 

A self-conjugate $\spinc$-structure $[k]$ on $Y_\Gamma$ corresponds to an equivalence class $[k]$ which is sent to itself under $J$. Explicitly, if $[k]$ is self-conjugate, then it is easily checked that there must exist some fixed $w_1, \ldots, w_{|\Gamma|}$, each equal to zero or one, such that $[k]$ is precisely equal to
\[
[k] = \left\{ \sum_{i = 1}^{|\Gamma|} c_i v_i^* \colon c_i \equiv w_i \bmod 2\right\}.
\]
We refer to the $w_i$ as the \textit{Wu coefficients} associated to $[k]$ and the characteristic element $\smash{w = w_1 v_1^* + \cdots + w_{|\Gamma|} v_{|\Gamma|}^*} \in [k]$ as the \textit{Wu element}. For example, if $W_\Gamma$ is spin, then the restriction of this spin structure to $Y_\Gamma$ corresponds to the equivalence class $[k]$ with Wu coefficients $w_1 = \cdots = w_{\Gamma} = 0$.

Assume $[k]$ is self-conjugate. Then $J$ induces an automorphism on the set of lattice cubes in $[k]$, simply by sending a lattice cube $\square_d$ to its reflection $J(\square_d)$. Working over $\F = \Z/2\Z$ so as to disregard signs, the $\F[U]$-linear extension of this automorphism defines a chain map
\[
J \colon \Cla(\Gamma, [k]) \rightarrow \Cla(\Gamma, [k])
\]
which by abuse of notation we also denote by $J$. This involution was studied in \cite{DaiMonopole} and used to compute the Pin(2)-equivariant monopole Floer homology of $Y_\Gamma$. (A similar strategy was carried out in \cite{DaiManolescu} in the setting of involutive Heegaard Floer homology.)  Although we will not spell out the details here, it is possible to use $J$ to construct a $\Pin(2)$-equivariant spectrum whose underlying $S^1$-spectrum is $\Hty(\Gamma,[k])$, although it is not clear to us if the $\Pin(2)$-spectrum defined this way is suitably independent of the choices in its construction.\footnote{The construction of $G$-spectra from $G$-diagrams is treated in Dotto-Moi \cite{dotto-moi};  a very special case of this is also performed in \cite{stoffregen-zhang}. However, that framework does not apply exactly, since the diagram category here is itself sent to equivariant spectra.}    

However, in our context, we will be interested in putting a Pin(2)-structure on the path lattice spectrum $\Hty_0(\Gamma, [k])$, as this is the object which admits a map into $\SWF(Y_\Gamma, [k])$. Unfortunately, $\Hty_0(\Gamma, [k])$ turns out to be somewhat unnatural from the point of view of $J$. To understand why, observe that there is exactly one lattice cube $\square_J$ which is sent to itself by $J$, but the dimension of this cube will be given by the number of non-zero Wu coefficients. Thus, in general $[k]$ may not admit any $J$-invariant path, in which case defining a $\Pin(2)$-structure on $\Hty_0(\Gamma, [k])$ is not entirely natural. As we will see, this leads to an important additional subtlety with the proof of Theorem~\ref{thm:1.2}. See Figure~\ref{fig:6.1}.

\begin{figure}[h!]
\includegraphics[scale = 0.9]{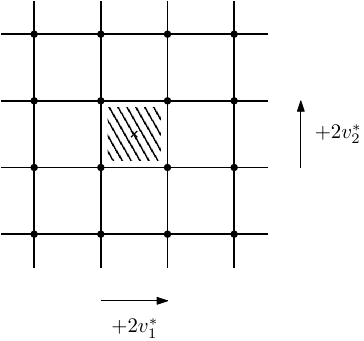}
\caption{An example of the action of $J$ on the affine lattice $[k]$. Here, the action of $J$ is given by reflection through the center of the single $J$-invariant lattice square, which is shaded.}\label{fig:6.1}
\end{figure}

The following lemma is taken from the proof of \cite[Theorem 3.2]{DaiMonopole}:

\begin{lemma}\cite[Theorem 3.2]{DaiMonopole}\label{lem:6.1}
The weight function $w$ takes the same value on all vertices of the $J$-invariant lattice cube $\square_J$.
\end{lemma}

\begin{proof}
In \cite[Lemma 3.1]{DaiMonopole}, it is shown that if $i$ and $j$ are two distinct indices with $w_i$ and $w_j$ both non-zero, then $v_i$ and $v_j$ are not adjacent in $\Gamma$; that is, $(v_i, v_j) = 0$. Now, the vertices of $\square_J$ are precisely the characteristic elements
\[
k = \sum_{i = 1}^{|\Gamma|} \epsilon_i w_i v_i^*
\]
with each $\epsilon_i \in \{-1, +1\}$. Expanding the self-pairing $(k, k)$ and applying the above fact shows that the value of the weight function $w$ on each of these vertices is the same.
\end{proof}

\begin{definition} \label{def:6.2}
We say that a path $\gamma$ is \textit{almost $J$-invariant} if we can decompose $\gamma = \gamma_\Theta \cup \gamma_0$, where:
\begin{enumerate}
\item The path $\gamma_\Theta$ connects two opposite corners $k_\Theta$ and $J(k_\Theta)$ of the invariant lattice cube $\square_J$. This path is contained in $\square_J$ but may not itself be $J$-invariant. Note that the weight function $w$ is constant along $\gamma_\Theta$.
\item The subset $\gamma_0$ consists of a pair of paths that are interchanged by $J$. One of these has an endpoint on $k_\Theta$, while the other has an endpoint on $J(k_\Theta)$; the paths are otherwise disjoint from $\square_J$.
\end{enumerate}
See Figure~\ref{fig:6.2}.
\end{definition}

\begin{figure}[h!]
\includegraphics[scale = 1]{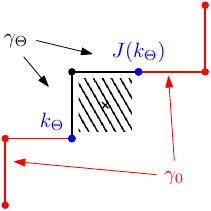}
\caption{An almost $J$-invariant path. Note that both $\gamma_\Theta$ and $\gamma_0$ contain the vertices $k_\Theta$ and $J(k_\Theta)$.}\label{fig:6.2}
\end{figure}

Using Lemma~\ref{lem:6.1}, it is not difficult to modify the algorithm of \cite[Section 7]{NemethiOS} to find a path $\gamma$ in $[k]$ which carries the lattice homology and is almost $J$-invariant.

\subsection{A Pin(2)-structure on $\Hty_0(\Gamma, [k])$} \label{sec:6.2}
Let $\gamma$ be an almost $J$-invariant path carrying the lattice homology. Clearly, we have
\[
\Hty(\gamma, [k]) = \Hty(\gamma_\Theta, [k]) \cup \Hty(\gamma_0, [k]),
\]
glued along $\mathcal{F}(k_\Theta) \vee \mathcal{F}(J(k_\Theta))$, as shown in Figure~\ref{fig:6.3}. We give each of these two pieces a Pin(2)-action as follows:
\begin{enumerate}
\item $\Hty(\gamma_\Theta, [k])$: Choose the cutoff weight $h$ so that $(w(\square_J) - h)$ is $0\bmod{4}$; that is to say $\mathbb{C}^{(w(\square_J)-h)/2}\cong \mathbb{H}^{(w(\square_J)-h)/4}$ so that $S^{((w(\square_J)-h)/2)\mathbb{C}}$ has the structure of a Pin(2)-sphere. It follows from Lemma~\ref{lem:6.1} that $\Hty(\gamma_\Theta, [k])$ is $S^1$-homotopy equivalent to
\[
S^{((w(\square_J) - h)/2)\C} \wedge I^+,
\]
where we view $I = [-1, 1]$. Let $j$ act on $I$ by $-1$, with the action of $S^1$ on $I$ being trivial. We then give $\Hty(\gamma_\Theta, [k])$ the product Pin(2)-structure.

\item $\Hty(\gamma_0, [k])$: First construct half of $\Hty(\gamma_0, [k])$ by applying the usual construction of Section~\ref{sec:5.2} to one of the two connected components of $\gamma_0$. Instead of constructing the other half arbitrarily, we introduce it as the formal image of the first half under $j$. More precisely, any $S^1$-sphere $(\C^n)^+$ may be completed to a pair of $S^1$-spheres
\[
(\C^n)^+ \vee j(\C^n)^+
\]
such that $j$ interchanges the two spheres and is skew-invariant with respect to the action of $S^1$. (That is, $e^{i \theta} j = j e^{-i \theta}$.) Let $\square_0$ and $J(\square_0)$ thus be a pair of lattice points in $\gamma_0$ related by $J$. Instead of introducing $\mathcal{F}(\square_0)$ and $\mathcal{F}(J(\square_0))$ independently, we introduce them simultaneously, so that 
\[
\mathcal{F}(\square_0) \vee \mathcal{F}(J(\square_0)) = S^{((w(\square_0) - h)/2)\C} \vee jS^{((w(\square_0) - h)/2)\C}
\]
as Pin(2)-spaces. A similar construction for the lattice edges and the gluing maps gives $\Hty(\gamma_0, [k])$ a Pin(2)-action. 
\end{enumerate}
We then glue together $\Hty(\gamma_\Theta, [k])$ and $\Hty(\gamma_0, [k])$ in a Pin(2)-equivariant manner to give a Pin(2)-action on all of $\Hty(\gamma, [k])$. We stress that as an $S^1$-space, the above construction is homotopy equivalent to that of $\Hty(\gamma, [k])$ in Section~\ref{sec:5.2}, as all we have done is re-parameterize $\Hty(\gamma_\Theta, [k])$ and chosen particular gluing/inclusion maps. See Figure~\ref{fig:6.3}.
\begin{figure}[h!]
\includegraphics[scale = 0.8]{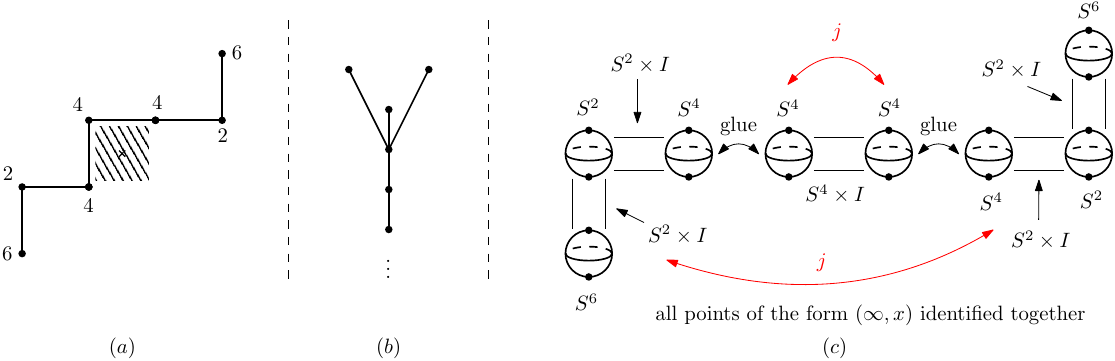}
\caption{An illustration of the $j$-action on $\Hty(\gamma, [k])$. Left: $(a)$ example weights on the vertices of $\gamma$. Middle: $(b)$ the corresponding graded root. Right: $(c)$ the decomposition of $\Hty(\gamma, [k])$ into the pieces $\Hty(\gamma_\Theta, [k])$ and $\Hty(\gamma_0, [k])$; all points of the form $(\infty, x)$ are identified together. On $\Hty(\gamma_0, [k])$, the action of $j$ interchanges the two halves. On $\Hty(\gamma_\Theta, [k])$, the action of $j$ is given by identifying $S^4 = \mathbb{H}^+$ and letting $j$ on $I = [-1, 1]$ be multiplication by $-1$. The product action descends to the quotient $(S^4 \times I)/(\{\infty\} \times I)$.}\label{fig:6.3}
\end{figure}

The following lemma shows that with these choices, the map $\mathcal{T}$ is compatible with the $\Pin(2)$-structure:

\begin{lemma}\label{lem:6.3}
The map $\mathcal{T}$ can be chosen such that
\[
\mathcal{T} \colon \Hty_0(\gamma, [k]) \rightarrow \SWF(Y_\Gamma, [k])
\]
is Pin(2)-equivariant on $\Hty(\gamma_0, [k])$.
\end{lemma}
\begin{proof}
Recall from the proof of Lemma~\ref{lem:5.6} that $\mathcal{T}$ is not canonically defined: we may choose any representative of the relative Bauer-Furuta invariant when defining $\mathcal{T}$ for lattice points, and similarly any homotopy (in the adjunction relation) when defining $\mathcal{T}$ for lattice edges. We use that, for any pair of vertices related by $J$, the associated pair of spheres in $\Hty(\gamma_0,[k])$ is identified with the domain of (\ref{eq:non-spin-bf}).  The construction of Section \ref{subsec:relative-bauer-furuta} then gives a $\Pin(2)$-equivariant map on the vertices of $\Hty(\gamma_0,[k])$.  The homotopies involved in the definition of $\mathcal{T}$ are defined to be $\Pin(2)$-equivariant by defining them first on one half $\gamma_L$ of $\gamma_0$, and then reflecting to define a homotopy on $J(\gamma_L)$.  
\end{proof}
If $W_\Gamma$ is spin and $[k]$ is the restriction of this spin structure to $Y_\Gamma$, then $\gamma_\Theta$ will consist of a single point corresponding to the zero characteristic element. In this case, $\mathcal{T}$ can be made Pin(2)-equivariant on all of $\Hty(\gamma, [k])$, since $\Psi_{W_\Gamma, 0}$ can be made Pin(2)-equivariant. However, in general we only know that $\mathcal{T}$ is an $S^1$-equivariant map on $\Hty(\gamma_\Theta, [k])$.

\subsection{Proof of Theorem \ref{thm:1.2}.} \label{sec:6.3}

We now complete the proof of Theorem~\ref{thm:1.2}. Throughout, let $[X, Y]_G$ denote stable homotopy classes of $G$-maps between two $G$-spectra; for notation on $G$-spectra we refer to \cite[Section 7]{sasahira-stoffregen}.  To consider $G$-spectra, it is necessary to fix a choice of \emph{universe}.  The natural choice of universe for our $\Pin(2)$-spectra is exactly $\mathcal{U}=\mathbb{R}^\infty\oplus \tilde{\mathbb{R}}^\infty\oplus \mathbb{H}^{\infty}$; note that $\mathcal{U}$ defines a $S^1$-universe $\mathcal{U}_{S^1}\cong \mathbb{R}^\infty\oplus \mathbb{C}^\infty$ by restricting along the inclusion $S^1\to \Pin(2)$.  A universe $\mathcal{V}$ is called \emph{complete} if every finite-dimensional orthogonal representation of $G$ embeds in $\mathcal{V}$.  Write $G$-$\mathrm{Sp}$ for the category of $G$-spectra, with the universe suppresed from the notation. 

For $H\subset G$ a closed subgroup, we will also frequently make use of the \emph{geometric $H$-fixed points} functor $\Phi^H\colon G$-$\mathrm{Sp}\to W_GH$-$\mathrm{Sp}$.  Here, $W_GH=N_GH/H$ is the Weyl group of $H$. If $\mathcal{V}$ is the universe associated to the domain of $\Phi^H$, then the universe associated to the codomain is $\mathcal{V}^H$, the $H$-fixed points of the universe $\mathcal{V}$.  In particular, for us the $S^1$-fixed points functor takes the universe $\mathcal{U}$ to a complete $C_2$-universe.  All spectra we consider below come from objects of $\mathfrak{C}_{S^1}$ and $\mathfrak{C}_{\Pin(2)}$; that is, we need only suspension spectra of finite $G$-CW complexes $X$, written as $\Sigma^{\infty}X$. For these, $\Phi^H(\Sigma^{\infty}X)=\Sigma^{\infty}X^H$, so fixed-point spectra are much like fixed-point sets. When there is no risk of confusion, we will sometimes conflate $X$ with its suspension spectrum.

Finally, two objects $(W_0,m_0,n_0)$ and $(W_1,m_1,n_1)$ of $\mathfrak{C}_{\Pin(2)}$ are defined to be equivalent if $n_1-n_0$ is an integer and the spectra 
\[
\Sigma^{m_1\tilde{\mathbb{R}}\oplus (n_0-n_1)\mathbb{H}}W_0 \quad \text{and} \quad \Sigma^{m_0\tilde{\mathbb{R}}}W_1
\]
are equivalent in the category of $\Pin(2)$-spectra over the universe $\mathcal{U}$.

We list the main tools of the proof of Theorem \ref{thm:1.2} here: 
\begin{enumerate}[label=(ho-\arabic*)]
\item The tom Dieck splitting \cite{tom-dieck-splitting, lewis-may-stein-mcclure}. For a very readable description see \cite[Section 4]{lin-mukherjee}:
\begin{theorem}\label{thm:tom-dieck}
Let $X$ be a pointed $G$-space ($G$ a compact Lie group), and $\mathcal{V}$ be a $G$-universe.  Then there is an isomorphism:
\[
\zeta=\sum \zeta^G_{H}\colon \bigoplus_{[H]} \pi^{W_GH}_{k,\mathcal{V}^H}(EW_GH_+\wedge X^H)\to \pi^G_{k,\mathcal{V}}(X)
\]
where the sum runs over conjugacy classes of closed subgroups of $G$ that arise as the isotropy group of some point $p$ in the universe $\mathcal{V}$. Here, the subscripts $\mathcal{V},\mathcal{V}^H$ serve to remind us that the homotopy groups are taken in the category of $G$ (etc.)-spectra over the universes $\mathcal{V}$ (resp. $\mathcal{V}^H$).  
\end{theorem}

We will apply the tom Dieck splitting in the setting $G = C_2$ and $k = 0$. This gives an isomorphism 
\[
\zeta^{C_2}_{e} \oplus \zeta^{C_2}_{C_2} \colon \pi_0^{C_2}((EC_2)_+ \wedge X) \oplus \pi_{0}(X^{C_2}) \rightarrow \pi^{C_2}_0(X).
\]

\item For a cofibration sequence of spectra $X\to Y\to Z$, there are exact sequences $[A,X]\to [A,Y]\to [A,Z]$ and $[Z,B]\to [Y,B]\to [X,B]$; the corresponding statement holds in $G$-spectra as well.  We note that fixed point functors $\Phi^H$ take cofibration sequences of spectra to cofibration sequences, so we obtain corresponding exact sequences on fixed points as well.  (We will usually refer to cofibration sequences of spectra as just \emph{exact sequences} of spectra)
\item \label{item:adjunction} Let $G$ be a compact Lie group and $H$ be a subgroup of $G$. Let $i\colon H \to G$ denote the inclusion and $i^*$ be the change-of-groups functor on spectra. There are natural bijections for $X$ a finite $H$-spectrum and $Y$ a finite $G$-spectrum:
\[
[G_+\wedge_{H} X,Y]_{G} \leftrightarrow [X,i^*Y]_{H}
\]
and
\[
[Y,G_+\wedge_{H} X]_{G} \leftrightarrow [\Sigma^{-L(H)}i^*Y,X]_{H}.
\]
See \cite[Theorem 5.1]{Adams-prerequisites} and \cite[Theorem 5.2]{Adams-prerequisites} for finite groups, or \cite[XVI.4, Theorem 4.9]{alaska} more generally.  Here, $L(H)$ is the tangent space, at the identity, of $G/H$, viewed as a representation of $H$.  

The bijections above show that (for $G$ a compact Lie group) the forgetful functor $i^*$ is right adjoint to $G_+\wedge -$, while $\Sigma^{-L(H)}i^*$ is left adjoint to $G_+\wedge -$. Explicitly, the construction of the above adjunctions are as follows: there is an $H$-map
\[
\alpha \colon X \rightarrow G_+\wedge_H X,
\]
sending $x\to (e,x)$ where $e\in G$ is the identity. The first bijection 
\[
\alpha^*\colon [G_+\wedge_H X,Y]_G \to [X,i^*Y]_H
\]
is obtained by applying the forgetful map $[G_+\wedge_H X,Y]_G\to [G_+\wedge_H X,i^*Y]_{H}$, followed by pullback along $\alpha$ to $[X,i^*Y]_H$. For the second bijection, in the case of finite $G$ there is likewise an $H$-map
\[
\alpha \colon G_+\wedge_H X \rightarrow X,
\]
see the proof of \cite[Theorem 5.2]{Adams-prerequisites}. The bijection 
\[
\alpha_* \colon [Y, G_+ \wedge_H X]_G \rightarrow [i^*Y, X]_H
\]
is similarly obtained by applying the forgetful map and then taking the pushforward along $\alpha$ to $[i^*(Y), X]_H$.  For the construction of the second bijection for general compact Lie groups, see \cite[XVI.4]{alaska}. 
\end{enumerate}

We now introduce the topological objects which will be used in the proof. Up to suspension, we may assume that the $\Pin(2)$-homotopy type of $\Hty(\gamma_\Theta, [k])$ is $S^0$; for notational convenience, let such a suspension be fixed. Let
\[
\Gamma_0 = \Hty(\gamma_0, [k]) \quad \text{and} \quad \Gamma = \Hty(\gamma, [k]).
\]
We remind the reader what these are in Figure~\ref{fig:aid} below. The crucial point is that $\Gamma$ is the union of $\Gamma_0$ and $\Hty(\gamma_\Theta, [k])$, which are comparatively simple: note that $\Gamma_0$ has a free $\Pin(2)$-action away from the basepoint and admits a $\Pin(2)$-map into $\SWF(Y, \s)$, while $\Hty(\gamma_\Theta, [k])$ is $\Pin(2)$-homotopy equivalent to $S^0$.

\begin{figure}[h!]
\includegraphics[scale = 0.65]{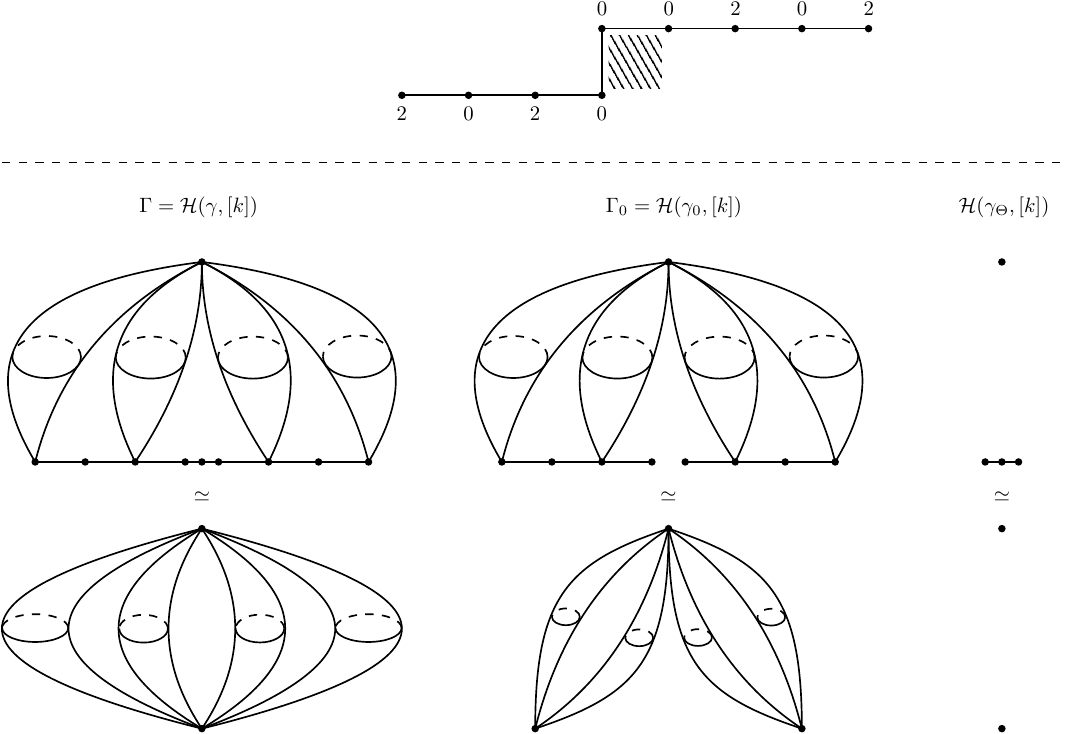}
\caption{Top: a sample almost $J$-invariant path $\gamma$. Middle: $\Gamma$, $\Gamma_0$, and $\Hty(\gamma_\Theta, [k])$. Each of the nine vertices of $\gamma$ contributes either an $S^0$ or an $S^2$ to $\Gamma$. Each sphere has a copy of the origin $0$ and the infinity point $\infty$; all nine $\infty$-points are identified together, while the nine copies of $0$ are displayed along the bottom of $\Gamma$. (See Remark~\ref{rem:5.1}.) Bottom: contracting the line segment along the bottom of $\Gamma$ is a homotopy equivalence.}\label{fig:aid}
\end{figure}

The idea of the proof of Theorem \ref{thm:1.2} is as follows. First observe that by direct inspection of the inclusion $\Gamma_0 \rightarrow \Gamma$, we have 
\[
\Sigma^{\tilde{\R}} S^0 = \text{Cone}(\Gamma_0 \rightarrow \Gamma)
\]
and thus an exact sequence of Pin(2)-spectra:
\begin{equation}\label{eq:exactsequence1}
\cdots \rightarrow \Gamma_0 \rightarrow \Gamma \rightarrow \Sigma^{\tilde{\R}} S^0 \rightarrow \cdots.
\end{equation}
As discussed in Lemma~\ref{lem:6.3}, $\mathcal{T}$ restricts to a Pin(2)-map from $\Gamma_0$ into $\SWF(Y, \s)$. Let
\[
\Theta = \text{Cone}(\Gamma_0 \xlongrightarrow{\mathcal{T}|_{\Gamma_0}} \SWF(Y, \s)).
\]
Note that this cone inherits a Pin(2)-structure. We have an analogous sequence of Pin(2)-spectra:
\begin{equation}\label{eq:exactsequence2}
\cdots \rightarrow \Gamma_0 \xrightarrow{\mathcal{T}|_{\Gamma_0}} \SWF(Y, \s) \rightarrow \Theta \rightarrow \cdots.
\end{equation}
Because $\mathcal{T}$ is an $S^1$-equivalence, (\ref{eq:exactsequence1}) and (\ref{eq:exactsequence2}) are isomorphic if we forget about the Pin(2)-action and consider them as $S^1$-sequences. That is, we have a commutative diagram
\begin{equation}\label{eq:exactsequences}
\cdots
\begin{tikzcd}
	{\Gamma_0} & \Gamma & {\Sigma^{\tilde{\mathbb{R}}} S^0} & {\Sigma^{\R} \Gamma^0} & {\Sigma^{\R} \Gamma}\\
	{\Gamma_0} & {\mathit{SWF}(Y, \mathfrak{s})} & \Theta & {\Sigma^{\R} \Gamma^0} & {\Sigma^{\R} \SWF(Y, \s)}
	\arrow[from=1-2, to=1-3]
	\arrow[from=1-3, to=1-4]
	\arrow[from=1-1, to=1-2]
	\arrow[from=1-4, to=1-5]
	\arrow[from=2-1, to=2-2]
	\arrow["{\mathcal{T}|_{\Gamma_0}}", from=2-2, to=2-3]
	\arrow[from=2-3, to=2-4]
	\arrow["\mathrm{id}",from=1-1, to=2-1]
	\arrow[from=1-3, to=2-3]
	\arrow["\mathrm{id}",from=1-4, to=2-4]
	\arrow["\Sigma^{\R}{\mathcal{T}}",from=1-5, to=2-5]
	\arrow["{\mathcal{T}}", from=1-2, to=2-2]
	\arrow[from=2-4, to=2-5]
\end{tikzcd}
\cdots
\end{equation}
where the rows are exact sequences of Pin(2)-spectra, but the second, third, and fifth vertical maps should \textit{a priori} be interpreted only as equivalences of $S^1$-spectra. 

We focus on the third square
\begin{equation}\label{eq:fundamentalsquare}
\begin{tikzcd}
{\Sigma^{\tilde{\mathbb{R}}}S^0} & {\Sigma^{\R}\Gamma_0} \\
{\Theta} & {\Sigma^{\R}\Gamma_0}
\arrow["\mathrm{id}", from=1-2, to=2-2]
\arrow["f", from=1-1, to=1-2]
\arrow["g", from=2-1, to=2-2]
\arrow["", from=1-1, to=2-1]
\end{tikzcd}
\end{equation}
appearing in $(\ref{eq:exactsequences})$. Now, from (\ref{eq:exactsequences}), we have
\begin{equation}\label{eq:cones}
\Sigma^{\R} \Gamma = \text{Cone}(\Sigma^{\tilde{\R}} S^0 \xrightarrow{f} \Sigma^{\R} \Gamma_0) \quad \text{and} \quad \Sigma^{\R} \SWF(Y, \s) = \text{Cone}(\Theta \xrightarrow{g} \Sigma^{\R} \Gamma_0)
\end{equation}
as Pin(2)-mapping cones. We claim that we can lift the left-hand vertical map in $\ref{eq:fundamentalsquare})$ -- which is \textit{a priori} only an equivalence of $S^1$-spectra -- to an equivalence of $\Pin(2)$-spectra such that $(\ref{eq:fundamentalsquare})$ commutes as a square of $\Pin(2)$-maps. By $(\ref{eq:cones})$, this immediately shows $\Sigma^{\R} \Gamma$ and $\Sigma^{\R} \SWF(Y, \s)$ are $\Pin(2)$-equivalent, and thus that $\Gamma$ and $\SWF(Y, \s)$ are $\Pin(2)$-equivalent, as desired.

The claim of the previous paragraph will follow quickly from a calculation of $\Theta$ as a $\Pin(2)$-spectrum. Note that since the vertical maps of $(\ref{eq:exactsequences})$ are $S^1$-equivalences, we have
\[
\Theta \simeq_{S^1} \smash{\Sigma^{\tilde{\R}} S^0}
\]
as $S^1$-spectra. We sometimes still write $\smash{\Sigma^{\tilde{\R}} S^0}$, even though as an $S^1$-spectrum this is just $\smash{\Sigma^{\R} S^0}$. In Lemma~\ref{lem:pin2calculationtheta} below, it is shown that
\[
\smash{\Theta \simeq_{\Pin(2)} \Sigma^{\tilde{\R}} S^0}.
\]
Given this, the remainder of the proof of Theorem~\ref{thm:1.2} is relatively short:

\begin{proof}[Proof of Theorem~\ref{thm:1.2}]
By Lemma~\ref{lem:pin2calculationtheta}, $\smash{\Theta\simeq_{\Pin(2)} \Sigma^{\tilde{\mathbb{R}}}S^0}$ as $\mathrm{Pin}(2)$-spectra. To ask whether the left-hand vertical map in (\ref{eq:fundamentalsquare}) lifts, we thus consider the map
\[
[\Sigma^{\tilde{\R}} S^0, \Theta]_{\Pin(2)} =  [S^0, S^0]_{\Pin(2)} \rightarrow [\Sigma^{\tilde{\R}} S^0, \Theta]_{S^1} = [S^0, S^0]_{S^1}.
\]
Now, it is not difficult to check that
\[
[S^0,S^0]_{S^1} = \pi_0^{S_1}(S^0) = \mathbb{Z}.
\]
Indeed, in general $\pi_0^G(S^0)$ is isomorphic to the Burnside ring $A(G)$; for the case of $G$ a compact Lie group, see \cite[Theorem 8.5.1]{tom_dieck_transformation}. In the case of a compact Lie group, $A(G)$ has a slightly subtle definition -- as discussed in \cite[Proposition 5.5.1]{tom_dieck_transformation}, $A(G)$ is additively isomorphic to the free abelian group generated by conjugacy classes of subgroups $H$ such that $N(H)/H$ is finite. Since the only such subgroup of $S^1$ is $S^1$ itself, we easily obtain the above computation. It follows immediately that: 
\begin{enumerate} 
\item Every element of $[S^0,S^0]_{S^1}$ lifts to an element of $[S^0,S^0]_{\Pin(2)}$ (since $1 \in [S^0, S^0]_{S^1}$ lifts and this generates $[S^0, S^0]_{S^1}$); and,
\item There are exactly two equivalences $\pm 1 \in [S^0,S^0]_{S^1}$. These clearly lift to equivalences $\pm 1 \in [S^0,S^0]_{\Pin(2)}$.
\end{enumerate}

We moreover claim that if $h$ is any $S^1$-map on the left-hand side of $(\ref{eq:fundamentalsquare})$ which makes $(\ref{eq:fundamentalsquare})$ homotopy commute as a square of $S^1$-maps, then any $\Pin(2)$-lift $\smash{\tilde{h}}$ of $h$ will automatically make $(\ref{eq:fundamentalsquare})$ commute as a square of $\Pin(2)$-maps. Roughly speaking, the reason for this is as follows. Recall that $\Gamma_0$ -- and thus $\Sigma^{\R}\Gamma_0$ -- has a free $\Pin(2)$-action away from the basepoint. The adjunction relation \ref{item:adjunction} then shows that the horizontal maps $f$ and $g$ in $(\ref{eq:fundamentalsquare})$ are determined by their behavior as $S^1$-maps. Thus, if $(\ref{eq:fundamentalsquare})$ commutes as a square of $S^1$-maps, then any $\Pin(2)$-lift will automatically commute. The precise version of this argument is given below.

Write
\[
\Gamma_0 = \Pin(2)_+\wedge_{S^1} \mathcal{H}(\gamma_L, [k])
\]
where $\gamma_L$ is one of the two components of $\gamma_0$. Suppose $h \colon \Sigma^{\tilde{\R}} S^0 \rightarrow \Theta$ is an $S^1$-map on the left-hand side of $(\ref{eq:fundamentalsquare})$ which makes $(\ref{eq:fundamentalsquare})$ commute as a square of $S^1$-maps, and let $\smash{\tilde{h}}$ be any $\Pin(2)$-lift. Consider the commutative diagram
\begin{equation}
\begin{tikzcd}\label{eq:M-square}
	{[\Theta, \Sigma^{\R} \Gamma_0]_{\Pin(2)}} & {[i^*(\Theta), \Sigma^{\R} \mathcal{H}(\gamma_L, [k])]_{S^1}} \\
	{[\Sigma^{\tilde{\R}} S^0, \Sigma^{\R} \Gamma_0]_{\Pin(2)}} & {[i^*(\Sigma^{\tilde{\R}} S^0), \Sigma^{\R} \mathcal{H}(\gamma_L, [k])]_{S^1}} 
	\arrow["{h^*}", from=1-2, to=2-2]
	\arrow["{\alpha_*}", from=1-1, to=1-2]
	\arrow["{\alpha_*}", from=2-1, to=2-2]
	\arrow["{\tilde{h}^*}", from=1-1, to=2-1].
\end{tikzcd}
\end{equation}
Here, the horizontal rows are given by \ref{item:adjunction}; the fact that the diagram commutes is proven easily from the definition of $\alpha_*$. For the maps $\smash{f \in [\Sigma^{\tilde{\R}} S^0, \Sigma^{\R} \Gamma_0]_{\Pin(2)}}$ and $\smash{g \in [\Theta, \Sigma^{\R} \Gamma_0]_{\Pin(2)}}$ of $(\ref{eq:fundamentalsquare})$, we have assumed 
\[
i^*(f) = i^*(g) \circ h
\]
where $i^*(f)$ and $i^*(g)$ are the maps $f$ and $g$ considered as $S^1$-morphisms. We claim
\[
f  = g \circ \tilde{h}.
\]
Indeed, observe that by definition of $\alpha_*$, we have
\[
\alpha_*(f) = \alpha \circ i^*(f) = \alpha \circ i^*(g) \circ h = h^*(\alpha_*(g)).
\]
Because $\alpha_*$ is a bijection, commutativity of (\ref{eq:M-square}) shows $f = \smash{\tilde{h}}^*(g)$; that is, $f = g \circ \smash{\tilde{h}}$. Hence $(\ref{eq:fundamentalsquare})$ commutes as a square of $\Pin(2)$-maps.

Putting everything together, we know that there exists a $\Pin(2)$-equivalence which lifts the original $S^1$-equivalence along the left-hand vertical column of $(\ref{eq:fundamentalsquare})$. With this lift, $(\ref{eq:fundamentalsquare})$ still commutes as a square of $\Pin(2)$-maps. It follows that the $\Pin(2)$-mapping cones of $f$ and $g$ are equivalent. By $(\ref{eq:cones})$, this completes the proof of Theorem \ref{thm:1.2}.
\end{proof}

The remainder of this subsection is devoted to proving Lemma~\ref{lem:pin2calculationtheta} and carrying out the computation of $\Theta$ as a $\Pin(2)$-spectrum. We begin by producing a candidate map
\[
\eta \colon \Theta \rightarrow \Sigma^{\tilde{\R}} S^0.
\]
This is done by first mapping $\Theta$ into some large $\Pin(2)$-sphere, and then successively factoring this map through smaller and smaller $\Pin(2)$-spheres. As we will see in our proof of Lemma~\ref{lem:pin2calculationtheta} it will also be useful to control the behavior of $\eta$ on $\Pin(2)$-fixed point spectra.
\begin{lemma}\label{lem:6.6}
For $p$ and $q$ sufficiently large, there exists a Pin(2)-map
\[
\eta' \colon \Theta \rightarrow \Sigma^{q \tilde{\R} + p \mathbb{H}} S^0.
\]
This may be taken to be an equivalence on $\Pin(2)$-fixed point spectra.
\end{lemma}
\begin{proof}
Recall the exact sequence (\ref{eq:exactsequence2}) of $\Pin(2)$-spectra
\[
\cdots \rightarrow \Gamma_0 \rightarrow \SWF(Y, \s) \xrightarrow{m} \Theta \rightarrow \Sigma^{\R} \Gamma_0 \cdots
\]
where we have denoted the map $\SWF(Y, \s) \rightarrow \Theta$ by $m$. Taking $\Phi^{\Pin(2)}$, we obtain an exact sequence
\[
\cdots \rightarrow \Phi^{\Pin(2)} \Gamma_0 \rightarrow \Phi^{\Pin(2)}\SWF(Y, \s) \xrightarrow{\Phi^{\Pin(2)}m} \Phi^{\Pin(2)}\Theta \rightarrow \Phi^{\Pin(2)} \Sigma^{\R} \Gamma_0 \rightarrow \cdots
\]
Take maps from these exact sequences into $\smash{\Sigma^{q\tilde{\mathbb{R}}+p\mathbb{H}}S^0}$ and $\Phi^{\Pin(2)} \smash{\Sigma^{q\tilde{\mathbb{R}}+p\mathbb{H}}S^0} = S^0$, respectively. We focus on the commutative square
\begin{equation}\label{eq:mapisomorphism}
\begin{tikzcd}
	{[\Theta,\Sigma^{q\tilde{\mathbb{R}} + p\mathbb{H}}S^0]_{\mathrm{Pin}(2)}} & {[\mathit{SWF}(Y,\mathfrak{s}),\Sigma^{q\tilde{\mathbb{R}} + p\mathbb{H}}S^0]_{\mathrm{Pin}(2)}} \\
	{[\Phi^{\Pin(2)} \Theta, S^0]} & {[\Phi^{\Pin(2)}\SWF(Y, \s),S^0]}
	\arrow["m^*", from=1-1, to=1-2]
	\arrow["{\Phi^{\Pin(2)}m^*}",from=2-1, to=2-2]
	\arrow["{\Phi^{\Pin(2)}}"', from=1-1, to=2-1]
	\arrow["{\Phi^{\Pin(2)}}", from=1-2, to=2-2]
\end{tikzcd}
\end{equation}
where the bottom row is obtained from the top row by taking $\Pin(2)$-fixed points. Now, using \cite[Proposition 4.2]{Adams-prerequisites}, it follows from the cell description of $\Gamma_0$ that 
\[
[\Gamma_0, \Sigma^{q\tilde{\mathbb{R}}+p\mathbb{H}}S^0]_{\Pin(2)}=0
\]
for $p$ and $q$ sufficiently large, and similarly for $\Sigma^{\mathbb{R}}\Gamma_0$. Thus the top row of $(\ref{eq:mapisomorphism})$ is a bijection for such $p$ and $q$. Note also that since $\Gamma_0$ has no $\Pin(2)$-fixed points -- other than the basepoint -- we have that $\Phi^{\Pin(2)} \Gamma_0$ and $\Phi^{\Pin(2)} \Sigma^{\R} \Gamma_0$ are trivial, and so $\smash{\Phi^{\Pin(2)}m}$ is an equivalence.

Now, for $p$ and $q$ sufficiently large, there exists a $\Pin(2)$-map 
\[
\eta'' \colon \SWF(Y, \s) \to \Sigma^{q\tilde{\mathbb{R}}+p\mathbb{H}} S^0
\]
which is an equivalence on $\Pin(2)$-fixed points. This may be deduced from the fact that $Y$ bounds some spin $4$-manifold, which has a Bauer-Furuta map of this form. The map $\eta''$ is an element in the top-right corner of (\ref{eq:mapisomorphism}); let $\eta'$ be the corresponding element of the top-left corner. By commutativity of (\ref{eq:mapisomorphism}), we have
\[
(\Phi^{\Pin(2)} \eta') \circ (\Phi^{\Pin(2)} m) = \Phi^{\Pin(2)} \eta''.
\]
Since $\Phi^{\Pin(2)} m$ and $\Phi^{\Pin(2)} \eta''$ are equivalences, we have $\Phi^{\Pin(2)} \eta'$ is an equivalence, as desired.
\end{proof}

We now show that the map $\eta'$ of Lemma~\ref{lem:6.6} can be factored through successively smaller and smaller $\Pin(2)$-spheres to obtain a map $\smash{\eta \colon \Theta \rightarrow \Sigma^{\tilde{\R}} S^0}$. Denote by 
\[
c \colon S^0 \rightarrow \Sigma^{\tilde{\R}} S^0 \quad \text{and} \quad V \colon S^0 \rightarrow \Sigma^{\mathbb{H}} S^0
\]
the standard $\Pin(2)$-inclusions.

\begin{lemma}\label{lem:6.7}
Let $q \geq 2$. Any $\Pin(2)$-map 
\[
\eta' \colon \Theta \rightarrow \Sigma^{q \tilde{\R} + p \mathbb{H}} S^0
\]
can be factored as $\eta' = c^{q-1} V^p \eta$ for some $\Pin(2)$-map
\[
\eta \colon \Theta \rightarrow \Sigma^{\tilde{\R}} S^0.
\]
If $\eta'$ induces an equivalence on $\Pin(2)$-fixed point spectra, then so does $\eta$.
\end{lemma}
\begin{proof}
Consider the exact sequence
\begin{equation}\label{eq:c-exact}
S^0\xlongrightarrow{c} \Sigma^{\tilde{\R}} S^0 \to \Sigma^{\R} (C_2)_+
\end{equation}
of $\Pin(2)$-spectra. Suspend this sequence by $\tilde{\R}^{q-1}$ and take maps from $\Theta$ into the result. We claim that
\begin{equation}\label{eq:c-transfer}
[\Theta,\Sigma^{(q-1)\tilde{\mathbb{R}}}S^0]_{\Pin(2)}\xlongrightarrow{c_*}[\Theta,\Sigma^{q\tilde{\mathbb{R}}}S^0]_{\Pin(2)}
\end{equation}
is surjective for $q\geq 2$.  Indeed, observe that the third term in the exact sequences vanishes, since by \ref{item:adjunction} we have
\begin{align*}
[\Theta, \Sigma^{(q-1) \tilde{\mathbb{R}}} (\Sigma^{\R}(C_2)_+)]_{\Pin(2)}=[i^*(\Theta), \Sigma^{(q-1) \tilde{\R}} (\Sigma^{\R} S^0)]_{S^1} = [\Sigma^{\R} S^0, \Sigma^{q\R} S^0]_{S^1}
\end{align*}
and this is zero whenever $q \geq 2$. This shows that $\eta'$ can be lifted through $c^{q-1}$. 

To lift $\eta'$ through $V^p$, consider the analogous exact sequence
\begin{equation}\label{eq:V-exact}
S^0\xlongrightarrow{V} \Sigma^{\mathbb{H}} S^0 \to \Sigma^{\R} (S(\mathbb{H})_+),
\end{equation}
where $S(\mathbb{H})$ is the unit sphere in $\mathbb{H}$. Suspend this sequence by $\mathbb{H}^{p-1}$ and take maps from $\Theta$ into the result. We likewise claim that
\begin{equation}\label{eq:V-transfer}
[\Theta,\Sigma^{(p-1)\mathbb{H}}S^0]_{\Pin(2)}\xlongrightarrow{V_*}[\Theta,\Sigma^{p\mathbb{H}}S^0]_{\Pin(2)}
\end{equation}
is surjective for $p \geq 1$. For this, it suffices to prove
\begin{equation}\label{eq:p-arg}
[\Theta,\Sigma^{(p-1)\mathbb{H}}\Sigma^{\mathbb{R}}(S(\mathbb{H})_+)]_{\mathrm{Pin}(2)}=0
\end{equation}
whenever $p \geq 1$. We establish this computation by repeatedly applying \ref{item:adjunction} to show that any morphism as in (\ref{eq:p-arg}) can be homotoped to have image in successively smaller and smaller skeleta of the codomain. Consider the standard $\mathrm{Pin}(2)$-equivariant CW structure on $S(\mathbb{H})$ having one cell each in degrees zero, one, and two. Write $X_i$ for the equivariant $i$-skeleton of $S(\mathbb{H})_+$, so that $X_2=S(\mathbb{H})_+$.  Fix any morphism as in (\ref{eq:p-arg}). Following this by the (suspension of the) map $S(\mathbb{H})_+\to X_2/X_1=\Sigma^{2\mathbb{R}}(\mathrm{Pin}(2)_+)$ gives an element of 
\[
[\Theta,\Sigma^{(p-1)\mathbb{H}}\Sigma^{3\mathbb{R}}(\mathrm{Pin}(2)_+)]_{\mathrm{Pin}(2)}.
\]
Now apply \ref{item:adjunction} with $G = \Pin(2)$ and $H = \{e\}$. Noting that $L(H) = \R$ in this case, we have that the above set is in bijection with
\[
[\Sigma^{-\R} \Theta,\Sigma^{(p-1)\mathbb{H}}\Sigma^{3\mathbb{R}} S^0]=[S^0,\Sigma^{(p-1)\mathbb{H}}\Sigma^{3\mathbb{R}} S^0],
\]
where we have used that $\Theta = \Sigma^{\R} S^0$ as a non-equivariant spectrum. (As discussed previously, $\Theta = \Sigma^{\R}S^0$ as $S^1$-equivariant spectra.) This is clearly zero, so a morphism as in (\ref{eq:p-arg}) must factor through $\Sigma^{(p-1)\mathbb{H}}\Sigma^{\mathbb{R}}X_1$. Repeating the above argument using the map $X_1 \mapsto X_1/X_0 = \Sigma^{\R} (\mathrm{Pin}(2)_+)$ shows that our morphism likewise factors through $\Sigma^{(p-1)\mathbb{H}}\Sigma^{\mathbb{R}}X_0$. Since $X_0 = \Pin(2)_+$, this is just an element of
\[
[\Theta,\Sigma^{(p-1)\mathbb{H}}\Sigma^{\mathbb{R}}\mathrm{Pin}(2)_+]_{\mathrm{Pin}(2)}.
\]
Applying \ref{item:adjunction} once last time, we have that the above set is in bijection with
\[
[\Sigma^{-\R} \Theta,\Sigma^{(p-1)\mathbb{H}}\Sigma^{\R} S^0] = [S^0,\Sigma^{(p-1)\mathbb{H}}\Sigma^{\R} S^0],
\]
which is again zero. This establishes (\ref{eq:p-arg}) and hence that (\ref{eq:V-transfer}) is surjective, as desired.

We have thus shown that $\eta'$ can be factored as $\eta' = c^{q-1} V^p \eta$. If $\eta'$ is an equivalence on $\Pin(2)$-fixed point spectra, then clearly $\eta$ is also, since $c$ and $V$ are equivalences of $\Pin(2)$-fixed point spectra.
\end{proof}

Lemmas~\ref{lem:6.6} and \ref{lem:6.7} give a $\Pin(2)$-map $\eta \colon \Theta \rightarrow \Sigma^{\R} S^0$ which is an equivalence on $\Pin(2)$-fixed points. We will need one more ingredient for the proof of Lemma~\ref{lem:pin2calculationtheta}, which is that $\eta$ can be taken to be an equivalence on $S^1$-fixed points. For this, we use the following technical computation regarding the $S^1$-fixed point spectrum of $\Theta$:

\begin{lemma}\label{lem:6.5}
We have
\[
\Phi^{S^1}\Theta \simeq_{C_2} \Phi^{S^1} \Sigma^{\tilde{\R}}S^0 = \Sigma^{\tilde{\R}}S^0.
\]
\end{lemma}
\begin{proof}
The argument is similar to the proof of Theorem~\ref{thm:1.2}. Applying the $S^1$-fixed point functor to $(\ref{eq:exactsequence1})$ and $(\ref{eq:exactsequence2})$ gives the exact sequences of $C_2$-spectra
\[
\cdots \rightarrow \Phi^{S^1}\Gamma_0 \rightarrow \Phi^{S^1}\Gamma \rightarrow \Phi^{S^1}\Sigma^{\tilde{\R}} S^0 \rightarrow \cdots
\]
and
\[
\cdots \rightarrow \Phi^{S^1}\Gamma_0 \rightarrow \Phi^{S^1}\SWF(Y, \s) \rightarrow \Phi^{S^1}\Theta \rightarrow \cdots.
\]
These fit together into the commutative diagram obtained by applying $\Phi^{S^1}$ to (\ref{eq:exactsequences}). However, note that the second, third, and fifth columns obtained by doing so are \textit{a priori} maps of non-equivariant spectra. Consider the leftmost commutative square
\begin{equation}\label{eq:exactfixed}
\begin{tikzcd}
	{\Phi^{S^1}\Gamma_0} & {\Phi^{S^1}\Gamma} \\
	{\Phi^{S^1}\Gamma_0} & {\Phi^{S^1}\mathit{SWF}(Y, \mathfrak{s})} 
	\arrow["f", from=1-1, to=1-2]
	\arrow["g", from=2-1, to=2-2]
	\arrow[from=1-2, to=2-2]
	\arrow["\mathrm{id}", from=1-1, to=2-1].
\end{tikzcd}
\end{equation}
Here, the two horizontal rows are maps of $C_2$-spectra, whereas the right-hand vertical map should initially only be interpreted as a non-equivariant equivalence. Now,
\[
\Phi^{S^1}\Sigma^{\tilde{\R}}S^0 = \text{Cone}(\Phi^{S^1}\Gamma_0 \xrightarrow{f} \Phi^{S^1}\Gamma) \quad \text{and} \quad \Phi^{S^1}\Theta = \text{Cone}(\Phi^{S^1}\Gamma_0 \xrightarrow{g} \Phi^{S^1}\SWF(Y, \s)).
\]
As in the proof of Theorem~\ref{thm:1.2}, it thus suffices to show that the right-hand vertical map can be lifted to a $C_2$-equivalence such that $(\ref{eq:exactfixed})$ commutes as a square of $C_2$-maps.

By direct inspection, we have
\[
\Phi^{S^1}\Gamma \simeq_{C_2} S^0 \quad \text{and} \quad \Phi^{S^1}\SWF(Y, \s) \simeq_{C_2} S^0.
\]
To ask whether the right-hand vertical map in (\ref{eq:exactfixed}) lifts, we thus consider the map
\[
[\Phi^{S^1}\Gamma, \Phi^{S^1}\SWF(Y, \s)]_{C_2} = [S^0, S^0]_{C_2} \rightarrow [\Phi^{S^1}\Gamma, \Phi^{S^1}\SWF(Y, \s)] = [S^0, S^0].
\]
Since $[S^0, S^0] \cong \Z$, we immediately see that: 
\begin{enumerate} 
\item Every element of $[S^0,S^0]$ lifts to an element of $[S^0,S^0]_{C_2}$ (since $1 \in [S^0, S^0]$ lifts and this generates $[S^0, S^0]$); and,
\item There are exactly two equivalences $\pm 1 \in [S^0,S^0]$. These clearly lift to equivalences $\pm 1 \in [S^0,S^0]_{C_2}$.
\end{enumerate}

We moreover claim that if $h$ is any non-equivariant map on the right-hand side of (\ref{eq:exactfixed}) which makes (\ref{eq:exactfixed}) commute as a square of non-equivariant maps, then any $C_2$-lift $\smash{\tilde{h}}$ of $h$ will automatically make (\ref{eq:exactfixed}) commute as a square of $C_2$-maps. For this, consider the diagram
\begin{equation}\label{eq:commutative1}
\begin{tikzcd}
	{[\Phi^{S^1}\Gamma_0, \Phi^{S^1}\Gamma]_{C_2}} & {[S^0, i^*(\Phi^{S^1}\Gamma)]} \\
	{[\Phi^{S^1}\Gamma_0, \Phi^{S^1}\mathit{SWF}(Y, \mathfrak{s})]_{C_2}} & {[S^0, i^*(\Phi^{S^1}\mathit{SWF}(Y,\mathfrak{s}))]}
	\arrow["{h_*}", from=1-2, to=2-2]
	\arrow["{\alpha^*}", from=1-1, to=1-2]
	\arrow["{\alpha^*}", from=2-1, to=2-2]
	\arrow["{\tilde{h}_*}",from=1-1, to=2-1]
\end{tikzcd}
\end{equation}
Here, the horizontal rows are given by \ref{item:adjunction} using the fact that $\Phi^{S_1}\Gamma_0 = (C_2)_+$ as $C_2$-spectra. The fact that the diagram commutes is easily proven using the definition of $\alpha^*$. Let $h$ and $\smash{\tilde{h}}$ be as described. For the maps $\smash{f \in [\Phi^{S^1}\Gamma_0, \Phi^{S^1}\Gamma]_{C_2}}$ and $\smash{g \in [\Phi^{S^1}\Gamma_0, \Phi^{S^1}\SWF(Y, \s)]_{C_2}}$ of (\ref{eq:exactfixed}), we have supposed that
\[
i^*(g) = h\circ i^*(f) 
\]
where $i^*(f)$ and $i^*(g)$ are the maps $f$ and $g$ considered as non-equivariant morphisms. We claim
\[
g = \tilde{h} \circ f. 
\]
Indeed, observe that by definition of $\alpha^*$, we have
\[
\alpha^*(g) = i^*(g) \circ \alpha = h \circ i^*(f) \circ \alpha = h_*(\alpha^*(f)).
\]
Because $\alpha^*$ is a bijection, the commutativity of (\ref{eq:commutative1}) shows $g = \tilde{h}_*(f)$; that is, $g = \tilde{h} \circ f$. 

We thus see that there exists a $C_2$-equivalence which lifts the original non-equivariant equivalence along the right-hand vertical column of (\ref{eq:exactfixed}) With this lift, (\ref{eq:exactfixed}) still commutes as a square of $C_2$-maps. This completes the proof.
\end{proof}

We now use Lemma~\ref{lem:6.5} to establish the following:

\begin{lemma}\label{lem:thetacharacterization}
There exists a $\Pin(2)$-map
\[
\eta \colon \Theta \rightarrow \Sigma^{\tilde{\R}} S^0
\]
which is an equivalence on $\Pin(2)$- and $S^1$-fixed point sets.
\end{lemma}
\begin{proof}
Using Lemmas~\ref{lem:6.6} and \ref{lem:6.7}, we already have a $\Pin(2)$-map $\eta$ that induces an equivalence on $\Pin(2)$-fixed point sets. Recall that $\eta$ is constructed by factoring the map $\eta'$ of Lemma~\ref{lem:6.6} through smaller and smaller $\Pin(2)$-spheres. Observe that there is an ambiguity in the construction of $\eta$: in the case $q = 2$, suspending (\ref{eq:c-exact}) gives
\[
\cdots \to \Sigma^{\tilde{\R}} (C_2)_+ \to \Sigma^{\tilde{\R}}  S^0\xlongrightarrow{c} \Sigma^{2\tilde{\R}} S^0 \to \Sigma^{\tilde{\R}} (\Sigma^{\R} (C_2)_+) \to \cdots.
\] 
Take maps from $\Theta$ into this sequence. We showed in the proof of Lemma~\ref{lem:6.7} that the fourth term $\smash{[\Theta, \Sigma^{\tilde{\R}} (\Sigma^{\R} (C_2)_+)]_{\Pin(2)} = 0}$, so that $\smash{c_* \colon [\Theta, \Sigma^{\tilde{\R}}S^0]_{\Pin(2)} \rightarrow [\Theta, \Sigma^{2\tilde{\R}} S^0]_{\Pin(2)}}$ is surjective. However, note that
\[
[\Theta, \Sigma^{\tilde{\R}} (C_2)_+]_{\Pin(2)} = [i^*(\Theta), \Sigma^{\R} S^0]_{S^1} = [\Sigma^{\R} S^0, \Sigma^{\R} S^0]_{S^1},
\]
which is not zero. Hence lifting through $c$ in the case $q = 2$ has an ambiguity given by the image of the connecting morphism:
\[
[\Theta,\Sigma^{\tilde{\R}} (C_2)_+]_{\Pin(2)} \rightarrow [\Theta, \Sigma^{\tilde{\R}} S^0]_{\Pin(2)}.
\]
We claim that we can use this ambiguity to specify the behavior of $\eta$ on the $S^1$-fixed point spectrum. (There is a similar ambiguity from the exact sequence for $V$ in the case $p = 1$, although we will not use this here.)

To investigate the behavior of $\eta$ on $S^1$-fixed points, consider the commutative diagram
\begin{equation}\label{eq:tie-to-s1-fix}
\begin{tikzcd}
	{[\Theta,\Sigma^{\tilde{\R}} (C_{2})_+]_{\Pin(2)}} & {[\Theta,\Sigma^{\tilde{\mathbb{R}}} S^0]_{\Pin(2)}} \\
	{[\Phi^{S^1}\Theta,\Sigma^{\tilde{\R}} (C_2)_+]_{C_2}} & {[\Phi^{S^1}\Theta,\Sigma^{\tilde{\mathbb{R}}} S^0]_{C_2}}
	\arrow[from=1-1, to=1-2]
	\arrow[from=2-1, to=2-2]
	\arrow["{\Phi^{S^1}}"', from=1-1, to=2-1]
	\arrow["{\Phi^{S^1}}", from=1-2, to=2-2],
\end{tikzcd}
\end{equation}
where the first row is the aforementioned connecting homomorphism and the bottom row is its counterpart after taking $S^1$-fixed points. The top-left corner of this square is given by
\[
[\Theta, \Sigma^{\tilde{\R}}(C_2)_+]_{\Pin(2)} = [i^*(\Theta), \Sigma^{\tilde{\R}}S^0]_{S^1} = [\Sigma^{\tilde{\R}}S^0, \Sigma^{\tilde{\R}}S^0]_{S^1} = [S^0, S^0]_{S^1}.
\]
Here, the first equality follows from \ref{item:adjunction} and the second equality follows from the calculation of $\Theta$ as an $S^1$-spectrum. The bottom-left corner of the square is given by
\[
[\Phi^{S^1}\Theta,\Sigma^{\tilde{\R}} (C_2)_+]_{C_2} = [i^*(\Phi^{S^1}\Theta), \Sigma^{\tilde{\R}}S^0] = [\Sigma^{\tilde{\R}} S^0, \Sigma^{\tilde{\R}} S^0] = [S^0, S^0].
\]
Here, the first equality again follows from \ref{item:adjunction} and the second equality follows from the calculation of $\smash{\Phi^{S^1} \Theta}$ as a non-equivariant spectrum (which is a consequence of Lemma~\ref{lem:6.5}). Finally, the bottom-right corner of the square is given by
\[
[\Phi^{S^1}\Theta,\Sigma^{\tilde{\R}} S^0]_{C_2} = [\Sigma^{\tilde{\R}} S^0, \Sigma^{\tilde{\R}} S^0]_{C_2} = [S^0, S^0]_{C_2},
\]
with the first equality following from Lemma~\ref{lem:6.5}.

We now observe the following:

\begin{enumerate}
\item The left-hand vertical map in (\ref{eq:tie-to-s1-fix}) is a bijection. This corresponds to the map
\[
[S^0, S^0]_{S^1} \rightarrow [S^0, S^0].
\]
Note that we know both of these groups are $\Z$ (as discussed in the proof of Theorem~\ref{thm:1.2}) and hence generated by $1$. The claim follows immediately.
\item To understand the bottom map in (\ref{eq:tie-to-s1-fix}), we calculate the bottom-left and bottom-right corners of (\ref{eq:tie-to-s1-fix}) slightly differently. Applying the tom Dieck splitting to the bottom-left corner of (\ref{eq:tie-to-s1-fix}), we obtain
\begin{align*}
[\Phi^{S^1}\Theta,\Sigma^{\tilde{\R}} (C_2)_+]_{C_2} &= [S^0, (C_2)_+]_{C_2} = \pi_0^{C_2}((EC_2)_+ \wedge (C_2)_+) \oplus \pi_0^{\{e\}}(\{\text{pt}\}) \\
&= \pi_0((EC_2)_+) \oplus \{e\} \cong \Z.
\end{align*}
Applying the tom Dieck splitting to the bottom-right corner of (\ref{eq:tie-to-s1-fix}), we obtain
\begin{align*}
[\Phi^{S^1}\Theta,\Sigma^{\tilde{\mathbb{R}}} S^0]_{C_2} &= [S^0, S^0]_{C_2} = \pi_0^{C_2}((EC_2)_+\wedge S^0) \oplus \pi_0^{\{e\}}(S^0) \\
&= \pi_0((BC_2)_+) \oplus \pi_0(S^0) = \Z\langle w, 1 \rangle.
\end{align*}
Since the tom Dieck splitting is functorial, the bottom row of (\ref{eq:tie-to-s1-fix}) is identified with
\[
\pi_0^{C_2}((EC_2)_+\wedge (C_2)_+) = \pi_0((EC_2)_+) \to \pi_0^{C_2}((EC_2)_+\wedge S^0) = \pi_0((BC_2)_+).
\]
This is the same as a map (coming from the natural map of nonequivariant spaces) $\pi_0((EC_2)_+)\to \pi_0((BC_2)_+)$. The map $(EC_2)_+\to (BC_2)_+$ sends the non-basepoint component of $(EC_2)_+$ to the non-basepoint component of $(BC_2)_+$, and so is an isomorphism on $\pi_0$. In particular, under the above identifications the bottom map of (\ref{eq:tie-to-s1-fix}) sends $\Z$ isomorphically onto $\Z\langle w \rangle \subset \Z\langle w, 1 \rangle$.
\end{enumerate}
Let us now take our map $\smash{\eta \in [\Theta, \Sigma^{\tilde{\R}} S^0]_{\Pin(2)}}$, which lies in the top-right corner of (\ref{eq:tie-to-s1-fix}) and consider $\smash{\Phi^{S^1} \eta \in \Z\langle w, 1 \rangle}$.  Say $\Phi^{S^1}(\eta)=aw+b\in A(C_2)=\mathbb{Z}\langle w,1\rangle$. Since $\Phi^{S^1}(\eta)$ is an equivalence on $C_2$-fixed points, we have $b=\pm 1$.  

By our description of the bottom row of (\ref{eq:tie-to-s1-fix}), the image of the bottom-left corner in the bottom-right corner is $\mathbb{Z}\langle w\rangle$.    Thus, there is some $\omega$ in the bottom-left corner of (\ref{eq:tie-to-s1-fix}) so that $\smash{\Phi^{S^1} \eta}+\omega\in \mathbb{Z}\langle 1\rangle \subset A(C_2)$; since $\Phi^{S^1}(\eta)=aw\pm 1$, we have $\smash{\Phi^{S^1}\eta}+\omega=\pm 1$, so $\Phi^{S^1}(\eta)+\omega$ is an equivalence.  Since the left-hand vertical map of (\ref{eq:tie-to-s1-fix}) is a bijection, we may take an element $\tilde{\omega}$ in the top-left corner such that $\smash{\Phi^{S^1} \tilde{\omega} = \omega}$. If we change $\eta$ by the image of $\tilde{\omega}$, then by commutativity of (\ref{eq:tie-to-s1-fix}), we have that $\smash{\Phi^{S^1} \eta}$ is an equivalence. This completes the proof.
\end{proof}

Putting everything together, we have:

\begin{lemma}\label{lem:pin2calculationtheta}
We have a $\Pin(2)$-equivalence
\[
\smash{\Theta \simeq_{\Pin(2)} \Sigma^{\tilde{\R}} S^0}.
\]
\end{lemma}
\begin{proof}
By Lemma~\ref{lem:thetacharacterization}, we have a $\Pin(2)$-map $\smash{\eta \colon \Theta \rightarrow \Sigma^{\tilde{\R}} S^0}$ which is an equivalence on $\Pin(2)$- and $S^1$-fixed point spectra. The fact that $\eta$ is an equivalence on $S^1$-fixed point spectra implies that it induces an isomorphism on Tate homology. Since $\smash{\Theta \simeq_{S^1} \Sigma^{\tilde{\R}}} S^0$, the $S^1$-Borel homology of $\Theta$ consists of a single tower. An $S^1$-map from $\Theta$ to $\smash{\Sigma^{\tilde{\R}}S^0}$ which induces an isomorphism on Tate homology thus induces an isomorphism on Borel homology. By the same argument as in the proof of Lemma~\ref{lem:5.7}, we conclude that $\eta$ induces a (non-equivariant, stable) equivalence from $\Theta$ to $\smash{\Sigma^{\tilde{\R}}S^0}$. The equivariant Whitehead theorem then shows that $\eta$ is a Pin(2)-equivalence.
\end{proof}


\section{Computations} \label{sec:7}

We now give several consequences of Theorems~\ref{thm:1.1} and \ref{thm:1.2}. First, we prove that for AR homology spheres, the relation of local equivalence in the category of spectra \cite[Definition 2.7]{Stoffregen} coincides with the relation of local equivalence defined for chain complexes; see \cite[Definition 2.19]{Stoffregen} and \cite[Definition 8.5]{HM}. Following Remark~\ref{rem:2.1}, we also show that the construction of the lattice spectrum depends only on the homology of the graded root $R$, rather than its combinatorial structure. Finally, we establish Corollaries~\ref{cor:a} through \ref{cor:d}. Throughout this section, we write $G = \Pin(2)$.

\subsection{Graded roots and local maps} 

The goal of this section will be to show that various notions of local equivalence (e.g.\ for spectra, chain complexes, and graded roots) coincide for AR homology spheres. It will be helpful for the reader to be broadly familiar with the discussion of \cite{DaiManolescu} regarding symmetric graded roots. We review this briefly, adapting the original discussion over $\F = \Z/2\Z$ to $\Z$-coefficients.

\begin{definition}
Let $R$ be a graded root as in Section~\ref{sec:2.1}. We say that $R$ is \textit{symmetric} if it is invariant under reflection about a vertical axis; we denote the resulting involution by $J$. Note that in \cite{DaiManolescu} (and also Section~\ref{sec:6} above) we have used $\F$-coefficients, but it will be convenient for us to use $\Z$-coefficients in this subsection. Thus, $J$ is a $\Z[U]$-equivariant involution on (the $\Z[U]$-module corresponding to) $R$. It should be noted that if $x_i$ is a leaf of $R$, then $x_i - Jx_i$ is $U$-torsion, rather than $x_i + Jx_i$. 
\end{definition}

\begin{definition} \label{def:gradedrootmap}
Let $R_1$ and $R_2$ be two graded roots. A \textit{graded root map} is a grading-preserving, $\Z[U]$-equivariant map from the $\Z[U]$-module represented by $R_1$ to the $\Z[U]$-module represented by $R_2$. If $R_1$ and $R_2$ are symmetric, then we require $F$ to intertwine these symmetries. We denote such a map by
\[
F \colon R_1 \rightarrow R_2.
\]
Our terminology is slightly misleading, in that a graded root map need not preserve the combinatorial structure of the graded root and is simply a map between the corresponding $\Z[U]$-modules. We say that a graded root map is \textit{local} if it induces an isomorphism after inverting $U$. We say that $R_1$ and $R_2$ are \textit{locally equivalent} if there exist local maps in both directions.
\end{definition}

In \cite[Section 4]{DaiManolescu}, it is shown that every symmetric graded root $R$ is locally equivalent to a particularly simple subroot, which we call the \textit{monotone subroot of $R$}. The proof is sketched pictorially in Figure~\ref{fig:MSubroot}: the local map in one direction is obtained by collapsing $R$ onto its monotone subroot via pushing the (non-subroot) leaves of $R$ horizontally towards the central stem, while the map in the other direction is given by inclusion. These are clearly $J$-equivariant local maps. While the proof of \cite[Theorem 6.1]{DaiManolescu} gives this argument over $\F$, it is clear that both maps may be defined over $\Z$ in precisely the same manner.

\begin{figure}[h!]
\includegraphics[scale = 0.7]{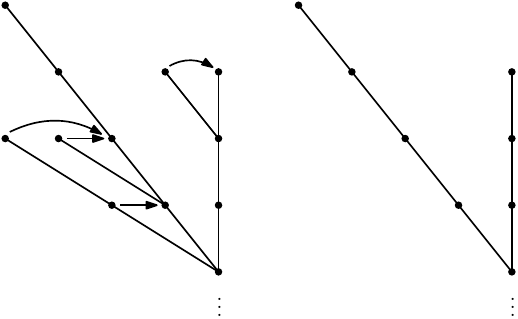}
\caption{Left: the left half of a symmetric graded root. Right: its monotone subroot. The arrows on the left schematically depict the local map which collapses the graded root onto its monotone subroot.}\label{fig:MSubroot}
\end{figure}

Although we will exclusively work with the homology associated to a graded root, it may be conceptually helpful to note that \cite{DaiManolescu} actually deals with chain complexes over $\f[U]$, rather than their homologies. More precisely, in \cite[Section 4]{DaiManolescu} the authors construct a particularly convenient chain complex associated to a graded root $R$, called the \textit{standard complex}. The standard complex associated to $R$ is then locally equivalent (as defined in \cite[Definition 8.5]{HM}) to the standard complex associated to its monotone subroot. While \cite{DaiManolescu} deals with $\F$-coefficients, it is again possible to lift this result to $\Z$-coefficients, although care must be taken with signs when defining the standard complex. Explicitly, if $R$ is a symmetric graded root, then we construct its standard complex over $\Z[U]$ as follows:

\begin{enumerate}
\item We introduce a generator $x_i$ (over $\Z[U]$) for each leaf of $R$, treating the uppermost $J$-invariant vertex of $R$ as a degenerate leaf.
\item We introduce a generator $\alpha_i$ (over $\Z[U]$) for each upwards-opening angle of $R$, as in Figure~\ref{fig:SComplex}. Furthermore, if there is a $J$-invariant angle, we subdivide it into two equal angles, as indicated in Figure~\ref{fig:SComplex}.
\item For every angle to the left of the axis of symmetry, we set 
\[
\partial \alpha_i = U^{(\gr(x_i) - \gr(\alpha_i))/2} x_i - U^{(\gr(x_{i+1}) - \gr(\alpha_i))/2} x_{i+1},
\]
while every angle to the right of the axis of symmetry, we set
\[
\partial \alpha_i = - U^{(\gr(x_i) - \gr(\alpha_i))/2} x_i + U^{(\gr(x_{i+1}) - \gr(\alpha_i))/2} x_{i+1}.
\]
\end{enumerate}

\begin{figure}[h!]
\includegraphics[scale = 0.9]{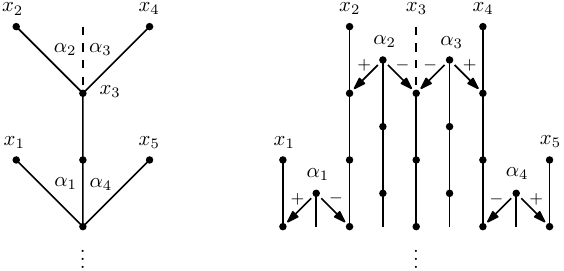}
\caption{Left: a symmetric graded root. Right: the generators of the standard complex. Arrows represent the differential, labeled with sign.}\label{fig:SComplex}
\end{figure}

We stress that the choice of signs in the differential is not analogous to the sign assignment in Definition~\ref{def:4.3}: the sign assignment here differs depending on whether $\alpha_i$ is to the left or right of the axis of symmetry. Due to the choice of signs in the differential, the involution $J$ which maps each generator of the standard complex to its corresponding reflection is a $U$-equivariant chain map. We emphasize that we are specifically working with graded roots, and that we do not claim to construct an $\Z[U]$-equivariant involution $J$ on the full lattice complex with $\Z$-coefficients.  Clearly, this standard complex reduces modulo two to that of \cite[Section 4]{DaiManolescu}. It is straightforward to check that \cite[Theorem 6.1]{DaiManolescu} can be generalized to this setting.

Note that our construction of the lattice spectrum $\Hty(\Gamma, [k])$ factors through the graded root associated to $\Hla(\Gamma, [k])$. In general, given an arbitrary graded root $R$, we denote by $\Hty(R)$ the resulting $S^1$-lattice homotopy type, with the understanding that if $R$ is symmetric then $\Hty(R)$ is a Pin(2)-homotopy type. The following fundamental lemma shows that any graded root map may be lifted to a map of spectra. If the original map is symmetric, then the result may be taken to be $\Pin(2)$-morphism under a mild grading hypothesis. 

\begin{lemma}\label{lem:graded-root-switch}
Let $R_1$ and $R_2$ be two graded roots and suppose that we have a graded root map $F \colon R_1 \rightarrow R_2$. Then there exists a $S^1$-equivariant morphism
\[
\mathcal{F} \colon \Hty(R_1) \rightarrow \Hty(R_2)
\]
which induces the map $F$ on ($S^1$-equivariant) co-Borel homology. Moreover, suppose $R_1$ and $R_2$ are symmetric graded roots and $F$ is equivariant. If the uppermost $J$-invariant vertices of $R_1$ and $R_2$ lie in the same grading $\bmod$ $4$, then $\mathcal{F}$ may be taken to be a $\Pin(2)$-morphism.  
\end{lemma}
\begin{proof}
We handle the $S^1$-case first. Recall from Section~\ref{sec:2.1} that $R_1$ is generated (over $\Z[U]$) by the leaves $\{x_i\}_{i = 1}^n$, subject to the relations
\begin{equation}\label{eq:F-relations}
U^{(\gr(x_i)-\gr(\alpha_i))/2}x_i \sim U^{(\gr(x_{i+1})-\gr(\alpha_i))/2}x_{i+1}
\end{equation}
ranging over the angles $\{\alpha_i\}_{i = 1}^{n-1}$. To construct $\Hty(R_1)$, for each leaf $x_i$ we introduce an $S^1$-sphere $S_{x_i}$, while for each angle $\alpha_i$ we introduce a cylinder $S_{\alpha_i} \times I$ connecting a subsphere of $S_{x_i}$ to a subsphere of $S_{x_{i+1}}$. The copies of $S_{\alpha_i}$ in $S_{x_i}$ and $S_{x_{i+1}}$ are canonically identified with (de)suspensions of $S_{x_{i}}$ and $S_{x_{i+1}}$, respectively. To define a map $\Hty(R_1) \rightarrow \Hty(R_2)$, it thus suffices to construct maps $\varphi_i\colon S_{x_i}\to \Hty(R_2)$ for each $i$, together with homotopies
 \begin{equation}\label{eq:homotopy-between-vertices}
 U^{(\gr(x_i)-\gr(\alpha_i))/2}\varphi_i \simeq U^{(\gr(x_{i+1})-\gr(\alpha_i))}\varphi_{i+1}
 \end{equation}
for each successive pair of spheres. Here, in (\ref{eq:homotopy-between-vertices}) recall that $U\varphi$ is the restriction 
 \[
 [\Sigma^{n\C} S^0, Y]_{S^1} \xlongrightarrow{U} [\Sigma^{n\C} S^0, \Sigma^{\C} Y]_{S^1} =  [\Sigma^{(n-1)\C}S^0, Y]_{S^1}
 \]
 as discussed in Section~\ref{subsec:adjunct}.
 
Denote the leaves of $R_2$ by $z_j$. Given $F$ as in the lemma statement, write
\[
F(x_i) = \sum_j c_{ij} z_j
\]
with each $\smash{c_{ij}  = U^{(\gr(z_j) - \gr(x_i))/2} n_{ij}}\in \mathbb{Z}[U]$. This expression will not in general be unique since there are $\Z[U]$-relations among the $z_j$; here we fix an arbitrary such representation of $F(x_i)$. It is straightforward to show, by an argument analogous to that for Lemma \ref{lem:6.7}, that 
\[
[S^0, \Sigma^{n\C} S^0]_{S^1} \cong
\begin{cases}
\Z & \text { if } n \geq 0 \\
0 & \text{ else}
\end{cases}
\]
and that 
\[
U \colon [S^0, \Sigma^{n\C} S^0]_{S^1} \rightarrow [S^0, \Sigma^{(n+1)\C}]_{S^1}
\]
is an isomorphism for $ n \geq 0$. Thus we may identify $\oplus_{n \geq 0} [S^0, \Sigma^{n\C}S^0]_{S^1}$ with $\Z[U]$, where $1$ corresponds to a generator of $[S^0, S^0]_{S^1}$. In particular, $c_{ij}$ corresponds to an element of
\[
[S^0, \Sigma^{((\gr(z_j)-\gr(x_i))/2) \C} S^0]_{S^1} = [S_{x_i}, S_{z_j}]_{S^1}.
\]
Since the stable homotopy category is additive, we may sum together all of the maps for a fixed $i$ (as $j$ varies) to obtain a map 
\[
\bar{\varphi}_i \in [S_{x_i}, \vee_j S_{z_j}]_{S^1}.
\]
We claim that (\ref{eq:homotopy-between-vertices}) holds for $\bar{\varphi}_i$ and $\bar{\varphi}_{i+1}$. Indeed,
\[
[X,\vee_j Y_j]=\bigoplus_j \ [X,Y_j], 
\]
so it suffices to compare $\bar{\varphi}_i$ and $\bar{\varphi}_{i+1}$ componentwise. Then (\ref{eq:F-relations}), combined with our identification of $\oplus_{n \geq 0} [S^0, \Sigma^{n\C}S^0]_{S^1}$ with $\Z[U]$, tautologically implies (\ref{eq:homotopy-between-vertices}). Hence the $\bar{\varphi}_i$ glue together to define a map on all of $\Hty(R_1)$. We define $\mathcal{F}$ by postcomposing this with the map from the wedge product into $\Hty(R_2)$:
\[
\mathcal{F} \colon \Hty(R_1) \rightarrow \vee_j S_{z_j} \rightarrow \Hty(R_2).
\]
The fact that the wedge product of the $S_{z_j}$ admits a map to $\Hty(R_2)$ can be seen from Remark~\ref{rem:5.1}. Clearly, $\mathcal{F}$ induces the map $F$ on co-Borel homology. 

Now suppose that $R_1$ and $R_2$ are symmetric and $F$ is equivariant.\footnote{In fact, one can show that if any homology isomorphism between two symmetric graded roots exists, then an equivariant isomorphism exists; see for example the proof of \cite[Theorem 3.16]{Stoffregen}.} Let $x_J$ be the uppermost $J$-invariant vertex of $R_1$ and $z_J$ be the uppermost $J$-invariant vertex of $R_2$. We regard these as degenerate leaves, if they are not leaves already. We define $\mathcal{F}$ as follows. On each sphere $S_{x_i}$ and cylinder $S_{\alpha_i} \times I$ in the left half of $R_1$, we construct $\mathcal{F}$ exactly as before, as an $S^1$-map. Using the adjunction relation \ref{item:adjunction} in Section~\ref{sec:6.3}, this allows us to define a $\Pin(2)$-map on all of $\Hty(R_1)$ except the central sphere $S_{x_J}$. Explicitly, we define $\mathcal{F}$ on the right half of $\Hty(R_1)$ by $\smash{j^{-1} \circ \mathcal{F} \circ j}$, so that $\mathcal{F}$ is tautologically $\Pin(2)$-equivariant except possibly on $S_{x_J}$. The hypothesis that $F$ is equivariant ensures that the restriction of $\mathcal{F}$ to the right half of $\Hty(R_1)$ also agrees, on co-Borel homology, with $F$. 

It remains to show that the map out of the central sphere $S_{x_J}$ constructed in the $S^1$-case can be lifted to a $\Pin(2)$-map. Write $F(x_J)$ as a linear combination of leaves in $R_2$. Since $F$ is equivariant, this linear combination can be chosen such that it consists of the sum of symmetric pairs of vertices, together with a single term supported by $z_J$. Recall that our $S^1$-map out of $S_{x_J}$ is the sum of many constituent maps in the stable homotopy category. Each time a symmetric pair $U^{\ell}(z_j + Jz_j)$ appears in $F(x_J)$, we group their corresponding $S^1$-maps together to ensure that they combine to form a $\Pin(2)$-map. This is done in the same way as before, by defining the summand of $\mathcal{F}$ from $S_{x_J}$ to $S_{Jz_j}$ as the conjugate of the summand from $S_{x_J}$ to $S_{z_j}$. 

Finally, we have the summand of $\mathcal{F}$ corresponding to a term $\smash{U^{(\gr(z_J) - \gr(x_J))/2} z_J}$ in $\smash{F(x_J)}$. By hypothesis, the exponent of $U$ is even. Thus it is not difficult to promote our $S^1$-map from $S_{x_J}$ to $S_{z_J}$ to a $\Pin(2)$-map of $\Pin(2)$-spheres. Indeed, by our construction of the $\Pin(2)$-action in Section~\ref{sec:6.2}, we have that as $\Pin(2)$-spheres
\[
S_{x_J} = \Sigma^{(\gr(x_J)/4) \mathbb{H}} S^0 \quad \text{and} \quad S_{z_J} = \Sigma^{(\gr(z_J)/4) \mathbb{H}} S^0.
\]
Thus there is a clear $\Pin(2)$-map
\[
S_{x_J} \rightarrow \Sigma^{((\gr(z_J) - \gr(x_J))/4)\mathbb{H}} S_{z_J}
\]
which lifts the $S^1$-map constructed earlier. This completes the definition of $\mathcal{F}$.  
\end{proof}

It is straightforward to use Lemma~\ref{lem:graded-root-switch} to reconcile the different notions of local equivalence arising in this context. Explicitly, let $(Y_1, \s_1)$ and $(Y_2, \s_2)$ be two AR homology spheres equipped with self-conjugate $\spinc$-structures. There are several different settings in which to consider local equivalence:

\begin{enumerate}
\item \textit{(spectrum) local equivalence} between the spectra associated to $(Y_1, \s_1)$ and $(Y_2, \s_2)$, as discussed in \cite[Definition 2.7]{Stoffregen}; \label{item:spectrumle}
\item \textit{(chain) local equivalence} between the $\SWF$-type chain complexes associated to $(Y_1, \s_1)$ and $(Y_2, \s_2)$, as discussed in \cite[Definition 2.19]{Stoffregen}; \label{item:chainle}
\item \textit{(involutive Heegaard Floer) local equivalence} between the $\iota$-complexes associated to $(Y_1, \s_1)$ and $(Y_2, \s_2)$, as discussed in \cite[Definition 8.5]{HM}; and, \label{item:involutivele}
\item \textit{(graded root) local equivalence} between the symmetric graded roots associated to the lattice homologies of $(Y_1, \s_1)$ and $(Y_2, \s_2)$. (This may be considered over $\Z[U]$ or $\F[U]$.) \label{item:gradedle}
\end{enumerate}

The first of these notions is defined in terms of spectra, while the second and third are defined in terms of chain complexes. In principle, spectrum local equivalence is stronger than chain local equivalence, while involutive Heegaard Floer local equivalence should be thought of as analogous to chain local equivalence (albeit defined in the involutive Heegaard Floer setting).\footnote{Conjecturally, involutive Heegaard Floer homology corresponds to $\Z/4\Z$-equivariant Seiberg-Witten Floer homology. If this were true, then (\ref{item:involutivele}) would be weaker than (\ref{item:chainle}).} We defined the fourth notion in Definition~\ref{def:gradedrootmap}.

In our context, each graded root is the Heegaard Floer or Seiberg-Witten Floer homology of an AR homology sphere. Thus, the difference between (for example) (\ref{item:involutivele}) and (\ref{item:gradedle}) simply consists of passing from a chain complex to its homology and remembering the action of $j$ or $\iota$ on the result. In principle, this means that (\ref{item:gradedle}) is the weakest possible notion of local equivalence. However, the arguments of \cite[Section 6]{DaiManolescu} show that (in the case of a graded root) (\ref{item:involutivele}) and (\ref{item:gradedle}) are the same, at least over $\F[U]$. 


\begin{theorem}\label{thm:lesame}
Let $(Y_1, \s_1)$ and $(Y_2, \s_2)$ be two AR homology spheres equipped with self-conjugate $\spinc$-structures. If $(Y_1, \s_1)$ and $(Y_2, \s_2)$ are locally equivalent according to any one of (\ref{item:spectrumle}) through (\ref{item:gradedle}), then they are locally equivalent according all of them.
\end{theorem}
\begin{proof}
We first show that graded root local equivalence over $\Z[U]$ and $\F[U]$ are the same. Clearly, the former implies the latter. Conversely, our discussion at the beginning of the subsection shows that if a graded root is locally equivalent to its monotone subroot over $\F[U]$, then it is locally equivalent to its monotone subroot over $\Z[U]$. Since two graded roots are locally equivalent over $\F[U]$ if and only if they have the same monotone subroot, this proves the claim.

We thus already have that (\ref{item:involutivele}) and (\ref{item:gradedle}) are equivalent and that (\ref{item:spectrumle}) implies (\ref{item:chainle}). It is not hard to see that (\ref{item:chainle}) implies (\ref{item:gradedle}) by passing to homology. Hence it suffices to show (\ref{item:gradedle}) implies (\ref{item:spectrumle}). This follows immediately from Lemma~\ref{lem:graded-root-switch}: if $F_1$ and $F_2$ are local equivalences between the symmetric graded roots $R_1$ and $R_2$ associated to $(Y_1, \s_1)$ and $(Y_2, \s_2)$, then we obtain Pin(2)-morphisms $\mathcal{F}_1$ and $\mathcal{F}_2$ between the lattice homotopy types $\Hty(R_1)$ and $\Hty(R_2)$. The fact that $F_1$ and $F_2$ are local implies $\mathcal{F}_1$ and $\mathcal{F}_2$ induce isomorphisms on Tate homology, which easily shows that $\mathcal{F}_1$ and $\mathcal{F}_2$ are local. 
\end{proof}

A similar argument establishes Remark~\ref{rem:2.1}. Let $R_1$ and $R_2$ be two graded roots (equipped with planar embeddings, as discussed in Remark~\ref{rem:2.1}). Then:

\begin{theorem}\label{thm:gradedrootindependence}
Let  $F \colon R_1 \rightarrow R_2$ be an isomorphism of graded roots. Then we have an $S^1$-equivariant homotopy equivalence
\[
\mathcal{F} \colon \Hty(R_1) \rightarrow \Hty(R_2).
\]
If $F$ is a $J$-equivariant isomorphism of symmetric graded roots, then this is a $\Pin(2)$-equivariant homotopy equivalence.
\end{theorem}

\begin{proof}
The lifted morphism $\mathcal{F}$ induces an isomorphism on co-Borel homology. As in the proof of Theorem~\ref{thm:1.1} it follows that $\mathcal{F}$ induces a homotopy equivalence on all relevant fixed-point sets and is thus an equivariant homotopy equivalence. 
\end{proof}

Finally, we discuss some numerical results established in \cite{DaiMonopole} that will be useful in the sequel. Let $Y$ be an AR homology sphere with self-conjugate $\spinc$-structure $\s$. The Floer-theoretic invariants of $(Y, \s)$ can be read off from its symmetric graded root $R$ as follows:
\begin{enumerate}
\item Denote the grading of the uppermost vertex in $R$ by $2\delta(R)$. By definition of the Fr\o yshov invariant, we have 
\begin{equation}\label{eq:delta}
\delta(R) = \delta(Y, \s).
\end{equation}
\item Denote the grading of the uppermost $J$-invariant vertex in $R$ by $2 \beta(R)$. It is shown in \cite[Theorem 2.4]{DaiMonopole} and \cite[Theorem 3.2]{DaiMonopole} that 
\begin{equation}\label{eq:beta}
\beta(R) = \beta(Y, \s) = - \bar{\mu}(Y, \s).
\end{equation}
\end{enumerate} 
In particular, note that $\delta(R)$ and $\beta(R)$ take the same value among all representatives in the (graded root) local equivalence class of $R$.


By Theorem~\ref{thm:lesame}, for the purposes of computing the $K$-theoretic cobordism invariants of $\SWF(Y, \s) = \Hty(R)$, it suffices to replace $R$ with its monotone subroot. It will be useful for us to consider the following especially simple class of graded roots:

\begin{definition}\label{def:xn}
Let $X_n$ be the symmetric graded root given by
\[
X_n = \spa_{\Z[U]} \{v, Jv\} \quad \text{with} \quad U^nv = U^n(Jv).
\]
We fix the grading shift by setting $\gr(v) = \gr(Jv) = 2n$. With this convention, $\delta(X_n) = n$ and $\beta(X_n) = 0$. Denote
\[
A_n = \Hty(X_n).
\]
Note that we allow the degenerate case $n = 0$, in which case $A_0$ is homotopy equivalent to $S^0$.
\end{definition}

\begin{definition}\label{def:projective}
We say that a symmetric graded root $R$ is \textit{projective} if its monotone subroot is given by $X_n$ for some $n$, up to grading shift. If $R$ arises from $(Y, \s)$ and $R$ is projective, we refer to $(Y, \s)$ as projective. By Theorem~\ref{thm:lesame}, if $(Y, \s)$ is projective, then (up to suspension) $\SWF(Y, \s)$ is locally equivalent to $A_n$. It follows from the preceding discussion that in this situation
\[
n = \delta(Y, \s) - \beta(Y, \s) = \delta(Y, \s) + \bar{\mu}(Y, \s)
\]
and the suspension in question is given by
\[
\SWF(Y, \s) = \Hty(R) \sim \Sigma^{\beta(Y, \s) \C} A_n = \Sigma^{- \bar{\mu}(Y, \s) \C} A_n.
\]
See \cite[Section 1]{Stoffregen} and \cite[Section 1]{DaiManolescu} for further discussion of Definition~\ref{def:projective}.
\end{definition}


\subsection{$K$-theory of $A_n$}\label{subsec:computations}

The next lemma gives a complete calculation of the $K$-theory of $A_n$. We use notation as in \cite[Section 3]{ManolescuIntersection}. 

\begin{lemma}\label{lem:k-calc-projective}
We have
\[
\tilde{K}_{\Pin(2)}(A_n) = \mathfrak{I}(A_n) = (z^{\lceil n/2 \rceil},w) \quad \text{and} \quad K^1_{\Pin(2)}(A_n) = \Z[\theta,\theta^{-1}]/(z,(1-\theta)^n).
\]
Here, the formula for $K^1_{\Pin(2)}(A_n)$ is interpreted as an $R(G)$-module; see the proof for the notation.  In particular,
\[
\kappa(A_n) =
\begin{cases}
0 & \text { if } n = 0 \\
2 & \text{ else}.
\end{cases}
\]
\end{lemma}

\begin{proof}
Consider the exact sequence 
\[
S^{n\mathbb{C}}\vee S^{n\mathbb{C}} \to A_n \to S^{\tilde{\mathbb{R}}} \to \Sigma^{\mathbb{R}}(S^{n\mathbb{C}}\vee S^{n\mathbb{C}})\to\cdots
\]
where the action of $j$ on $S^{n\mathbb{C}}\vee S^{n\mathbb{C}}$ interchanges the two wedge summands. Apply the $S^1$-fixed point functor $\Phi$. This gives the exact sequence 
\[
S^0\vee S^0\to S^0 \to S^{\tilde{\R}} \to \Sigma^{\R}(S^0 \vee S^0) \to \cdots.
\]
Taking $G=\Pin(2)$-equivariant $K$-theory, we have the commutative diagram 
\[
\cdots
\begin{tikzcd}[column sep=0.75cm]
	{K^1_G(S^{\tilde{\mathbb{R}}})} & {\tilde{K}_G(S^{n\C} \vee S^{n\C})} & {\tilde{K}_G(A_n)} & {\tilde{K}_G(S^{\tilde{\mathbb{R}}})} & {K^1_G(S^{n\C} \vee S^{n\C})} \\
	{K^1_G(S^{\tilde{\mathbb{R}}})} & {\tilde{K}_G(S^0\vee S^0)} & {R(G)} & {\tilde{K}_G(S^{\tilde{\mathbb{R}}})} & {K^1_G(S^0\vee S^0)}
	\arrow["{\delta}"', from=1-2, to=1-1]
	\arrow["{i^*}"', from=1-3, to=1-2]
	\arrow[from=1-4, to=1-3]
	\arrow[from=1-5, to=1-4]
	\arrow["{\delta}"', from=2-2, to=2-1]
	\arrow["{i^*}"', from=2-3, to=2-2]
	\arrow[from=2-4, to=2-3]
	\arrow[from=2-5, to=2-4]
	\arrow["{=}", from=1-1, to=2-1]
	\arrow[from=1-2, to=2-2]
	\arrow["{\Phi^*}", from=1-3, to=2-3]
	\arrow["{=}", from=1-4, to=2-4]
	\arrow[from=1-5, to=2-5]
\end{tikzcd}
\cdots
\]
Here, $R(G) = \tilde{K}_G(S^0)$ is the complex representation ring of $G$. Our goal is to calculate both the $K$-theory $\tilde{K}_G(A_n)$ and the ideal $\mathfrak{I}(A_n) = \im \Phi^*$ from which $\kappa$ is defined.

We begin by determining some of the groups and maps in the above commutative diagram. 
\begin{enumerate}
\item Since $j$ acts freely away from the basepoint on $S^0 \vee S^0$, we have
\[
\tilde{K}_G(S^0\vee S^0) = \tilde{K}_{S^1}(S^0) = R(S^1) = \Z[\theta, \theta^{-1}],
\]
where $\theta$ is the standard complex representation of $S^1$. Likewise,
\[
\tilde{K}_G(S^{n\C} \vee S^{n\C}) = \tilde{K}_{S^1}(S^{n\C}) = R(S^1) = \Z[\theta, \theta^{-1}],
\]
by using Bott periodicity. Under this identification, the second vertical map is given by multiplication by the equivariant Euler class $c^n$, where $c = 1-\theta$.
\item By a similar argument, 
\[
K^1_G(S^0\vee S^0) = K^1_{S^1}(S^0) = 0
\]
and 
\[
K^1_G(S^{n\C} \vee S^{n\C}) = K^1_{S^1}(S^{n\C}) = 0.
\]
\item In \cite[Section 2.2]{ManolescuIntersection} it is shown that 
\[
R(G)=\mathbb{Z}[z,w]/(2w-zw,w^2-2w).  
\]
Under the above identification of $\tilde{K}_G(S^0\vee S^0)$ with $\Z[\theta, \theta^{-1}]$, the map
\[
i^* \colon R(G) \rightarrow \tilde{K}_G(S^0\vee S^0)
\]
in the bottom row sends $1\to 1,z\to 2- (\theta+\theta^{-1})$, and $w\to 0$. Note that $\ker i^* = (w)$. 
\end{enumerate}
Putting this together gives the commutative diagram
\begin{equation}\label{eq:kexact}
\cdots
\begin{tikzcd}
	{K^1_G(S^{\tilde{\mathbb{R}}})} & {\Z[\theta, \theta^{-1}]} & {\tilde{K}_G(A_n)} & {\tilde{K}_G(S^{\tilde{\mathbb{R}}})} \\
	{K^1_G(S^{\tilde{\mathbb{R}}})} & {\Z[\theta, \theta^{-1}]} & {R(G)} & {\tilde{K}_G(S^{\tilde{\mathbb{R}}})}
	\arrow["{\delta}"', from=1-2, to=1-1]
	\arrow["{i^*}"', from=1-3, to=1-2]
	\arrow[hookleftarrow, from=1-3, to=1-4]
	\arrow["{\delta}"', from=2-2, to=2-1]
	\arrow["{i^*}"', from=2-3, to=2-2]
	\arrow[hookleftarrow, from=2-3, to=2-4]
	\arrow["{=}", from=1-1, to=2-1]
	\arrow["{c^n}", from=1-2, to=2-2]
	\arrow["{\Phi^*}", from=1-3, to=2-3]
	\arrow["{=}", from=1-4, to=2-4]
\end{tikzcd}
\end{equation}
We first observe $w \in \im \Phi^*$. Indeed, $i^* \colon R(G) \rightarrow \Z[\theta, \theta^{-1}]$ has $w$ in its kernel; hence $w$ lies in the image of the bottom-right map. The fact that the fourth column of (\ref{eq:kexact}) is an isomorphism then gives the claim. 

We now claim that a class $y \in R(G)$ is in $\im \Phi^*$ if and only if $i^*(y) \in \im c^n$. Indeed, if $y = \Phi^*(x)$ for some $x \in \tilde{K}_G(A_n)$, then it is clear that $i^*(y) \in \im c^n$ by writing
\[
i^*(y) = i^*(\Phi^*(x)) = c^n(i^*(x)).
\]
Conversely, suppose $i^*(y)$ is in $\im c^n$; say $i^*(y) = c^n x'$. By exactness, $i^*(y)$ is annihilated by $\delta$. The fact that the first column is an isomorphism then shows that $x'$ must also be annihilated by $\delta$. Thus $x' \in \im i^*$; say $x' = i^*(x)$. Then we have
\[
i^*(y) = c^n(i^*(x)) = i^*(\Phi^*(x)).
\]
Hence $y$ is in the image of $\Phi^*$ modulo the kernel of $i^* \colon R(G) \rightarrow \Z[\theta, \theta^{-1}]$, which is $(w)$. However, we have already shown $w \in \im \Phi^*$, so this establishes the claim.

We now use this to compute $\mathfrak{I}(A_n) = \im \Phi^*$. We have already seen that $w \in \mathfrak{I}(A_n)$; we thus consider $z$. It is straightforward to check
\[
i^*(z^t) = c^t\bar{c}^t
\]
where $\bar{c}=1-\theta^{-1}=-\theta^{-1}(1-\theta)$. This is in the image of $c^n$ if and only if $2t \geq n$; hence
\[
\mathfrak{I}(A_n) = (z^{\lceil n/2 \rceil},w), 
\]
as desired. We moreover claim that $\Phi^*$ is injective. Indeed, let $x \in \tilde{K}_G(A_n)$ be non-zero. If $i^*(x) \neq 0$, then the fact that multiplication by $c^n$ is injective, together with commutativity of the second square in (\ref{eq:kexact}), shows that $\Phi^*(x) \neq 0$. On the other hand, if $i^*(x) = 0$, then $x$ is in the image of an element in the top-right corner of (\ref{eq:kexact}). In this case, the fact that the fourth column injects into the third column, together with commutativity of the third square, shows $\Phi^*(x) \neq 0$. Hence $\Phi^*$ is injective and $\tilde{K}_G(A_n) = \mathfrak{I}(A_n)$.

The assertion regarding $K^1_G(A_n)$ may be established by continuing (\ref{eq:kexact}) one column to the left. We focus on the portion
\begin{equation}
\cdots
\begin{tikzcd}
	{K^1_G(A_n)} & {K^1_G(S^{\tilde{\mathbb{R}}})} & {\Z[\theta, \theta^{-1}]} \\
	{0} & {K^1_G(S^{\tilde{\mathbb{R}}})} & {\Z[\theta, \theta^{-1}]} 
	\arrow[twoheadrightarrow, from=1-2, to=1-1]
	\arrow["{\delta}"', from=1-3, to=1-2]
	\arrow[twoheadrightarrow, from=2-2, to=2-1]
	\arrow["{\delta}"', from=2-3, to=2-2]
	\arrow[from=1-1, to=2-1]
	\arrow["{=}", from=1-2, to=2-2]
	\arrow["{c^n}", from=1-3, to=2-3]
\end{tikzcd}
\cdots
\end{equation}
Here, we have used the fact that $K^1_G(S^{n\C} \vee S^{n\C}) = K^1_G(S^0 \vee S^0) = 0$. Now, the map $\delta$ in the bottom row is surjective.  Thus $K^1_G(A_n)$ is the quotient of $\tilde{K}^1_G(S^{\tilde{\mathbb{R}}})=\Z[\theta,\theta^{-1}]/\Z[z]$ by the submodule generated by $c^n(\Z[\theta,\theta^{-1}])$.
\end{proof}

In fact, up to grading shift, the calculation of $\kappa$ in Lemma~\ref{lem:k-calc-projective} holds for a general homotopy type obtained from any graded root, not just $A_n = \mathcal{H}(X_n)$. 

\begin{lemma}\label{lem:k-calc-monotone}
Let $R$ be any symmetric graded root with $\beta(R) = 0$. Then
\[
\kappa(\Hty(R)) =
\begin{cases}
0 & \text { if } \beta(R) = \delta(R) \\
2 & \text{ if } \beta(R) < \delta(R).
\end{cases}
\]
\end{lemma}

\begin{proof}
By Theorem~\ref{thm:lesame}, we may assume that $R$ is monotone. If $\beta(R) = \delta(R)$, then $R$ is isomorphic to $\Z[U]$ and the claim is trivial. Thus we assume $\beta(R) < \delta(R)$. First suppose that the uppermost $J$-invariant vertex of $R$ is not a leaf. Then it is not hard to construct roots $X_m$ and $X_n$ with $m$ and $n$ both non-zero such that there exist local maps from $X_m$ to $R$ and $R$ to $X_n$. This procedure is depicted in Figure~\ref{fig:rootmaps} and follows from a similar construction as in \cite[Section 6]{DaiManolescu}. Lift these local maps to maps of spectra. Then we have
\[
\kappa(A_m) \leq \kappa(\Hty(R)) \leq \kappa(A_n)
\]
and the claim follows from Lemma~\ref{lem:k-calc-projective}.

\begin{figure}[h!]
\includegraphics[scale = 0.8]{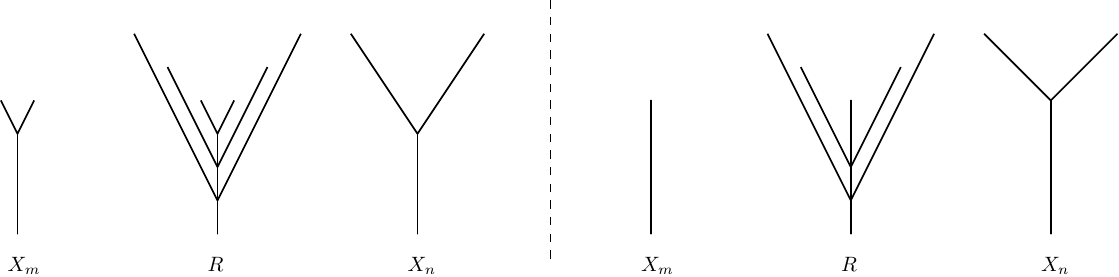}
\caption{Left: sandwiching a monotone root between $X_m$ and $X_n$. Right: the case when the highest-weight central vertex of $R$ is a leaf.}\label{fig:rootmaps}
\end{figure}

In the event that the uppermost $J$-invariant vertex is a leaf, the local map from $R$ to $X_n$ again gives the inequality $\kappa(\Hty(R))\leq 2$. In this case, however, the root $X_m$ which maps into $R$ is degenerate. Thus, on the other side we only obtain the inequality $0 \leq \kappa(\Hty(R))$. To complete the argument, we instead observe that we have a commutative diagram 
\[
\begin{tikzcd}
	{\tilde{K}_G(\Hty(R))} & {\tilde{K}_{S^1}(\Hty(R))} \\
	{R(G)} & {\Z[\theta,\theta^{-1}]} 
	\arrow[from=1-1, to=1-2]
	\arrow[from=2-1, to=2-2]
	\arrow[from=1-2, to=2-2]
	\arrow["\Phi^*"', from=1-1, to=2-1]
\end{tikzcd}
\]
where the bottom row is obtained by taking $S^1$-fixed points. Now, the image of the right-hand vertical map is easily seen to be $\smash{((1-\theta)^{\delta(R)})}$ using Bott periodicity.  Thus, if $\delta(R) > \beta(R) = 0$ then $1$ is not in the image of the right-hand vertical map. The commutativity of the diagram, combined with the form of the map $R(G) \rightarrow \Z[\theta, \theta^{-1}]$, gives $\kappa(\Hty(R))\geq 2$. 
\end{proof}

This easily gives the proof of Corollary~\ref{cor:a}:

\begin{proof}[Proof of Corollary~\ref{cor:a}]
Let $R$ be the symmetric graded root corresponding to $(Y, \s)$, so that $\kappa(Y, \s) = \kappa(\Hty(R))$. Suspend $\Hty(R)$ by $- \beta(Y, \s) \C$ so that the resulting symmetric graded root has $\beta = 0$. The claim then follows from applying Lemma~\ref{lem:k-calc-monotone}, undoing the suspension, and applying (\ref{eq:beta}).
\end{proof}

\subsection{Connected sums}\label{sec:connectedsum}

We now consider connected sums of projective AR homology spheres. Let $\{(Y_i, \s_i)\}_{i = 1}^n$ be such a family. As discussed in Definition~\ref{def:projective},
\[
\SWF(Y_i, \s_i) \sim \Sigma^{- \bar{\mu}(Y_i, \s_i) \C} A_{n_i} \quad \text{with} \quad n_i = \delta(Y_i, \s_i) - \beta(Y_i, \s_i)
\]
for each $i$. By the connected sum formula, we thus have 
\begin{equation}\label{eq:gradingshift}
\SWF(\#_{i = 1}^n Y_i, \#_{i =1}^n \s_i) \sim \Sigma^{- (\sum_i \bar{\mu}(Y_i, \s_i)) \C}\left(\wedge_{i = 1}^n A_{n_i}\right).
\end{equation}
Hence we must understand smash products of the $A_{n_i}$. In what follows, we describe two slightly different models for computing such a smash product. The first is the result of taking the smash product directly in the usual manner. Unfortunately, this turns out to be somewhat difficult to analyze, so we provide a second model with a simpler $j$-action which turns out to be more amenable to our purposes. However, it is not immediately clear that this second model gives the same $\Pin(2)$-homotopy type. The present subsection is devoted to establishing this fact, which is then leveraged in the sequel to carry out our $K$-theoretic calculations.

\begin{definition}\label{def:Acomplex}
Write 
\[
A = \wedge_{i = 1}^n A_{n_i}.
\]
Note that each $A_{n_i}$ may be thought of as the lattice space arising from a cubical complex homeomorphic to $[-1, 1]$. This complex consists of the two 0-cells $-1$ and $1$, together with the line segment between them. We define the weight function on each of these cells by setting $w(-1) = w(1) = 2n_i$ and giving the line segment weight zero. Applying the procedure of Sections~\ref{sec:5.1} and \ref{sec:6.2} to this cubical complex clearly gives the Pin(2)-homotopy type $A_{n_i}$. Note that in the degenerate case $A_0$, all three cells have weight zero. The $\Pin(2)$-homotopy type of $A$ is obtained by smashing together the $A_{n_i}$ in the usual manner.
\end{definition}

We now attempt to write down a model for $A = \wedge_{i = 1}^n A_{n_i}$ as the homotopy type associated to a weighted cubical complex. Let $[-1, 1]^n$ be the product of the cubical complexes $[-1, 1]$ associated to each of the $A_{n_i}$. We define the weight of a cube $\square_d \in [-1, 1]^n$ to be the sum of the weights of its factors. It is not difficult to see that forming the smash product corresponds to building a homotopy type out of $[-1, 1]^n$ in a similar manner as in the discussion of Section~\ref{sec:5.1}; namely, by introducing a sphere $\smash{S^{(w(\square_d)/2)\C} \wedge \square_d^+}$ for each $\square_d \in [-1, 1]^n$ and gluing these together. 

However, the gluing maps between the $\smash{S^{(w(\square_d)/2)\C} \wedge \square_d^+}$ are \textit{not} the ones discussed in Section~\ref{sec:5.1}. Indeed, the gluing maps of Section~\ref{sec:5.1} come from inclusion maps $\C^m \hookrightarrow \C^n$ which are always taken along the first $m$ coordinates in the codomain. In contrast, the gluing maps used to construct $A$ come from coordinate inclusions which vary depending on the cube $\square_d \in [-1, 1]^n$. This follows from the definition of the smash product. For instance, to build $A_{n_1} \wedge A_{n_2}$, we introduce four copies of $\smash{S^{((n_1 + n_2)/2)\C}}$ (one for each corner). We then introduce two copies of $\smash{S^{(n_1/2)\C} \times I}$ (one for each horizontal edge) and two copies of $\smash{S^{(n_2/2)\C}\times I}$ (one for each vertical edge). The former two cylinders are glued to the corner spheres using the inclusion
\[
\C^{n_1/2} \rightarrow \C^{n_1/2} \times \{(0, \ldots, 0)\} \subset \C^{(n_1 + n_2)/2},
\]
while the latter two cylinders are glued to the corner spheres using the non-standard inclusion
\[
\C^{n_2/2} \rightarrow \{(0, \ldots, 0)\} \times \C^{n_2/2} \subset \C^{(n_1 + n_2)/2}.
\]
Finally, we glue in a copy of $S^0 \times I^2$, although for this final piece there is no subtlety in the gluing. We are thus in the situation of Section~\ref{sec:homotopycoherent}, where the maps between spheres are induced be a varying family of linear inclusions. While $A$ has a $\Pin(2)$-action coming from the smash product, this makes explicitly identifying the action of $j$ in coordinates difficult. 

We contrast $A$ with the following:

\begin{definition}\label{def:Atildecomplex}
Suppose all of the $n_i$ are even. Construct a $\Pin(2)$-space $\smash{\tilde{A}}$ as follows. Form the cubical complex $[-1, 1]^n$ as in Definition~\ref{def:Acomplex}. For each $\square_d \in [-1, 1]^n$, we introduce a $\Pin(2)$-sphere $S^{(w(\square_d)/4)\mathbb{H}}$. To attach the spheres corresponding to cells $\square_d\subset \square_{d+1}$, we use the map $S^{(w(\square_{d+1})/4)\mathbb{H}}\to S^{(w(\square_d)/4)\mathbb{H}}$ induced by the standard inclusion \[\mathbb{H}^{w(\square_{d+1})/4}\to \mathbb{H}^{w(\square_{d+1})/4}\times \{(0, \ldots, 0)\} \subset \mathbb{H}^{w(\square_{d})/4},\] 
of a smaller-dimensional quaternionic vector space into a larger one along the first $w(\square_{d+1})/4$ coordinates. The resulting space $\smash{\tilde{A}}$ inherits a $\Pin(2)$-action as a subset of the $\Pin(2)$-space 
\[
S^{(\sum n_i/2)\mathbb{H}} \wedge [-1,1]^n_+,
\]
where the action of $j$ on $[-1, 1]^n$ is given by reflection through the origin. 
\end{definition}

The distinction between $A$ and $\tilde{A}$ is as follows: both spaces are built from copies of the same spheres (associated to the cubes $\square_d \in [-1, 1]^n$) but the inclusions among these spheres are different and the $j$-action is not the same. Nevertheless, we claim that $A$ and $\smash{\tilde{A}}$ are $\Pin(2)$-equivalent. We prove this by showing that in certain cases, any two linear homotopy coherent diagrams on the same weighted cube category have the same homotopy colimit. The reader should refer back to Section~\ref{sec:homotopycoherent} for a review of these notions.

We will need the following definitions. The first is an appropriate notion of $j$-equivariance, as we are interested in $A$ and $\smash{\tilde{A}}$ as $\Pin(2)$-spaces. 

\begin{definition}\label{def:jequivariantcubecategory}
Let $\Cat$ be a weighted cube category. We say that $\Cat$ is \textit{$j$-equivariant} if:
\begin{enumerate}
\item $\Cat$ is equipped with an involution $J$, coming from multiplication by $-1$ on the lattice, such that $w$ is $J$-invariant; and,
\item All of the $w(\square_d)$ are divisible by four.
\end{enumerate} 
In this case, each $\smash{S^{(w(\square_d)/2)\C}}$ may be viewed as a quaternionic sphere $S^{(w(\square_d)/4)\bbH}$; henceforth we implicitly identify $(w(\square_d)/2)\mathbb{C}$ and $(w(\square_d)/4)\mathbb{H}$.
\end{definition}

\begin{definition}\label{def:jequivarianthoco}
Let $\Cat$ be a $j$-equivariant cube category. A homotopy coherent diagram $F$ on of $\Cat$ is called \emph{$j$-equivariant} if for each composable sequence of morphisms 
\[
x_0 \xrightarrow{f_1} x_1 \xrightarrow{f_2} \cdots \xrightarrow{f_n} x_n
\]
the structure map of \ref{def:homco} satisfies
\begin{equation}\label{eq:hoco-extension}
	F(\phi_{x_{n-1}, x_n}, \ldots, \phi_{x_0, x_1})j= j F(\phi_{J(x_{n-1}), J(x_n)}, \ldots, \phi_{J(x_0), J(x_1)})
\end{equation}
as a map from $F(x_0)$ to $F(x_n)$. Here, the two actions of $j$ are given by treating $F(x_0)$ and $F(x_n)$ as $\Pin(2)$-spheres. The homotopy colimit of a $j$-equivariant cube category evidently inherits an action of $\Pin(2)$.
\end{definition}

The second definition is a technical property that will allow us to prove (in certain cases) that the homotopy colimit is invariant of the choice of coherent diagram:

\begin{definition}
Let $\Cat$ be a weighted cube category. We say that the weight function $w$ is \textit{strict} if we have the strict inequality $w(\square_{d + i}) < w(\square_d)$ whenever $\mathrm{Mor}(\square_{d+i}, \square_d)$ is non-empty (and not the identity). Note that technically, no weight function arising in the manner described in Section~\ref{sec:4.1} is strict, since in the geometric setting $w(\square_d)$ is equal to the minimum of $w$ over the vertices of $\square_d$. However, in many cases we can remove degenerate cubes to obtain a strict weighted cube category with the same homotopy colimit. For example, in Definition~\ref{def:Acomplex} we describe $A_n$ as arising from a strict weighted cube category with three objects.
\end{definition}

Finally, we recall:

\begin{definition}\label{def:homotopyhoco}
Let $F_1$ and $F_2$ be homotopy coherent diagrams on a category $\Cat$ with $F_1(x)=F_2(x)$ for all $x \in \mathrm{Ob}(\Cat)$. Let $\two$ be the poset of two objects $0$ and $1$ with $\phi_{0,1}\colon 0 \to 1$ the unique non-identity morphism. A \emph{homotopy} from $F_1$ to $F_2$ is a homotopy coherent diagram $G$ on $\Cat\times \two$ such that:
	\begin{enumerate}
	\item $G|_{\Cat\times \{0\}}=F_1$ and $G|_{\Cat\times \{1\}}=F_2$; and,
	\item Each arrow $G(\mathrm{id}_v\times \phi_{0,1})\colon F_1(v)\to F_2(v)$ is the identity. 
	\end{enumerate}
A homotopy between two homotopy coherent diagrams induces an equivalence between their homotopy colimits (see \cite{LLS} or \cite[(ho-1)]{stoffregen-zhang}), so that
\[
\hocolim(F_1) \simeq \hocolim(F_2).
\]

If $\Cat$ is a $j$-equivariant cube category and we put the identity involution on $\two$, then a \textit{$j$-equivariant homotopy} is simply a homotopy $G$ which is $j$-equivariant in the sense of Definition~\ref{def:jequivarianthoco}. (Note that $G$ is a homotopy coherent diagram on $\Cat \times \two$.) A $j$-equivariant homotopy induces a $\Pin(2)$-equivariant equivalence. This essentially follows from \cite{dotto-moi} or the discussion of \cite[Section 5]{stoffregen-zhang}, but can also be checked at a point-set level in our case. 
\end{definition}

With these in hand, we have:

\begin{lemma}\label{lem:lattice-homotopy-coherent-diagram}
Let $\Cat$ be a $j$-equivariant weighted cube category. Assume that the weight function on $\Cat$ is strict. Let $F_1$ and $F_2$ be two $j$-equivariant linear homotopy coherent diagrams on $\Cat$. Then
	\begin{equation}\label{eq:hocolim-f1-f2}
		\hocolim_{C}\, F_1 \simeq_{\Pin(2)} \hocolim_{C}\, F_2.	
	\end{equation}
\end{lemma}
\begin{proof}
We first ignore the $j$-action and give the non-equivariant version of the argument. Following Definition~\ref{def:homotopyhoco}, our goal will be to construct a homotopy coherent diagram $G$ on $\mathcal{C}\times \two$ so that the restrictions satisfy $G|_{\mathcal{C}\times \{i\}}=F_i$ and so that $G|_{\mathrm{id}_{\square}\times \phi_{0,1}}=\mathrm{id}$ for each object $\square$ of $\mathcal{C}$. As we will see, the linearity of $F_1$ and $F_2$ will be crucial. Recall that this means all morphisms and structure maps come from elements of the Stiefel manifold
\[
W(\square_d, \square_{d+i}) = V_{w(\square_{d+i})/2}(\C^{w(\square_d)/2}).
\]
Note that $W(\square_d, \square_{d+i})$ is $(w(\square_d) - w(\square_{d+i}))$-connected.

	
	We construct $G$ in stages.  First, $G$ is already defined on all of $\mathcal{C}\times \{0\}$ and $\mathcal{C}\times \{1\}$, and is necessarily defined to be the identity on arrows of the form $\mathrm{id}\times \phi_{0,1}$. Now, suppose we are given a sequence of composable arrows in $\mathcal{C}\times \two$ of length one -- i.e., a single arrow. This is either of the form $\mathrm{id}\times \phi_{0,1}$, or $\phi_{\square_{d+i},\square_d}\times \mathrm{id}$, or $\phi_{\square_{d+i},\square_d}\times \phi_{0,1}$ with $\phi_{\square_{d+i},\square_d} \neq \mathrm{id}$. We have already defined $G$ on the former two types; for the latter arrow, we define $G(\phi_{\square_{d+i},\square_d}\times \phi_{0,1})$ by using any element of $W(\square_{d},\square_{d+i})$.  Next, we must define $G$ over $2$-simplices -- i.e., given two composable arrows $\psi_1\colon x_0\to x_1$ and $\psi_{2}\colon x_1\to x_2$ in $\mathcal{C}\times \two$, we need to define the family of maps 
\begin{equation}\label{eq:desired-g}G(\psi_{1},\psi_{2})\colon [0,1]\times G(x_0)\to G(x_2).\end{equation} 	
Now, $G(\psi_{1},\psi_{2})|_{\partial [0, 1] \times G(x_0)}$ is already defined by the requirement that
\[
G(\psi_1,\psi_2)(0, - )=G(\psi_2)G(\psi_1) \quad \text{and} \quad G(\psi_1,\psi_2)(1, - )=G(\psi_2\circ \psi_1).
\]
 In the event that $x_0$ and $x_2$ both lie in $\mathcal{C}\times \{0\}$ or $\mathcal{C}\times \{1\}$, $G(\psi_1, \psi_2)$ is already defined; otherwise, we must fill in the definition of $G(\psi_1, \psi_2)$ to all of $[0, 1] \times G(x_0)$. Now, considering the sequence 
 \[
 x_0 \xrightarrow{\psi_1} x_1 \xrightarrow{\psi_2} x_2,
 \]
we see that exactly one of $\psi_1$ and $\psi_2$ must run from $\Cat \times \{0\}$ to $\Cat \times \{1\}$, with the other being internal to either $\mathcal{C}\times \{0\}$ or $\mathcal{C}\times \{1\}$. Hence we have that either $x_1$ is a face of $x_0$ or $x_2$ is a face of $x_1$. Now, $G(\psi_2)G(\psi_1)$ and $G(\psi_2\circ \psi_1)$ lie in $W(x_2, x_0)$. Since we have assumed $\Cat$ is strict and that all weights are divisible by four, we have $w(x_2)\geq w(x_0)+4$. In particular,
\[
W(x_2, x_0) \text{ is $4$-connected}.
\]
We can thus certainly choose a path in $W(x_2, x_0)$ connecting $G(\psi_2)G(\psi_1)$ and $G(\psi_2\circ \psi_1)$. This path defines the family $G(\psi_1,\psi_2)$ as in (\ref{eq:desired-g}).

The same process can be repeated over all of the simplices of $\mathcal{C}\times \two$ to construct $G$, using at each stage that the Stiefel manifold involved is highly-connected relative to the boundary of the simplex that must be filled. Explicitly, suppose we have a sequence of composable morphisms
\[
 x_0 \xrightarrow{\psi_1} x_1 \xrightarrow{\psi_2} \cdots \xrightarrow{\psi_k} x_k.
\]
We need to find a map
\begin{equation}\label{eq:desired-g-2}
G(\psi_1,\ldots,\psi_k)\colon [0,1]^{k-1}\times G(x_0)\to G(x_k),
\end{equation}
where we have already inductively defined $G(\psi_1,\ldots,\psi_k)$ over $\partial [0,1]^{k-1}$. We may assume there is a single $\psi_i$ that runs from $\Cat \times \{0\}$ to $\Cat \times \{1\}$, since otherwise $G$ is already defined. Thus there is at most one successive pair for which $x_i$ and $x_{i+1}$ correspond to the same lattice cube in $\Cat$; for all other successive pairs, $x_{i+1}$ is a face of $x_i$. It follows that $w(x_k) \geq w(x_0) + 4(k-1)$; hence 
\[
W(x_k, x_0) \text{ is $4(k-1)$-connected}. 
\]
However, the restriction of $G$ in (\ref{eq:desired-g-2}) to the boundary is determined by a $(k-2)$-dimensional sphere in $W(x_k, x_0)$.  By the connectivity of $W(x_k, x_0)$, we can fill in this family to get a $(k-1)$-ball in $W(x_k, x_0)$, which determines the map $G$ in (\ref{eq:desired-g-2}).  This completes the construction of the (non-equivariant) homotopy of homotopy coherent diagrams, establishing (\ref{eq:hocolim-f1-f2}) as $S^1$-spaces.

We now extend this argument to the equivariant case. Say we have constructed $G$ on sequences of arrows of length $k-1$, as in the non-equivariant case, and that (\ref{eq:hoco-extension}) holds on the maps defined so far. We now construct $G$ on sequences of length $k$. Consider such a sequence
\[
 x_0 \xrightarrow{\psi_1} x_1 \xrightarrow{\psi_2} \cdots \xrightarrow{\psi_k} x_k.
\]
We have an action of $J$ on the set of such sequences by sending $(\psi_1, \ldots, \psi_k)$ to $(J(\psi_1), \ldots, J(\psi_k))$. If $(\psi_1,\ldots,\psi_k)$ is not in the $J$-orbit of length-$k$ sequences has been constructed so far, then we define $G$ in (\ref{eq:desired-g-2}) exactly as in the non-equivariant case. On the other hand, if the tuple $(\psi_1,\ldots,\psi_k)=(J( \psi_1'),\ldots, J(\psi_k'))$ and $G$ has already been constructed for the tuple $(\psi_1',\ldots,\psi_k')$, the we simply define $G$ on $(\psi_1,\ldots, \psi_k)$ by the formula (\ref{eq:hoco-extension}). This constructs $G$ as a $j$-equivariant homotopy and thus completes the proof.  Note that the action of $J$ on tuples $\psi_1,\ldots,\psi_k$ has no fixed points. 
\end{proof}

With Lemma~\ref{lem:lattice-homotopy-coherent-diagram} in hand, we show $A$ and $\tilde{A}$ coincide.

\begin{lemma}\label{lem:identifyingjaction}
Assume that each $n_i$ is even and that $n_i \geq 2$ for all $i$. Then there is a Pin(2)-equivalence
\[
\wedge_{i = 1}^n A_{n_i} = A \simeq \tilde{A}.
\]
\end{lemma}

\begin{proof}
By Lemma~\ref{lem:lattice-homotopy-coherent-diagram}, it suffices to show that $A$ and $\smash{\tilde{A}}$ are both homotopy colimits of $j$-equivariant linear homotopy coherent diagrams. This is true for $\smash{\tilde{A}}$ by construction. To see the claim for $A$, recall that
\[
A = \wedge_{i = 1}^n A_{n_i}.
\]
Now, it is clear that each $A_{n_i}$ is the homotopy colimit of an obvious $j$-equivariant linear homotopy coherent diagram on the cube category $[-1, 1]$. By \cite[Section 4.2, (ho-3)]{LLS}, if $F_1$ and $F_2$ are homotopy coherent diagrams $F_i\colon C_i\to \topo$, then there is an induced functor
	\[
	F_1\wedge F_2\colon C_1\times C_2\to \topo,
	\] 
and $\hocolim(F_1\wedge F_2)$ is weakly equivalent to $\hocolim (F_1)\wedge \hocolim (F_2)$.  The structure maps in $F_1\wedge F_2$ are the natural ones; in particular, if $F_1$ and $F_2$ are linear, then so is $F_1\wedge F_2$.  This works equally well for $G$-$\topo$ in place of $\topo$, and is also compatible with the $j$-actions on $F_1$ and $F_2$ in our case. It is clear that the weight function on $[-1, 1]^n$ is strict, so applying Lemma \ref{lem:lattice-homotopy-coherent-diagram} completes the proof.
\end{proof}

\subsection{$K$-theory of connected sums}\label{sec:connectedsumcalculation}

We now complete our $K$-theoretic calculations. It will be helpful to consider the subset of $A$ given by restricting the construction of $A$ to the cubical subcomplex which constitutes the boundary of $[-1, 1]^n$. This also is a Pin(2)-space, which we denote by $M$. Thus, $M$ is constructed by gluing together spheres for each cube in $[-1, 1]^n$, with the exception of the single $n$-dimensional cube $\square_n$ that makes up the interior of $[-1, 1]^n$. It is easily checked that $M$ has no (non-basepoint) $\Pin(2)$-fixed points and that 
\[
A/M = S^{n\tilde{\R}}
\]
as a $\Pin(2)$-spectrum. Indeed, $A/M$ is constructed by taking $S^0 \times \square_n$ and crushing $\{\infty\} \times \square_n$ and $\{0\} \times \partial(\square_n)$ to the basepoint; the action of $j$ is $-1$ on $\square_n$. Denote the inclusion and quotient maps by
\[
M \xrightarrow{i} A \xrightarrow{p} A/M.
\]

We begin by bounding the $\kappa$-invariant of $A$. The following holds without any restriction on the $n_i$ or $n$:

\begin{lemma}\label{lem:upperboundconnectedsum}
The ideal $\mathfrak{J}(A)$ contains $w^{\lceil n/2 \rceil}$. In particular, 
\[
\kappa(A) \leq 2 \lceil n/2 \rceil.
\]
\end{lemma}
\begin{proof}
We first assume $n$ is even.  Let $A$ and $M$ be as above. Since $n$ is even, we write 
\[
A/M = S^{(n/2)\tilde{\C}}.
\]
Here, $\tilde{\C}$ is the complex $\Pin(2)$-representation where $S^1$ acts by the identity and $j$ acts by $-1$. Consider the map in $K$-theory induced by $p \colon A \rightarrow A/M$. This fits into the commutative square
\[
\begin{tikzcd}
	{\tilde{K}_G(A)} & {\tilde{K}_G(S^{(n/2)\tilde{\C}})} \\
	{R(G)} & {\tilde{K}_G(S^{(n/2)\tilde{\C}})} 
	\arrow["p^*"', from=1-2, to=1-1]
	\arrow["p^*"', from=2-2, to=2-1]
	\arrow["=", from=1-2, to=2-2]
	\arrow["\Phi^*"', from=1-1, to=2-1]
\end{tikzcd}
\]
where the bottom row is obtained by taking $S^1$-fixed points. The bottom-right corner may be identified with $R(G)$; the map $p^*$ along the bottom row is then multiplication by the equivariant Euler class of $(n/2)\tilde{\mathbb{C}}$, which is $w^{n/2}$. Commutativity then shows that $w^{n/2}$ is in the image of $\Phi^*$, from which the claim follows. This establishes the case when $n$ is even. If $n$ is odd, then we replace $A$ with $A \wedge A_0$; note that $A\simeq A\wedge A_0$ and that $A \wedge A_0$ has an even number of wedge summands.
\end{proof}

Under more restrictive conditions, we can carry out a sharper calculation. This will require a more explicit understanding of $M$. Here, Lemma~\ref{lem:identifyingjaction} will be crucial in allowing us to replace $A$ with $\smash{\tilde{A}}$, although to simplify notation we will continue to refer to our space simply as $A$.

\begin{lemma}\label{lem:exactidealcalculation}
Let $n$ be even. Assume that each $n_i$ is even and that $n_i \geq 2$ for all $i$. Then $\mathfrak{I}(A)=(z^{(\sum n_i/2)},w^{n/2})$. In particular, 
\[
\kappa(A) = n.
\]
\end{lemma}
\begin{proof}
Consider the $K$-theoretic exact sequence induced by $i$ and $p$. This gives 
\begin{equation}\label{eq:k-exact-triangle-3}
\cdots
\begin{tikzcd}
	{\tilde{K}_G(M)} & {\tilde{K}_G(A)} & {\tilde{K}_G(S^{(n/2)\tilde{\C}})} \\
	{\tilde{K}_G(M^{S^1})} & {R(G)} & {\tilde{K}_G(S^{(n/2)\tilde{\C}})}
	\arrow[twoheadrightarrow, "{i^*}"', from=1-2, to=1-1]
	\arrow["{p^*}"', from=1-3, to=1-2]
	\arrow[twoheadrightarrow, "{i^*}"', from=2-2, to=2-1]
	\arrow["{p^*}"', from=2-3, to=2-2]
	\arrow[from=1-1, to=2-1]
	\arrow["{\Phi^*}", from=1-2, to=2-2]
	\arrow["{=}", from=1-3, to=2-3]
\end{tikzcd}
\cdots
\end{equation}
where the bottom row is obtained by taking $S^1$-fixed points. Here, we use $\tilde{K}^1_G(S^{(n/2)\tilde{\mathbb{C}}})=0$ to conclude that the left-hand horizontal maps are surjective. As before, the bottom-right corner may be identified with $R(G)$, in which case the map $p^*$ along the bottom row is multiplication by $w^{n/2}$. Let $\tilde{S}^{n-1}$ denote the $n$-sphere equipped with the antipodal $j$-action. Then 
\[
\smash{M^{S^1}} = \tilde{S}^{n-1}_+.
\]
To see this, note that each sphere in the construction of $A$ has $S^1$-fixed point set $S^0$. Thus $\smash{M^{S^1}}$ is formed by taking $S^0 \times \partial([-1, 1]^n)$ and crushing $\{\infty\} \times \partial([-1, 1]^n)$ to the basepoint. 

Now observe that since each $n_i \geq 2$, we have a copy of $\Sigma^{\bbH}\tilde{S}^{n-1}_+$ in $M$. This is similar to our identification of $\smash{M^{S^1}}$: since each $n_i \geq 2$, every sphere associated to a lattice cube $\square_d$ in the boundary of $[-1, 1]^n$ is of (complex) dimension at least two. For each such sphere, consider the subsphere $S^{\bbH}$ corresponding to the first two (complex) coordinates. The subspheres $S^{\bbH}$ form a subset of $M$ which is obtained by taking $S^{\bbH} \times \partial([-1, 1]^n)$ and crushing $\{\infty\} \times \partial([-1, 1]^n)$ to the basepoint. Here we are using Lemma~\ref{lem:identifyingjaction} to implicitly replace $A$ by $\smash{\tilde{A}}$, so that the subspheres $S^{\bbH}$ are indeed glued together (and that we know the $j$-action). 

Consider the map in $K$-theory induced by this inclusion. This fits into the square
\[
\begin{tikzcd}
	{\tilde{K}_G(\Sigma^{\bbH}\tilde{S}^{n-1}_+)} & {\tilde{K}_G(M)} \\
	{\tilde{K}_G(\tilde{S}^{n-1}_+)} & {\tilde{K}_G(\tilde{S}^{n-1}_+)} 
	\arrow[from=1-2, to=1-1]
	\arrow["="', from=2-2, to=2-1]
	\arrow[from=1-2, to=2-2]
	\arrow[from=1-1, to=2-1]
\end{tikzcd}
\]
where the bottom row is obtained by taking $S^1$-fixed points. It is not hard to check that 
\[
\tilde{K}_G(M^{S^1})= \tilde{K}_G(\tilde{S}^{n-1}_+) = R(G)/(w^{n/2}).
\]
The bottom-left corner may thus be replaced by $R(G)/(w^{n/2})$. We may also identify the top-left corner with $R(G)/(w^{n/2})$, in which case the left-hand vertical map is multiplication by the equivariant Euler class $z$ in $R(G)/(w^{n/2})$. Hence in particular, the image of the right-hand vertical map must lie in the ideal $(z)$ of $R(G)/(w^{n/2})$. Returning to (\ref{eq:k-exact-triangle-3}), a diagram chase shows that the image of $\Phi^*$ must lie in $(z)$ modulo $(w^{n/2})$; that is, 
\[
\mathfrak{I}(A) \subset (z, w^{n/2}).
\]

Now consider the commutative diagram
\[
\begin{tikzcd}
	{\tilde{K}_{S^1}(A)} & {\tilde{K}_G(A)} \\
	{R(S^1)} & {R(G)} 
	\arrow[from=1-2, to=1-1]
	\arrow[from=2-2, to=2-1]
	\arrow["\Phi^*", from=1-2, to=2-2]
	\arrow[from=1-1, to=2-1]
\end{tikzcd}
\]
where the rows are induced by the inclusion of $S^1$ into $G$ and the columns are given by taking $S^1$-fixed points. The map in the bottom row is as in the proof of Lemma~\ref{lem:k-calc-projective}; this sends
\[
1 \to 1, \quad z\to 2-\theta-\theta^{-1}=-\theta^{-1} c^2, \quad \text{and} \quad w\to 0,
\]
where $c = 1- \theta$. Now, as an $S^1$-space, $A$ is locally equivalent to $S^{(\sum n_i)\mathbb{C}}$. Hence the left-hand vertical map is given by multiplication by $c^{(\sum n_i)}$. It follows that if $p(z)+aw$ is in the image of $\Phi^*$, we must have $p(c^2)\in (c^{(\sum n_i)})$. In particular,
\[
\mathfrak{I}(A)\subset (z^{(\sum n_i/2)},w).
\]
Combining this with our previous calculation involving $\mathfrak{I}(A)$, we have 
\[
\mathfrak{I}(A)\subset (z^{(\sum n_i/2)},w^{n/2}).
\]

This establishes half of the claim; we wish to show that the above inclusion is an equality. By Lemma~\ref{lem:upperboundconnectedsum}, we already know that $w^{n/2} \in \mathfrak{I}(A)$. It thus suffices to show that $z^{(\sum n_i/2)} \in \mathfrak{I}(A)$. Note that we may view
\[
M \subset \Sigma^{(\sum n_i/2)\mathbb{H}}\tilde{S}^{n-1}_+.
\]
Here we are again using Lemma~\ref{lem:identifyingjaction} to replace $A$ with $\tilde{A}$. As explained in Definition~\ref{def:Atildecomplex}, $\tilde{A}$ is naturally a subset of $\smash{S^{(\sum n_i/2)\mathbb{H}} \wedge [-1, 1]^n_+}$; the claim regarding $M$ follows as in our identification of $\smash{M^{S^1}}$. This inclusion gives a commutative diagram
\[
\begin{tikzcd}
	{\tilde{K}_{G}(M)} & {\tilde{K}_G(\Sigma^{(\sum n_i/2)\mathbb{H}}(\tilde{S}^{n-1}_+))} \\
	{\tilde{K}_G(\tilde{S}^{n-1}_+)} & {\tilde{K}_G(\tilde{S}^{n-1}_+)}
	\arrow[from=1-2, to=1-1]
	\arrow["="', from=2-2, to=2-1]
	\arrow[from=1-2, to=2-2]
	\arrow[from=1-1, to=2-1]
\end{tikzcd}
\]
where the bottom row is obtained by taking $S^1$-fixed points. Both the bottom-right and top-right corners may be identified with $R(G)/(w^{n/2})$, in which case the right-hand vertical arrow is multiplication by $\smash{z^{(\sum n_i/2)}}$. It is then clear that the left-hand vertical map has $\smash{z^{(\sum n_i/2)}}$ in its image. This map is precisely the leftmost vertical arrow in $(\ref{eq:k-exact-triangle-3})$; since the leftmost horizontal arrows in $(\ref{eq:k-exact-triangle-3})$ are surjections, it follows that $\smash{z^{(\sum n_i/2)}} \in \im \Phi^*$, as desired.
\end{proof}

We now proceed with the proofs of Corollary~\ref{cor:b} through \ref{cor:d}.

\begin{proof}[Proof of Corollary~\ref{cor:b}]
We have: 
\[
\kappa(\#_{i=1}^n (Y_i, \s_i))\leq -\sum_{i = 1}^n \bar{\mu}(Y_i, \s_i) + 2\lceil n/2 \rceil
\]
simply by applying Lemma~\ref{lem:upperboundconnectedsum} after taking into account the grading shift (\ref{eq:gradingshift}). 
\end{proof}

\begin{proof}[Proof of Corollary~\ref{cor:c}]
By Lemma~\ref{lem:upperboundconnectedsum}, we know that
\[
\kappa(\wedge_{i = 1}^n A_{n_i}) \leq n.
\]
On the other hand, observe that if $m < n$, then we have a local map from $X_m$ into $X_n$ and hence from $A_m$ into $A_n$. It follows from this that 
\[
n = \kappa(\wedge_{i = 1}^n A_2) \leq \kappa(\wedge_{i = 1}^n A_{n_i}),
\]
where we have applied Lemma~\ref{lem:exactidealcalculation} in the case that each $n_i = 2$. The claim then follows from taking into account the grading shift (\ref{eq:gradingshift}).
\end{proof} 

\begin{proof}[Proof of Corollary~\ref{cor:d}]
Fix any projective $(Y, \s)$ with $-\bar{\mu}(Y, \s) \leq \delta(Y, \s) - 2$. Then
\[
\kappa(\#_{i = 1}^n(Y, \s)) = - n \bar{\mu}(Y, \s) + n \leq n \delta(Y, \s) - n
\]
by Corollary~\ref{cor:c}. In \cite[Theorem 1.3]{Stoffregenconnectedsum}, it is shown that
\[
\alpha(\#_{i = 1}^n (Y, \s)) - n \delta(Y, \s), \quad \beta(\#_{i = 1}^n (Y, \s)) - n \delta(Y, \s), \quad \text{and} \quad \gamma(\#_{i = 1}^n (Y, \s)) - n \delta(Y, \s)
\]
are bounded functions of $n$. This establishes the claim.  The second statement follows by duality (for $K$-theory this takes the form $\kappa(Y)+\kappa(-Y)\geq 0$ for all $Y$).
\end{proof}

\section{Connections to Singularity Theory}\label{sec:ag}

\subsection{The vector space model}\label{sec:ag.1}

We now show that $\Hty(\Gamma, [k])$ has a close connection to the work of N\'emethi \cite{NemethiOS, Nemethi}. For this, it will be helpful to first have a minor re-phrasing of the construction of $\Hty(\Gamma, [k])$:

\begin{definition}\label{def:ag.1}
Let $(V_1, \ldots, V_n)$ be a sequence of (finite-dimensional) complex vector spaces. Suppose that for each consecutive pair $V_i$ and $V_{i+1}$, we are either given an inclusion $i \colon V_i \hookrightarrow V_{i+1}$ or an inclusion $i \colon V_{i+1} \hookrightarrow V_i$. Define $\Hty(V_1, \ldots, V_n)$ by contracting over the mapping cone of these inclusions:
\[
\Hty(V_1, \ldots, V_n) = \text{one-point compactification of } \left(\bigsqcup_i V_i\right)/\sim,
\]
where $\sim$ is constructed as follows. For each consecutive pair $V_i$ and $V_{i+1}$, we either have $i \colon V_i \hookrightarrow V_{i+1}$ or $i \colon V_{i+1} \hookrightarrow V_i$. In each case, glue the entire domain to the codomain via $i$, so that
\[
\begin{cases}
V_i \sim i(V_i) \subset V_{i+1} &\text{ if } V_i \hookrightarrow V_{i+1} \\
V_{i+1} \sim i(V_{i+1}) \subset V_{i} &\text{ if } V_{i+1} \hookrightarrow V_{i}.
\end{cases}
\]
This is a $S^1$-space with a basepoint at $\infty$. Note that up to $S^1$-homotopy equivalence, $\Hty(V_1, \ldots, V_n)$ does not depend on the choice of inclusion map from each $V_i$ to $V_{i+1}$ (or vice-versa), as any two inclusions are isotopic. 
\end{definition}

It is clear that the space $X$ of Section~\ref{sec:5.1} is $S^1$-homotopy equivalent to $\Hty(V_1, \ldots, V_n)$, where $V_i = \C^{(w(k_i) - h)/2}$ and we consider the standard inclusion maps
\[
\C^{(w(k_i) - h)/2} \hookrightarrow \C^{(w(k_{i+1}) - h)/2}  \quad \text{or} \quad \C^{(w(k_{i+1}) - h)/2} \hookrightarrow \C^{(w(k_i) - h)/2}
\]
according to whether $w(k_i) \leq w(k_{i+1})$ or $w(k_i) \geq w(k_{i+1})$. For example, $\SWF(\Sigma(2, 3, 7))$ may be constructed by taking two copies of $\C^1$, gluing a common copy of $\C^0$ to the origin in both, and then taking the one-point compactification. This construction may be viewed completely combinatorially, in the sense that the vector spaces $V_i$ and the inclusion maps between them need not be assigned any particular meaning. However, it is natural to ask:

\begin{question}\label{question:ag.2}
In the above construction, can the vector spaces and inclusions be interpreted in terms of the topology or geometry of $W_\Gamma$?
\end{question}

The impetus for Question~\ref{question:ag.2} is the following. If $W_\Gamma$ arises as the resolution of a singularity, then it inherits a preferred complex structure. As discussed in \cite[Section 2.2.5, Section 3.2.1]{Nemethi}, $w(k_i)$ is (up to certain convention changes) the holomorphic Euler characteristic of a certain sheaf $\O_{x_i}$ on $W_\Gamma$ associated to the characteristic element $k_i$. Hence there is already a natural candidate for $V_i$ (at least in the stable category): the formal difference $H^0(W_\Gamma, \O_{x_i}) - H^1(W_\Gamma, \O_{x_i})$.\footnote{Due to these convention changes, in our case $w(k_i)$ is not quite the Euler characteristic of this sheaf, but the difference is not significant. See the proof of Theorem~\ref{thm:ag.7}.} In order to make sense of this, we consider a slight modification of Definition~\ref{def:ag.1}.

\begin{definition}\label{def:ag.3}
If $A$ and $B$ are (finite-dimensional) complex vector spaces, we call $A - B$ their \textit{formal difference} and define $\dim(A - B) = \dim(A) - \dim(B)$. A \textit{formal difference map}
\[
f \colon (A - B) \rightarrow (A' - B')
\]
consists of a pair of linear maps $(f_1, f_2)$ with $f_1 \colon A \rightarrow A'$ and $f_2 \colon B' \rightarrow B$. (Note the domain and codomain of $f_2$.) This gives a map $f_1 \oplus f_2 \colon A \oplus B' \rightarrow A' \oplus B$. We say that $f$ is an inclusion if $f_1$ and $f_2$ are both inclusions, and so on.
\end{definition}

\begin{definition}\label{def:ag.4}
Let $(V_1, \ldots, V_n)$ be a sequence of formal differences
\[
V_i = A_i - B_i.
\]
Suppose that for each consecutive pair $V_i$ and $V_{i+1}$, we are either given an inclusion $i \colon V_i \hookrightarrow V_{i+1}$ or an inclusion $i \colon V_{i+1} \hookrightarrow V_i$, in the sense of Definition~\ref{def:ag.3}. By stabilizing, we can construct the mapping cone over these inclusions. Explicitly, consider the sequence of vector spaces
\[
V_i' = A_i \oplus \left(\bigoplus_{j \neq i} B_j\right).
\]
An inclusion map $i = (i_1, i_2)$ from $V_i$ to $V_{i+1}$ naturally induces an inclusion map from $V_i'$ to $V_{i+1}'$ which sends $A_i$ to $A_{i+1}$ via $i_1$, $B_{i+1}$ to $B_i$ via $i_2$, and $B_j$ isomorphically to $B_j$ via the identity map for $j \neq i$ or $i+1$. Similarly, an inclusion from $V_{i+1}$ to $V_i$ naturally induces an inclusion from $V_{i+1}'$ to $V_i'$. Define the contracted mapping cone to be the formal desuspension of the construction from Definition~\ref{def:ag.1}:
\[
\Hty(V_1, \ldots, V_n) = \Sigma^{-(\oplus^n_{i=1} B_i)} \Hty(V_1', \ldots, V_n').
\]
\end{definition}

Note that if $(V_1, \ldots, V_n)$ is \textit{any} sequence of formal differences with $\dim V_i = w(k_i) - h$, then the mapping cone $\Hty(V_1, \ldots, V_n)$ is equivalent (up to suspension) to the obvious mapping cone defined by setting $V_i = \C^{(w(k_i)-h)}$. This follows from the fact that the construction of $\Hty$ in Definition~\ref{def:ag.1} does not actually depend on the precise inclusion maps. Our goal will thus be to find such a sequence of such formal differences (and inclusion maps between them) which arises naturally from the topology or geometry of $W_\Gamma$. 

\subsection{The work of N\'emethi}
Our answer to Question~\ref{question:ag.2} will follow immediately from the collected work of N\'emethi \cite{NemethiOS, Nemethi} regarding path lattice cohomology and the geometric genus. We thus review these results here. Our discussion is taken from \cite[Section 6.2]{Nemethi}.

Let $(X, 0)$ be a normal surface singularity and $\smash{\rX}$ be a good resolution of the origin (see for example \cite[Section 1]{Nemethifive} for definitions of these terms). We assume that the link of $X$ is a rational homology sphere. Note that $\smash{\widetilde{X}}$ is homeomorphic to some plumbed manifold $W_\Gamma$ but is endowed with a particular complex structure. Embedded in the resolution $\smash{\widetilde{X}}$, we have a number of exceptional spheres which are in bijection with the vertices of $\Gamma$; we denote these spheres also by $v_i$. For any divisor $x = \sum_i m_i v_i$ on $\smash{\rX}$, let
\[
\O_x = \O_{\rX}/\O_{\rX}(-x).
\]
Here, $\smash{\O_{\rX}}$ is the sheaf of holomorphic functions on $\smash{\rX}$ and $\smash{\O_{\rX}}(-x)$ is the sheaf of holomorphic functions vanishing to order at least $m_i$ along each $v_i$. By the Riemann-Roch theorem, we have
\begin{equation}\label{eq:ag.1}
\chi(\rX, \O_x) = h^0(\rX, \O_x) - h^1(\rX, \O_x) = - \dfrac{1}{2} \left(x^2 + K(x)\right),
\end{equation}
where $K$ is the canonical characteristic element; see for example \cite[Appendix 2]{Nemethifive}.

A fundamental analytic invariant of $\rX$ is the \textit{geometric genus}
\[
p_g = \dim H^1(\rX, \O_{\rX}).
\]
Although $p_g$ is not a topological invariant of $W_\Gamma$, significant work due to N\'emethi has gone into bounding $p_g$ in terms of topological invariants such as lattice homology \cite{NemethiOS, Nemethi}. The most relevant idea here is the following. Let us equip the set of divisors with a natural partial order, where $x' = \sum_i m_i' v_i$ is greater than or equal to $x = \sum_i m_i v_i$ if $m_i' \geq m_i$ for each $i$. If $x' \geq x$, then we obtain a map of sheaves $\O_{x'} \rightarrow \O_x$. By the formal neighborhood theorem (see \cite[Appendix 2]{Nemethifive}), 
\[
p_g = \dim \left( \lim_{\substack{\longleftarrow \\ x > 0}} H^1(\rX, \O_x) \right).
\]
That is, $p_g$ can be computed by understanding $H^1(\rX, \O_x)$ for $x$ sufficiently large. 

This idea manifests itself in the following manner. Let $(x_i)_{i = 0}^\infty$ be a sequence of divisors with $x_0 = 0$ and
\[
x_{i+1} = x_{i} + v_{j(i)}
\] 
for some exceptional sphere $v_{j(i)}$. For $i$ sufficiently large, $h^1(\rX, \O_{x_i})$ stabilizes and gives the geometric genus. The strategy of \cite[Section 6.2]{Nemethi} is to estimate $\smash{h^1(\rX, \O_{x_{i+1}})}$ in terms of $\smash{h^1(\rX, \O_{x_i})}$, and hence obtain a cumulative estimate for $\smash{h^1(\rX, \O_{\rX})}$. This is done as follows. Consider the cohomology exact sequence
\begin{align*}
0 \rightarrow &H^0(v_{j(i)}, \mathcal{M}_i) \rightarrow H^0(\rX, \O_{x_{i+1}}) \rightarrow H^0(\rX, \O_{x_{i}}) \rightarrow \\
&H^1(v_{j(i)}, \mathcal{M}_i) \rightarrow H^1(\rX, \O_{x_{i+1}}) \rightarrow H^1(\rX, \O_{x_{i}}) \rightarrow 0;
\end{align*}
see the proof of \cite[Proposition 6.2.2]{Nemethi}. Here, $\mathcal{M}_i$ denotes the restriction of $\smash{\O_{v_{j(i)}}(-x_i)}$ to $\smash{v_{j(i)}}$. Now, since $v_{j(i)}$ is a sphere, we have that $\mathcal{M}_i$ is completely determined by its degree over $\smash{v_{j(i)}}$. As is well-known, there are two possibilities. (See for example \cite[Appendix 2]{Nemethifive}.) If $\deg(\mathcal{M}_i) > - 2$, then $\smash{H^1(v_{j(i)}, \mathcal{M}_i) = 0}$. In this case, we have a surjection on $H^0$ and an isomorphism on $H^1$:
\begin{align}
\begin{split}\label{eq:ag.2}
&H^0(\rX, \O_{x_{i+1}}) \twoheadrightarrow H^0(\rX, \O_{x_{i}})\quad \text{and}\\
&H^1(\rX, \O_{x_{i+1}}) \xrightarrow{\cong} H^1(\rX, \O_{x_{i}}).
\end{split}
\end{align}
On the other hand, if $\deg(\mathcal{M}_i) \leq - 2$, then $\smash{H^0(v_{j(i)}, \mathcal{M}_i) = 0}$. In this situation, we only have the bound
\begin{align}\label{eq:ag.3}
0 \leq h^1(\rX, \O_{x_{i+1}}) - h^1(\rX, \O_{x_{i}}) \leq h^1(v_{j(i)}, \mathcal{M}_i).
\end{align}
Now,
\[
h^1(v_{j(i)}, \mathcal{M}_i) = - \chi(v_{j(i)}, \mathcal{M}_i) = - (\chi(\rX, \O_{x_{i+1}}) - \chi(\rX, \O_{x_{i}})). 
\]
Thus the right-hand side of (\ref{eq:ag.3}) is combinatorial and may be calculated as the difference of Riemann-Roch functions. Write
\[
\Delta_i = \chi(\rX, \O_{x_{i+1}}) - \chi(\rX, \O_{x_{i}}).
\]
Note that if $\smash{h^1(\rX, \O_{x_{i+1}}) - h^1(\rX, \O_{x_{i}}) = - \Delta_i}$, so that the right-hand side of (\ref{eq:ag.3}) is sharp, we have an isomorphism on $H^0$ in addition to having a surjection on $H^1$:
\begin{align}
\begin{split}\label{eq:ag.4}
&H^0(\rX, \O_{x_{i+1}}) \xrightarrow{\cong} H^0(\rX, \O_{x_{i}})\quad \text{and}\\
&H^1(\rX, \O_{x_{i+1}}) \twoheadrightarrow H^1(\rX, \O_{x_{i}}).
\end{split}
\end{align}

Putting (\ref{eq:ag.2}) and (\ref{eq:ag.3}) together, it follows that for any path $(x_i)_{i = 0}^\infty$ as in the previous paragraph, we have the bound
\[
p_g \leq \sum_{i = 0}^{t-1} \max\{0, -\Delta_i\};
\]
for all $t$; see \cite[Example 6.2.3]{Nemethi}. As a sanity check, note that an examination of (\ref{eq:ag.1}) (together with the fact that the intersection form on $W_\Gamma$ is negative-definite) shows that $\Delta_i > 0$ for $i$ sufficiently large. Thus our upper bound does not accumulate nontrivial terms after some finite cutoff index. 

This leads to the following important definition:

\begin{definition}\label{def:ag.5}
Let $\Gamma$ be a plumbing graph of spheres. We say that $\Gamma$ is \textit{extremal} if there exists an analytic structure on $W_\Gamma$ and a path of divisors $(x_i)_{i = 0}^n$ such that:
\begin{enumerate}
\item We have $x_0 = 0$ and $x_{i+1} = x_i + v_{j(i)}$ for all $i$; and,
\item The sequence of characteristic vectors 
\[
k_i = K + 2x_i^* \in [K]
\]
carries the lattice homology of $(\Gamma, [K])$ in the sense of Theorem~\ref{thm:4.9}; and,
\item The above bound for $p_g$ is sharp; that is, we have
\[
p_g = \sum_{i = 0}^{n-1} \max\{0, -\Delta_i\}.
\]
\end{enumerate}
In this case, we call the given analytic structure on $W_\Gamma$ an \textit{extremal analytic structure} and the path $(x_i)_{i = 0}^n$ an \textit{extremal path}. Clearly, if $(x_i)_{i = 0}^n$ is an extremal path, then each instance of inequality (\ref{eq:ag.3}) is sharp. Hence for each $i$, we either have (\ref{eq:ag.2}) or (\ref{eq:ag.4}).
\end{definition}

N\'emethi and Sigurdsson \cite[Theorem 1.1.1]{NS} establish several families of graphs which are extremal; for further discussion, see \cite[Section 6.2]{Nemethi}. The most familiar example from \cite[Theorem 1.1.1]{NS} in the topological setting is the case of a Brieskorn sphere $Y$, which may be realized as the link of a weighted homogenous singularity. Here, the resolution of this singularity gives an extremal analytic structure on the usual negative-definite plumbing for $Y$.

The above discussion immediately indicates:

\begin{lemma}\label{lem:ag.6}
Let $\Gamma$ be an extremal (almost-rational) plumbing graph equipped an extremal analytic structure. Let $(x_i)_{i = 0}^n$ be an extremal path of divisors. Consider the sequence of formal differences
\[
V_i = H^1(\rX, \O_{x_i})^* - H^0(\rX, \O_{x_i})^*,
\]
where $H^i(\rX, \O_{x_i})^*$ is the dual of $H^i(\rX, \O_{x_i})$. Then for each $i$, the map of sheaves $\O_{x_{i+1}} \rightarrow \O_{x_i}$ either induces an inclusion $V_i \hookrightarrow V_{i+1}$ or an inclusion $V_{i+1} \hookrightarrow V_i$. 
\end{lemma}
\begin{proof}
As discussed above, there are two possibilities when comparing $\O_{x_{i+1}}$ and $\O_{x_i}$. In the first case, we have (\ref{eq:ag.2}). Dualizing these maps gives 
\begin{align*}
&H^0(\rX, \O_{x_{i}})^* \hookrightarrow H^0(\rX, \O_{x_{i+1}})^*\quad \text{and}\\
&H^1(\rX, \O_{x_{i}})^* \xrightarrow{\cong} H^1(\rX, \O_{x_{i+1}})^*.
\end{align*}
Inverting the latter map gives an inclusion from $V_{i+1}$ to $V_i$ in the sense of Definition~\ref{def:ag.4}. In the second case, we have (\ref{eq:ag.4}). Dualizing gives 
\begin{align*}
&H^0(\rX, \O_{x_{i}})^* \xrightarrow{\cong} H^0(\rX, \O_{x_{i+1}})^*\quad \text{and}\\
&H^1(\rX, \O_{x_{i}})^* \hookrightarrow H^1(\rX, \O_{x_{i+1}})^*.
\end{align*}
Inverting the former map gives an inclusion from $V_{i}$ to $V_{i+1}$.
\end{proof}

Hence if $(x_i)_{i = 0}^n$ is an extremal path, then the $V_i$ of Lemma~\ref{lem:ag.6} form a sequence of formal differences as in Definition~\ref{def:ag.4}. We claim that taking the contracted mapping cone over these formal differences gives, up to suspension, the lattice homotopy type $\Hty(\Gamma, [K])$:

\begin{theorem}\label{thm:ag.7}
Let $\Gamma$ be an extremal (almost-rational) plumbing graph equipped an extremal analytic structure. Let $(x_i)_{i = 0}^n$ be an extremal path of divisors. Consider the sequence of formal differences
\[
V_i = H^1(\rX, \O_{x_i})^* - H^0(\rX, \O_{x_i})^*
\]
equipped with the inclusions induced by the sheaf maps $\O_{x_{i+1}} \rightarrow \O_{x_i}$. Up to suspension, the contracted mapping cone $\Hty(V_1, \ldots, V_n)$ is $S^1$-homotopy equivalent to the lattice homotopy type $\Hty(\Gamma, [K])$. 
\end{theorem}
\begin{proof}
All that remains is to translate between various conventions. As in Definition~\ref{def:ag.5}, let
\[
k_i = K + 2x_i^*.
\]
The weight function $w(k_i)$, as given in Definition~\ref{def:4.2}, is
\[
w(k_i) = \dfrac{1}{4}(k_i^2 + |\Gamma|) = x_i^2 + K(x_i) + \dfrac{1}{4}(K^2 + |\Gamma|) = -2 \chi(\rX, \O_{x_i}) + \dfrac{1}{4}(K^2 + |\Gamma|).
\]
Hence $\dim V_i = - \chi(\rX, \O_{x_i})$ differs from $w(k_i)/2$ by a constant shift, establishing the claim.
\end{proof}

\begin{remark}\label{rem:ag.8}
As discussed in \cite[Section 6.2]{Nemethi}, it is possible to bound $\dim H^1(\rX, \mathcal{L})$ for more general $\mathcal{L}$ by taking the discussion of this section and replacing $\O_{x_i}$ with $\O_{x_i} \otimes \mathcal{L}$ throughout. The presence of an extremal path in this setting gives an interpretation of the lattice homotopy type in $\spinc$-structures other than $[K]$. We emphasize the case presented here due to a lack of examples establishing extremality for general $\mathcal{L}$.
\end{remark}

If $\Gamma$ is extremal, Theorem~\ref{thm:ag.7} provides an interpretation of the lattice homotopy type $\Hty(\Gamma, [K])$ in terms of the algebraic geometry associated to a particular analytic structure on $W_\Gamma$. We stress that although $\Hty(\Gamma, [K])$ depends only on the topological type of $W_\Gamma$, Theorem~\ref{thm:ag.7} depends on a choice of extremal analytic structure. This is in fact unsurprising. Indeed, in \cite{MOY} it is shown that if $Y$ is a Seifert fibered space, the critical points of the Chern-Simons-Dirac functional (and flows between them) may be interpreted in terms of the algebraic geometry of a particular ruled surface. Hence one should expect that any algebro-geometric interpretation of the Floer-theoretic invariants of $Y$ should be specific to especially meaningful choices of complex structure on $W_\Gamma$.

\section{Appendix}\label{sec:8}

\subsection{Lattice cohomology}\label{sec:8.1}
We now briefly discuss the change in conventions between the present paper and \cite{OSplumbed, NemethiOS, Nemethi}. The most notable difference is the use of lattice \textit{co}homology in \cite{OSplumbed, NemethiOS, Nemethi}, as opposed to our use of lattice homology. Let $\Cla_d = \Cla_d(\Gamma, [k])$ be the free $\Z[U]$-module from Definition~\ref{def:4.3} and denote
\[
\mathcal{T}_0^+ = \Z[U, U^{-1}]/U\Z[U].
\]
We then set
\[
\Cla_d^\vee = \Hom_{\Z[U]}(\Cla_d, \mathcal{T}_0^+).
\]
As before, $d$ is called the \textit{dimensional grading}. Each $\Cla_d^\vee$ also has a \textit{Maslov grading}, which is given by its grading shift as a map from $\Cla_d$ to $\mathcal{T}_0^+$. Explicitly, an element $\phi \in \Cla_d^\vee$ is homogenous of Maslov grading $\gr(\phi)$ if for each homogenous chain $x \in \Cla_d$, the image $\phi(x)$ is a homogenous element of $\mathcal{T}_0^+$ with
\begin{equation}\label{eq:8.1}
\gr(\phi) = \gr(\phi(x)) - \gr(x).
\end{equation}
Here, $\gr(x)$ is the Maslov grading of $x$ (as in Definition~\ref{def:4.3}) and $\gr(\phi(x))$ is the Maslov grading on $\mathcal{T}_0^+$ fixed by the fact that $\gr(1) = 0$ and $\deg U = -2$. The action of of $U$ on $\Cla_d^\vee$ is given by
\[
(U\phi)(x) = \phi(Ux).
\]
Note that $U$ has grading $-2$. 

The \textit{lattice cohomology} is then the homology of the cochain complex
\[
\Cla_0^\vee \xrightarrow{\partial^\vee} \Cla_1^\vee \xrightarrow{\partial^\vee} \Cla_2^\vee \xrightarrow{\partial^\vee} \cdots;
\]
see \cite[Definition 3.1.3]{Nemethi}. Following \cite{Nemethi}, we denote this by $\Hla^d(\Gamma, [k])$.\footnote{In \cite{OSplumbed}, the zeroth lattice cohomology is denoted by $\Hla^+(G)$ and the zeroth lattice homology is denoted by $\mathbb{K}^+(G)$.} As pointed out in \cite[Section 2]{OSplumbed} it is easy to check that
\begin{equation}\label{eq:8.2}
\Hla^0(\Gamma, [k]) \cong \Hom_{\Z[U]}(\Hla_0(\Gamma, [k]), \cT_0^+).
\end{equation}
This follows from the fact that $\Hla_0(\Gamma, [k])$ contains no $\Z$-torsion, which we formally record:

\begin{lemma}\label{lem:8.1}
The zeroth lattice homology $\Hla_0(\Gamma, [k])$ contains no $\Z$-torsion.
\end{lemma}
\begin{proof}
This is immediate from (for example) Theorem~\ref{thm:4.9} and an analysis of the equivalence relation $\sim$ of (\ref{eq:4.1}).
\end{proof}

Following \cite{OSplumbed}, consider the map
\[
T^+ \colon \HF^+(-Y_\Gamma, \s) \rightarrow \Hla^0(\Gamma, [k])
\]
which sends $\xi \in \HF^+(-Y_\Gamma, \s)$ to the element $\phi \in \Hla^0(\Gamma, [k])$ defined by
\[
k \mapsto \phi(k) = F^+_{W, k}(\xi) \in \HF^+(-S^3) \cong \cT_0^+.
\]
Here, we are viewing $W$ as a cobordism from $-Y_\Gamma$ to $-S^3$, so that we have the associated Floer cobordism map
\[
F^+_{W, k} \colon \HFp(-Y_\Gamma, \s) \rightarrow \HFp(S^3).
\]

It is easily checked that $\phi$ descends to a map out of $\Hla_0(\Gamma, [k])$ due to the adjunction relations; see \cite[Section 1]{OSplumbed}. In \cite{OSplumbed, NemethiOS}, it is shown that $T^+$ is an isomorphism of $\Z[U]$-modules in the case that $\Gamma$ is a one-bad-vertex or AR plumbing. Given the surgery exact sequence for Seiberg-Witten Floer homology established by the second and third authors \cite{Sasahira-Stoffregen_Triangle}, the same proof applies verbatim to show that $T^+$ is an isomorphism in the setting of Seiberg-Witten Floer homology:
\[
T^+ \colon \B(\SWF(-Y_\Gamma, \s)) \rightarrow \Hla^0(\Gamma, [k]).
\]
For readers unfamiliar with the proof, we give a brief sketch of this in Section~\ref{sec:8.3}. It is then not difficult to see that $T^+$ is an isomorphism if and only if $\T$ is an isomorphism; this is because $\T$ is essentially the dual of $T^+$. However, there is a slight subtlety: the lattice cohomology is defined by dualizing over $\Z[U]$, while usually duality in Floer homology is phrased in terms of $\Z$. We explain this minor point below.

\subsection{Duality}\label{sec:8.2}
Let $i$ be an arbitrary Maslov grading. Denote the lattice homology in grading $i$ by $\Hla_{0,i}(\Gamma, [k])$ and the lattice cohomology in grading $-i$ by $\Hla^{0,-i}(\Gamma, [k])$. We then have maps
\begin{equation}\label{eq:8.3}
\Hla_{0,i}(\Gamma, [k]) \xlongrightarrow{\T} c\widetilde{H}_i^{S^1}(\SWF(Y_\Gamma, \s))
\end{equation}
and
\begin{equation}\label{eq:8.4}
\Hla^{0,-i}(\Gamma, [k]) \xlongleftarrow{T^+} \widetilde{H}_{-i}^{S^1}(\SWF(-Y_\Gamma, \s)).
\end{equation}

We begin by considering the right-hand side of (\ref{eq:8.3}). By definition, we have
\[
c\widetilde{H}_i^{S^1}(\SWF(Y_\Gamma, \s)) = \widetilde{H}^{-i}_{S^1}(\SWF(-Y_\Gamma, \s)),
\]
where the latter is Borel \textit{co}homology. Since the Borel homology of $\SWF(-Y_\Gamma, \s)$ has no $\Z$-torsion (being isomorphic to lattice cohomology), we moreover have that
\[
\widetilde{H}^{-i}_{S^1}(\SWF(-Y_\Gamma, \s)) \cong \Hom_\Z(\widetilde{H}_{-i}^{S^1}(\SWF(-Y_\Gamma, \s)), \Z).
\]
Hence it follows that there is a perfect $\Z$-valued pairing $(a, b) \mapsto a(b)$ between the right-hand side of (\ref{eq:8.3}) and the right-hand side of (\ref{eq:8.4}). 

As discussed in the previous section, we have $\Hla^0(\Gamma, [k]) \cong \Hom_{\Z[U]}(\Hla_0(\Gamma, [k]), \cT_0^+)$. Hence there is similarly a $\Z$-valued pairing $\langle c, d \rangle \mapsto d(c)$ between the left-hand side of (\ref{eq:8.3}) and the left-hand side of (\ref{eq:8.4}), given by evaluating an element of $\Hla^{0,-i}(\Gamma, [k])$ on an element of $\Hla_{0,i}(\Gamma, [k])$. Due to the definition of the Maslov grading (\ref{eq:8.1}), this gives an element in $\cT_0^+$ of Maslov grading zero; that is, an element of $\Z$. We claim that this pairing is perfect.

For brevity, denote 
\[
H_* = \Hla_0(\Gamma, [k]), \quad H_i = \Hla_{0, i}(\Gamma, [k]), \quad H^* = \Hla^0(\Gamma, [k]), \quad \text{and} \quad H^{-i} = \Hla^{0, -i}(\Gamma, [k]),
\]
so $\langle \cdot, \cdot \rangle \colon H_i \times H^{-i} \rightarrow \Z$. Since $H_i$ and $H^{-i}$ are free, finitely-generated $\Z$-modules, to show that $\langle \cdot, \cdot \rangle$ is perfect, it suffices to prove that the evaluation map mediates an isomorphism
\begin{equation}
H^{-i} \cong \Hom_{\Z}(H_i, \Z).
\end{equation}
Note that this is not quite the same duality statement as (\ref{eq:8.2}), which says that $H^{-i}$ is isomorphic to the part of $\Hom_{\Z[U]}(H_*, \cT_0^+)$ in Maslov grading $-i$. Indeed, the elements of $\Hom_{\Z}(H_i, \Z)$ are homomorphisms which are only defined on elements of $H_i$, while the elements of $\Hom_{\Z[U]}(H_*, \cT_0^+)$ are defined on all of $H_*$. However, there is an obvious map of $\Z$-modules
\[
(\text{part of } \Hom_{\Z[U]}(H_*, \cT_0^+) \text{ in Maslov grading } -i) \rightarrow \Hom_{\Z}(H_i, \Z)
\]
given by restriction to $H_i \subset H_*$. We claim that this is an isomorphism. 

The isomorphism is most easily seen by choosing a generating set for $H_*$ as a $\Z[U]$-module, as shown in Figure~\ref{fig:8.1}. This consists of elements $\{x_1, \ldots, x_n\}$ where each $x_j$ has a particular $U$-torsion order $t_j$. Here, $t_1 = \infty$, while all the other $t_j$ are finite. This gives a $\Z$-basis for $H_i$ in the following sense: for each $x_j$ with $\smash{\gr(x_j) -2 t_j \leq i \leq \gr(x_j)}$, we obtain the basis element $\smash{U^{(\gr(x_j)-i)/2} x_j}$; otherwise $x_j$ contributes nothing. Then any $\Z$-morphism $\phi \in \Hom_{\Z}(H_i, \Z)$ extends to a $\Z[U]$-morphism $\smash{\Phi \in \Hom_{\Z[U]}(H_*, \cT_0^+)}$ of Maslov grading $-i$ by setting
\[
\Phi(x_j) = U^{-(\gr(x_j)-i)/2} \phi(U^{(\gr(x_j)-i)/2} x_j)
\]
whenever $\smash{\gr(x_j) -2 t_j \leq i \leq \gr(x_j)}$ and $\Phi(x_j) = 0$ otherwise. Note that the latter equality is forced by the condition that the Maslov grading of $\Phi$ be $-i$. Indeed, if $\gr(x_j) < i$, then we have
\[
\gr(\Phi(x_j)) = \gr(x_j) + (-i) < 0
\]
which forces $\Phi(x_j) = 0$. On the other hand, suppose $\gr(x_j) - 2 t_j > i$. Then
\[
\gr(\Phi(x_j)) = \gr(x_j) + (-i) > 2t_j.
\]
Since $U^{t_j + 1} \Phi(x_j) = \Phi(U^{t_j+1} x_j) = \Phi(0) = 0$, combining the above grading inequality with the structure of $\cT_0^+$ shows that $\Phi(x_j) = 0$. It follows that the extension $\Phi$ is unique, which easily implies $\langle \cdot, \cdot \rangle$ is a perfect pairing.

\begin{figure}[h!]
\includegraphics[scale = 1]{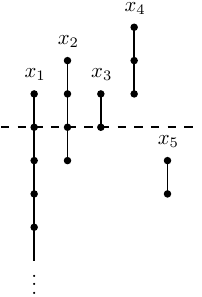}
\caption{The $\Z[U]$-module $H_*$ with a set of generators $\{x_1, \ldots, x_5\}$. The Maslov grading $i$ is represented by the dotted line, so that there are three generators for $H_i$.}\label{fig:8.1}
\end{figure}

Now let $k \in \Hla_{0, i}(\Gamma, [k])$ and $\xi \in \widetilde{H}_{-i}^{S^1}(\SWF(-Y_\Gamma, \s))$. We claim that
\begin{equation}\label{eq:8.5}
\langle k, T^+ \xi \rangle = (\T k, \xi).
\end{equation}
This is straightforward from unwinding the definitions. On the left, we have that
\[
\langle k, T^+ \xi \rangle = (T^+ \xi)(k) = F^+_{W, k}(\xi)
\]
by definition of $T^+$. On the right, we have that $(\T k, \xi) = [F_{W, k}(1)](\xi)$, where
\[
F_{W, k} \colon c\widetilde{H}_i^{S^1}(\SWF(S^3)) \rightarrow c\widetilde{H}_i^{S^1}(\SWF(Y_{\Gamma}, \s))
\]
and we interpret $F_{W, k}(1) \in c\widetilde{H}_i^{S^1}(\SWF(Y_{\Gamma}, \s))$ as an element of $\smash{\widetilde{H}^{-i}_{S^1}(\SWF(-Y_{\Gamma}, \s))}$. Since this has no $\Z$-torsion, we have that $F_{W, k}$ is dual to the map 
\[
F^+_{W, k} \colon \widetilde{H}_{-i}^{S^1}(\SWF(-Y_\Gamma, \s)) \rightarrow \widetilde{H}_{-i}^{S^1}(\SWF(-S^3))
\]
on Borel homology. Hence $[F_{W, k}(1)](\xi) = 1(F^+_{W, k}(\xi)) = F^+_{W, k}(\xi)$, as desired.

Given that $(\cdot, \cdot)$ and $\langle \cdot, \cdot \rangle$ are perfect, it is a straightforward argument in basic algebra to show that (\ref{eq:8.5}) implies $\T$ is an isomorphism if and only if $T^+$ is an isomorphism. This establishes the desired claim.

\subsection{The isomorphism theorem}\label{sec:8.3}

For readers unfamiliar with the proof of the lattice isomorphism theorem of \cite{OSplumbed, Nemethi, NemethiOS}, we give an abbreviated version of the argument in the setting of Seiberg-Witten Floer homology. The rough idea is to perform an induction on the size of $\Gamma$ and the framings of its vertices. There are two important structural results which will allow us to carry out this induction. 

The first is a naturality statement regarding the behavior of $T^+$ under blow-downs. Let $\Gamma$ be a plumbing graph with a specified vertex $v$. Define the following:
\begin{enumerate}
\item Let $\Gamma'$ be constructed from $\Gamma$ by introducing an additional vertex $x$ with framing $-1$, which is attached only to $v$. 
\item Let $\Gamma_{+1}$ be constructed from $\Gamma$ by increasing the framing of $v$ by one. 
\end{enumerate}
Note that $\Gamma_{+1}$ may be obtained from $\Gamma'$ by blowing down the leaf $x$. Thus $Y_{\Gamma_{+1}}$ and $Y_{\Gamma'}$ are homeomorphic $3$-manifolds. However, since $\Gamma_{+1}$ and $\Gamma'$ are not the same graph, it is not obvious that their corresponding lattice cohomologies are isomorphic. The following lemma asserts that $\Hla^0(\Gamma_{+1})$ and $\Hla^0(\Gamma')$ are indeed isomorphic, and that this isomorphism is natural with respect to $T^+$:

\begin{lemma}\cite[Proposition 2.5]{OSplumbed}\label{lem:8.2}
We have the following commutative square:
\[
\begin{tikzcd}
\B(\SWF(-Y_{\Gamma'})) \arrow[dd, "T^+"] \arrow[r, "\cong"]       & \B(\SWF(-Y_{\Gamma_{+1}})) \arrow[dd, "T^+"] \\
                                       &                                         &   \\
\Hla^0(\Gamma') \arrow[r, "\cong"] & \Hla^0(\Gamma_{+1}).
\end{tikzcd}
\]
\end{lemma}
\begin{proof}
See \cite[Proposition 2.5]{OSplumbed}. The map between lattice cohomologies along the bottom row is defined combinatorially and is the same as in the proof of \cite[Proposition 2.5]{OSplumbed}. The proof of commutativity relies only on the blow-up formula, which holds for both Seiberg-Witten Floer homology and Heegaard Floer homology. 
\end{proof}

The second is a commuting pair of surgery sequences, one from lattice cohomology and one from Seiberg-Witten Floer homology. To this end, let $\Gamma$ be a negative-definite plumbing graph with a specified vertex $v$. Define the following:
\begin{enumerate}
\item Let $\Gamma - v$ be the graph $\Gamma$ with $v$ deleted.
\item As before, let $\Gamma'$ be constructed from $\Gamma$ by introducing an additional vertex $x$ with framing $-1$, which is attached only to $v$.
\end{enumerate}
Note that $Y_{\Gamma - v}$, $Y_\Gamma$, and $\smash{Y_{\Gamma'}}$ form a surgery triple.

\begin{lemma}\cite[Lemmas 2.9 and 2.10]{OSplumbed}\label{lem:8.3}
There exist maps $\mathbb{A}$ and $\mathbb{B}$ such that the diagram
\[
\begin{tikzcd}
\cdots \arrow[r] & \B(\SWF(-Y_{\Gamma'})) \arrow[dd, "T^+"] \arrow[r]       & \B(\SWF(-Y_{\Gamma})) \arrow[dd, "T^+"] \arrow[r]       &\B(\SWF(-Y_{\Gamma-v})) \arrow[dd, "T^+"] \arrow[r]       & \cdots \\
                 &                                           &                                         &                                              &        \\
  0 \arrow[r] & \Hla^0(\Gamma') \arrow[r, "\mathbb{A}"] & \Hla^0(\Gamma) \arrow[r, "\mathbb{B}"] & \Hla^0(\Gamma-v) &  
\end{tikzcd}
\]
commutes up to sign and the top and bottom rows are exact.\footnote{Note that the bottom row does \textit{not} continue to the right.}
\end{lemma}
\begin{proof}
See \cite[Lemmas 2.9 and 2.10]{OSplumbed}. The maps $\mathbb{A}$ and $\mathbb{B}$ are defined combinatorially and are the same as in the proof of \cite[Lemmas 2.9 and 2.10]{OSplumbed}, where it is shown that the bottom row is exact; see also \cite{Greenesurgery, Nemethisequence}. The proof of commutativity relies only on the formal behavior of cobordism maps under composition. For a discussion of signs in the exact sequence, see the discussion at the end of \cite[Section 2]{OSplumbed}.
\end{proof}

We now sketch the inductive argument itself. The base case corresponds to a plumbing graph $\Gamma$ with only a single vertex. In this situation, $-Y_\Gamma$ is an L-space; an easy computation shows that $\Hla^0(\Gamma)$ is likewise a copy of $\cT_0^+$ in each $\spinc$-structure. Moreover, a calculation of the grading shift associated to $W_\Gamma$ combined with the fact that $W_\Gamma$ is negative-definite shows that $T^+$ is an isomorphism.

The proof proceeds in several stages. First consider the category of zero-bad-vertex plumbings. We will need the fact that if $\Gamma$ is any zero-bad-vertex plumbing, then $Y_\Gamma$ is an L-space; see \cite[Lemma 2.6]{OSplumbed}. In particular, $\HFp_\text{odd}(-Y_\Gamma) = 0$.\footnote{Here, we mean that in each $\spinc$-structure, the Floer homology vanishes in gradings which are congruent to $(d + 1) \bmod 2$.} Invoking the isomorphism between Heegaard Floer, monopole Floer, and Seiberg-Witten Floer homology \cite{KLT, CGH, LidmanMan} shows that this holds for $\smash{\B(\SWF(-Y_\Gamma))}$ also. 

We now proceed by induction on the size of $\Gamma$. The base case of a single vertex has already been addressed; for the inductive step, let $\Gamma$ be a zero-bad-vertex plumbing. Then:
\begin{enumerate}
\item If $\Gamma$ has a leaf with framing $-1$, blow down to obtain the zero-bad-vertex plumbing $\Gamma_{+1}$. Lemma~\ref{lem:8.2} combined with the inductive hypothesis establishes that $T^+$ is an isomorphism for $\Gamma$.
\item Otherwise, select any leaf $v$ of $\Gamma$. We proceed by a downwards sub-induction on the framing of $v$. The base case where the framing of $v$ is $-1$ is covered above. Otherwise, $\Gamma - v$ and $\Gamma_{+1}$ are also zero-bad-vertex plumbings. (These may be assumed to be negative-definite; see for example \cite[Lemma 6.1]{ABDS}.) Lemma~\ref{lem:8.3} gives an exact sequence
\[
\begin{tikzcd}
0 \arrow[r] & \B(\SWF(-Y_{\Gamma'})) \arrow[dd, "T^+"] \arrow[r]       & \B(\SWF(-Y_{\Gamma})) \arrow[dd, "T^+"] \arrow[r]       &\B(\SWF(-Y_{\Gamma-v})) \arrow[dd, "T^+"] \arrow[r]       & 0 \\
                 &                                           &                                         &                                              &        \\
  0 \arrow[r] & \Hla^0(\Gamma') \arrow[r, "\mathbb{A}"] & \Hla^0(\Gamma) \arrow[r, "\mathbb{B}"] & \Hla^0(\Gamma-v) &  
\end{tikzcd}
\]
where in the top-left and top-right, we have used the fact that $\smash{\B(\SWF(-Y_{\Gamma-v}))}$ and $\smash{\B(\SWF(-Y_{\Gamma'}) = \B(\SWF(-Y_{\Gamma_{+1}}))}$ are zero in odd gradings. By inductive hypothesis, $T^+$ is an isomorphism for $\Gamma - v$. The inductive hypothesis implies that $T^+$ is an isomorphism for $\Gamma_{+1}$; Lemma~\ref{lem:8.2} then shows that $T^+$ is an isomorphism for $\Gamma'$. Applying the five-lemma to the above diagram shows that $T^+$ is an isomorphism for $\Gamma$.
\end{enumerate}
This completes the proof in the case that $\Gamma$ has no bad vertices.

Now suppose that $\Gamma$ has a single bad vertex $v$. We proceed by upwards induction on the framing of $v$. If the framing of $v$ is sufficiently negative, then $\Gamma$ has no bad vertices and hence $T^+$ is an isomorphism by the previous paragraph. Otherwise, we have an exact sequence
\[
\begin{tikzcd}
0 \arrow[r] & \B(\SWF(-Y_{\Gamma'})) \arrow[dd, "T^+"] \arrow[r]       & \B(\SWF(-Y_{\Gamma})) \arrow[dd, "T^+"] \arrow[r]       &\B(\SWF(-Y_{\Gamma-v})) \arrow[dd, "T^+"] \arrow[r]       & \cdots \\
                 &                                           &                                         &                                              &        \\
  0 \arrow[r] & \Hla^0(\Gamma') \arrow[r, "\mathbb{A}"] & \Hla^0(\Gamma) \arrow[r, "\mathbb{B}"] & \Hla^0(\Gamma-v) &  
\end{tikzcd}
\]
Here, in the top-left corner, we have used the fact that $\Gamma - v$ has no bad vertices. (Note that, unlike the previous case, we do \textit{not} know that the top-right corner is zero, since $-Y_{\Gamma'} \cong - Y_{\Gamma_{+1}}$ is not necessarily an L-space.) By the previous paragraph, we also know that $T^+$ is an isomorphism for $\Gamma - v$. The inductive hypothesis is that $T^+$ is an isomorphism for $\Gamma$; by the five-lemma, $T^+$ is thus an isomorphism for $\Gamma'$. By Lemma~\ref{lem:8.2}, this shows that $T^+$ is an isomorphism for $\Gamma_{+1}$ and completes the inductive step. This gives the proof for one-bad-vertex manifolds.

The case when $\Gamma$ is an AR plumbing is similar, as discussed in \cite[Section 8]{NemethiOS}. The essential point is to show that a plumbing is rational if and only if it is a lattice cohomology L-space; see \cite[Theorem 6.3]{NemethiOS}. Then a similar argument as in the first part of the above proof shows that $T^+$ is an isomorphism for all rational plumbings. Recall that an AR plumbing has a vertex $v$ such that decreasing the framing of $v$ makes $\Gamma$ into a rational plumbing (see for example \cite[Section 6.2]{NemethiOS} for a definition). The same argument as for one-bad-vertex plumbings then applies to show that $T^+$ is an isomorphism for AR plumbings. See \cite[Theorem 8.3]{NemethiOS}.

\bibliographystyle{amsalpha}
\bibliography{bib}

\end{document}